\documentclass[12pt,letterpaper, twoside, openany]{article}
\usepackage[top=1in, left =1in, right=1in, bottom=1in]{geometry}

\usepackage{lineno}
\usepackage{lipsum}
\usepackage{amsmath, amssymb}
\usepackage{bm}
\usepackage{natbib, amsfonts, xr, enumerate, graphicx}
\usepackage{amsthm}
\usepackage{xcolor}
\usepackage{csquotes}
\usepackage{float}
\usepackage{rotating}
\usepackage{url}
\usepackage{titlesec}
\usepackage{bbm}
\usepackage{stmaryrd}
\usepackage{authblk}
\usepackage{float}

\usepackage{subcaption}

\DeclareMathOperator*{\Var}{Var}

\DeclareMathOperator*{\Cov}{Cov}

\theoremstyle{plain}

\newtheorem{theorem}{Theorem}[section]
\newtheorem{lemma}[theorem]{Lemma}

\newtheorem{corollary}[theorem]{Corollary}
\newtheorem{remark}[theorem]{Remark}

\newtheorem*{bootstrap}{Bootstrap Procedure}

\usepackage{xr}
\makeatletter

\newcommand*{\addFileDependency}[1]{
\typeout{(#1)}
\@addtofilelist{#1}
\IfFileExists{#1}{}{\typeout{No file #1.}}
}\makeatother

\author[1]{Kwun Chuen Gary Chan}
\author[2]{Hok Kan Ling}
\author[3]{Chuan-Fa Tang}
\author[4]{Sheung Chi Phillip Yam}
\affil[1]{Department of Biostatistics, University of Washington}
\affil[2]{Department of Mathematics and Statistics,
Queen's University}
\affil[3]{Mathematical Sciences, The University of Texas at Dallas}
\affil[4]{Department of Statistics and Data Science, The Chinese University of Hong Kong}

\begin{document}
	\baselineskip=28pt 
	
	\newgeometry{top=1in,  left =1in, right=1in, bottom=1in}
	\title{\Large Likelihood-based Spacings Goodness-of-Fit Statistics for Univariate Shape-constrained Densities}

	\date{\vspace{-5ex}}
	
	\maketitle

\begin{abstract}
A variety of statistics based on sample spacings have been studied for testing goodness-of-fit to parametric distributions.  To test the goodness-of-fit to a nonparametric class of univariate shape-constrained densities, including widely studied classes such as $k$-monotone and log-concave densities, a likelihood ratio test with a working alternative density estimate based on the spacings of the observations is considered, and is shown to be asymptotically normal and distribution-free under the null, consistent under fixed and certain local alternatives, and admits bootstrap calibration.  The distribution-freeness under the null comes from the fact that the asymptotic dominant term depends only on a function of the spacings of transformed outcomes that are uniformly distributed.  Applications and extensions of theoretical results in the literature of shape-constrained estimation are required to show that the average log-density ratio converges to zero at a faster rate than the sample spacing term under the null, and diverges under the alternatives.  Numerical studies are conducted to demonstrate that the test is applicable to various classes of shape-constrained densities and has a good balance between type-I error control under the null and power under alternative distributions.
\end{abstract}



\section{Introduction}\label{sect:intro}
Scientific knowledge and theoretical understanding of a problem can often provide hypotheses about the underlying data-generating mechanism, which could take the form of a completely known parametric distribution, a parametric family of distributions with unknown finite-dimensional parameters, or flexible nonparametric constraints on the shape of the density, \emph{e.g.}, monotonicity or log-concavity, which  have motivated many recent theoretical and methodological developments.

Goodness-of-fit tests of parametric distributions using sums of functions of sample spacings, which are the lengths of the intervals between adjacent order statistics from a sample, have been studied since the 1940s.  Early developments include \cite{greenwood1946statistical, kimball1947some, moran1947random}, and \cite{darling1953class}.  Since one can use the distribution function to transform data with a known distribution to the uniform distribution on the unit interval, it suffices to study the theoretical properties of testing the goodness-of-fit to the uniform distribution. Asymptotic power of tests based on sample spacing was studied in \cite{chibisov1961tests} and \cite{rao1975weak}, who showed a theoretical limitation but motivated subsequent developments of goodness-of-fit tests based on higher-order spacings, or $\nu$-spacings ($\nu>1$), which are the lengths of the intervals between two order statistics whose order differs by $\nu$ \citep{cressie1979optimal, del1979asymptotic}, and have a power advantage over sample spacings $(\nu=1)$.  \cite{hall1986powerful} and \cite{jammalamadaka1989asymptotic} studied the power of $\nu$-spacings goodness-of-fit statistics when the order $\nu$ diverges with the sample size $n$.  

To test whether the data-generating mechanism fits a parametric model with unknown parameters, which is typically a composite null hypothesis, \cite{cheng1989goodness, wells1992tests, wells1993large} studied functions of sample spacings using an estimated distribution function transformed with a plug-in estimator of the unknown parameters, whose null distribution was shown to be asymptotically equivalent to the case where the parameters are indeed known.  However, it is unclear whether similar procedures can be used for testing the goodness-of-fit to a nonparametric class with an estimated density.

Nonparametric density estimation under shape constraints has received increasing attention in recent years;  see, e.g., \cite{groeneboom2014nonparametric} for various examples. A particularly valuable feature of  estimation under shape constraints is that nonparametric maximum likelihood estimators are often fully automatic, requiring no tuning parameters, in comparison to kernel smoothing methods. 

For univariate distributions, two general classes, $k$-monotonicity and log-concavity, have recently been studied in detail. A density function $f$ on $\mathbb{R}^+$ is $1$-monotone if it is nonincreasing.  \cite{grenander1956theory} first showed that the nonparametric maximum likelihood estimator (NPMLE) under the nonincreasing (1-monotone) density assumption can be characterized as the left derivative of the least concave majorant of the empirical distribution function. As a result, the NPMLE in this case is often referred to as the Grenander estimator. Its asymptotic distribution at a fixed interior point where the derivative is strictly negative was obtained by \cite{rao1969estimation} and \cite{Groeneboom1985}. {Estimating a monotone density in a Bayesian setting has been considered in \cite{salomond2014concentration}, \cite{chakraborty2022rates}, and \cite{jongbloed2021bayesian}.}

A density function is $2$-monotone if it is nonincreasing and convex. The corresponding NPMLE was studied in \cite{groeneboom2001estimation}. Furthermore, a density function is $k$-monotone for $k \geq 3$ if  $(-1)^j f^{(j)}$ is nonnegative, nonincreasing, and convex for $j=0,\ldots,k-2$, with the NPMLE being studied in \cite{balabdaoui2007estimation} and \cite{balabdaoui2010estimation}.  A density on $\mathbb{R}^+$ is called completely monotone if $(-1)^j f^{(j)}(t) \geq 0$ for all nonnegative integers $j$ and all $t > 0$, which can be seen as the limit of a $k$-monotone density as $k$ goes to infinity and can be represented as a scale mixture of exponential distributions. \cite{Jewell1982mixtures}  established the unique existence of the NPMLE for a completely monotone density and the consistency of the mixing distribution function. 
To the best of our knowledge, there is no pointwise limit distribution theory or global rate of convergence for continuous completely monotone densities. A Bayesian approach to inference for $k$-monotone densities has been explored in \cite{wang2025bayesian}.

A density function on $\mathbb{R}$ is log-concave if its logarithm is concave. This class is a subset of  unimodal densities, contains many of the commonly used parametric distributions, and can be regarded as an infinite-dimensional generalization of the class of Gaussian densities.  Estimation and theoretical properties have been extensively studied in recent years; see, for example, \cite{dumbgen2009maximum, balabdaoui2009limit, walther2009inference, cule2010theoretical, saumard2014log, samworth2018recent, mariucci2020bayesian, cui2024martingale}. 

Apart from estimation, various hypothesis testing problems related to shape restrictions on monotone functions have been considered in the literature. Many of these tests make use of distances between two estimators, where one is valid only under the null hypothesis and the other is valid in general.  For example, \cite{proschan1967tests} tests the null hypothesis of a constant hazard rate against the alternative of an increasing hazard rate; \cite{hall2005testing} tests the null hypothesis that a hazard rate is monotone nondecreasing;   \cite{groeneboom2012isotonic} tests for local monotonicity of a hazard function;
\cite{durot2010goodness} tests a parametric null hypothesis that respects a monotonicity constraint; and \cite{kulikov2004testing} tests for monotone density.
These tests typically have complicated null distributions, and additional regularity conditions on the true underlying distribution are often imposed. Bayesian tests for monotone functions or densities have also been investigated \citep{salomond2018testing, chakraborty2021convergence, chakraborty2022rates}.
For testing whether a multivariate density is log-concave, see \cite{chen2013smoothed} and \cite{dunn2021universal}. Hypothesis testing for a log-concave density versus a mixture of log-concave densities has been studied by \cite{walther2002detecting} and \cite{balabdaoui2018inference}.

In this article, we study a unified nonparametric likelihood ratio test (NPLRT) for testing whether the underlying univariate density belongs to a particular hypothesis class of functions, focusing on $k$-monotone, completely monotone, and log-concave densities. An attractive feature of our test is that the limiting null distribution remains the same across different hypothesis classes and does not depend on the unknown density or its derivatives.   In a conventional likelihood ratio test, we have to maximize the likelihood over the union of the null and alternative hypotheses. In such cases, the union is the collection of all densities on $\mathbb{R}$, where it is well known that the corresponding maximum likelihood is infinite, resulting in an ill-posed likelihood ratio statistic. 
Instead, we use a histogram-type estimator that depends on the $\nu$-spacings. It is worth noting that different types of likelihood ratio-based inference, distinct from the one considered in this work, have been studied in other shape-constrained inference problems, including testing for equality at fixed points, forming confidence intervals \citep{banerjee2001likelihood, banerjee2005likelihood, banerjee2007likelihood, nane2015likelihood, groeneboom2015nonparametric, doss2019inference, doss2019concave}, and conducting two-sample tests \citep{groeneboom2012likelihood}. Additionally, studies on local shape-constrained inference not based on likelihood ratios have been conducted \citep{deng2021confidence, deng2023inference}.

Our proposed likelihood ratio test statistic depends only on the NPMLE under the corresponding shape constraint and the spacings of the observations, and can be computed easily. Moreover, the asymptotic null distribution of the test statistic is distribution-free, which follows from the convergence of a sum of functions of uniform $\nu$-spacings.  
An arithmetic mean of the individual log-likelihood ratios, each evaluated at the observation using the NPMLE and the true density, is required to converge at a faster rate than the spacings term under the null and to diverge under fixed and certain local alternatives.  We verify these properties for the $k$-monotone, completely monotone, and log-concave hypothesis classes, by applying and extending existing results in the shape-constrained estimation literature.

We also consider bootstrapping for both theoretical and practical reasons.  Although it is known that the nonparametric bootstrap and bootstrapping from the NPMLE for 1-monotone densities do not work for finding the pointwise limiting distribution of the NPMLE at an interior point \citep{kosorok2008bootstrapping, sen2010inconsistency}, we show that bootstrapping from the NPMLE is valid for approximating the distribution of the NPLRT.  The main reason bootstrapping from the NPMLE works in this context is that the statistic is a global measure rather than a local one, and the bootstrap distribution only needs to be continuous, without further requirements on differentiability.
However, the nonparametric bootstrap remains inapplicable because of the lack of continuity of the bootstrap distribution. 
For practical reasons, given the relatively slow rate of convergence of the log-density ratio and the remainder term, bootstrap calibration of the critical value can provide better control of the type-I error and improve the power under alternatives.

The following sections are organized as follows. In Section \ref{sect:main_result}, we introduce our NPLRT and establish the asymptotic distribution of the NPLRT statistic under the null hypothesis, as well as the consistency of the test under the alternative hypothesis, given general conditions. In Section \ref{sect:shape}, we specifically discuss the average log-density ratio for $k$-monotone (including $1$-monotone and completely monotone) and log-concave densities under the null and alternative hypotheses, where we also  establish a rate of convergence for the maximum likelihood estimator for $k$-monotone densities under relaxed conditions. 
Section \ref{sect:bootstrap} discusses a valid bootstrap procedure. In Section \ref{sect:simulation}, simulation studies are performed to evaluate the performance of the test in various situations. 
Conclusion and additional remarks are provided in Section \ref{sect:discussion}.
All proofs of the theoretical results, along with some real data illustrations of our tests, are presented in the appendix.

\section{Nonparametric goodness-of-fit test and main results}\label{sect:main_result}

\subsection{Definition and main results}\label{sect:prelim}
Let $X_1,\ldots,X_n$ be a random sample from a univariate distribution $F_0$ with density $f_0$. The likelihood function is 
\begin{equation*}
	L_n(f) := \prod^n_{i=1} f(X_i).
\end{equation*}
Our aim is to propose a nonparametric likelihood ratio test (NPLRT) for testing the null hypothesis $H_0: f_0 \in \mathcal{F}$, where $\mathcal{F}$ is a nonparametric class of densities for which the NPMLE, $\hat{f}_n := \text{argmax}_{f \in \mathcal{F}} L_n(f)$, exists, versus $H_1: f_0 \notin \mathcal{F}$. In particular, we focus on the hypothesis classes of shape-constrained densities, including the classes of (i) decreasing densities, (ii) $k$-monotone densities ($k \geq 2$), (iii) completely monotone densities, and (iv) log-concave densities, for which all the NPMLEs exist and are unique.  

To define a likelihood ratio test, we need an estimator that works under both the null and alternative hypotheses. To this end, we consider a histogram-type estimator, which leads to the desirable properties that the test statistic is asymptotically distribution-free under the null hypothesis. Additionally, it does not require strong assumptions, such as bounded support or upper-boundedness of the underlying density, to establish the asymptotic null distribution and the consistency of the test. Let $Z_1 <\ldots < Z_n$ be the order statistics of $X_1,\ldots,X_n$, and let $\nu \in \mathbb{N}$. 

Without loss of generality, and for simplicity of presentation, we assume throughout the article that $\frac{n-1}{\nu}$ is an integer, and we denote {it} by $n_{\nu}$. We define a piecewise constant density function with $n_{\nu}$ steps as follows:
for $j=0,\ldots, n_{\nu} - 1$, $x \in ( Z_{j\nu+1}, Z_{(j+1)\nu+1}]$, let
\begin{align*}
	f^H_{n,\nu}(x) := \frac{1}{n_{\nu}(Z_{(j+1)\nu+1} - Z_{j\nu+1})}.
\end{align*}
Define $f^H_{n,\nu}(Z_1) := f^H_{n,\nu}(Z_1+)$ and $f^H_{n,\nu}(x) = 0$ if $x \notin [Z_1, Z_n]$. Thus, $f^H_{n,\nu}$ is a function of $\nu$-spacings.  Modifications can be made accordingly when $n_{\nu}$ is not an integer by dividing the data into $\lfloor n_{\nu} \rfloor$ groups as evenly as possible, where for $x \in \mathbb{R}_+$, $\lfloor x \rfloor$ denotes the greatest integer less than or equal to $x$.

Our proposed likelihood ratio test statistic is
\begin{equation*}
T_n :=  -\frac{1}{n}\log \frac{\prod^n_{i=1} \hat{f}_n(X_i)}{\prod^n_{i=1}f^H_{n,\nu}(X_i)}.
\end{equation*}
To study its asymptotic properties, we can write
\begin{equation} \label{eq:decomp1}
T_n = - \frac{1}{n}\sum^n_{i=1} \log \frac{\hat{f}_n(X_i)}{f_0(X_i)} - \frac{1}{n}\sum^n_{i=1} \log \frac{f_0(X_i)}{f^H_{n,\nu}(X_i)}.
\end{equation}
Denote
\[
S_n := - \frac{1}{n}\sum^n_{i=1} \log \frac{\hat{f}_n(X_i)}{f_0(X_i)}
\]
which is the first term in the right-hand-side of (\ref{eq:decomp1}).  We further decompose the second term into two terms,
\begin{equation}\label{eq:decomposition_hist}
-\frac{1}{n}\sum^n_{i=1} \log \frac{f_0(X_i)}{f^H_{n,\nu}(X_i)} =M_n+R_n  ,    
\end{equation}
where
\begin{align*}
M_n &:= -  \frac{\nu}{n} \sum_{j=0}^{n_{\nu} - 1} 
\log \frac{F_0(Z_{(j+1)\nu+1}) - F_0(Z_{j\nu+1})}{\nu/(n-1)},\\
R_n&:=	- \frac{1}{n} \sum_{j=0}^{n_{\nu} - 1} 
\sum^\nu_{l=1} \log \frac{f_0(Z_{j\nu+l+1})(Z_{(j+1)\nu+1} - Z_{j\nu+1})}{F_0(Z_{(j+1)\nu+1}) - F_0(Z_{j\nu+1})} -\frac{1}{n} \log \frac{f_0(Z_1)}{f^H_{n,\nu}(Z_1)}.
\end{align*}

We first provide a high-level description of how the decomposition $T_n=S_n+M_n+R_n$ will be used in the theoretical development.  The term $S_n$ is the arithmetic mean of the individual log-likelihood ratios, each evaluated at the observation using the NPMLE and the true density, $M_n$ is an average log-spacings of transformed data, and $R_n$ is an asymptotically negligible remainder term under both the null and the alternative hypotheses. For simplicity, we refer to $S_n$ as the average log-density ratio throughout the article.
The NPMLE under a specific hypothesis class appears only in the first term, $S_n$, but not in the other two terms. The key observation is that the terms, $S_n$ and $M_n$, contribute differently to the behavior of the test statistic under the null and the alternative hypotheses.  Under the null, $M_n$ asymptotically dominates $S_n$, leading to an asymptotically normal and pivotal null distribution.  Under fixed and certain local alternatives, $S_n$ diverges, ensuring the test is consistent. 

However, the NPMLE for each hypothesis class, and therefore $S_n$, differs substantially. Thus, the above argument needs to be established on a case-by-case basis, as discussed in Section \ref{sect:shape}.  The convergence of $M_n$ is discussed in Section \ref{sect:spacings}.  Sufficient conditions for $R_n$ being asymptotically negligible are provided in Section \ref{sect:suff_cond}. 

Let $\Gamma(\cdot), \psi(\cdot)$ and $\psi_1(\cdot)$ denote the gamma, digamma, and trigamma functions, respectively. Specifically, $\Gamma(y) = \int^\infty_0 t^{y-1}e^{-t}\,dt$ for $y > 0$,  $\psi(y) = \frac{d}{dy}\log \Gamma(y)$, and $\psi_1(y) = \frac{d}{dy} \psi(y)$.  The following result provides the asymptotic null distribution of our goodness-of-fit test statistics.

\begin{theorem}[Asymptotic null distribution]\label{coro:main_result}
Suppose that $\nu = O(n^{1/3}(\log n)^{-1})$, $\sqrt{n \nu} S_n = o_p(1)$ and $R_n = O_p(\frac{\nu \log n}{n})$. Then,
\begin{equation}\label{eq:asy_dist_under_H0}
	\sqrt{ \frac{n\nu}{\nu^2\psi_1(\nu)-\nu}} \left( T_n -  \log \nu +  \psi(\nu) \right) \stackrel{d}{\rightarrow} N(0, 1).
\end{equation}
\end{theorem}

The following theorem establishes the consistency of our test under alternatives that may converge to the null, where the key condition is the existence of some sequence $L_n \uparrow \infty$ such that $\sqrt{n \nu}S_n > L_n$ with probability approaching $1$. In Theorems C.6 and E.3 in the appendix, we show that this will be satisfied when the minimum Hellinger distance between the true underlying density and any densities in the hypothesis class does not go to $0$ too quickly, where the Hellinger distance between two densities $f$ and $g$ is defined as
	\begin{align*}
		h(f,g) = \frac{1}{\sqrt{2}}\left[ \int_0^\infty \left\{ \sqrt{f(t)} - \sqrt{g(t)}\right\}^2\,dt\right]^{1/2}. 
	\end{align*} 

\begin{theorem}[Consistency]\label{corollary:consistency}
Suppose that $\nu = O(n^{1/3}(\log n)^{-1})$. Under $H_1$, suppose that there exists some sequence $L_n \uparrow \infty$ such that $\lim_{n \rightarrow \infty} \mathbb{P}(\sqrt{n \nu }S_n > L_n) = 1$  and $R_n = O_p(\frac{\nu \log n}{n})$. Then, the NPLRT is consistent. That is, for any $0 \leq c < \infty$,
\begin{equation*}
\lim_{n \rightarrow \infty} \mathbb{P}     \left(
\sqrt{ \frac{n\nu}{\nu^2\psi_1(\nu)-\nu}} \left( T_n -  \log \nu +  \psi(\nu) \right) > c\right) = 1.
\end{equation*}
\end{theorem}

\begin{remark}\label{remark:maximize_power}
Under the conditions of Theorem \ref{corollary:consistency}, the consistency of the test relies on the decomposition $T_n = S_n + M_n + R_n$, where $\sqrt{n \nu}\, S_n \to \infty$ with probability one, $\sqrt{n\nu / (\nu^2 \psi_1(\nu) - \nu)}\, (M_n - \log \nu + \psi(\nu)) \stackrel{d}{\rightarrow} N(0,1)$ (Theorem \ref{thm:main_hist_asy_dist}), and $\sqrt{n\nu / (\nu^2 \psi_1(\nu) - \nu)}\, R_n = o_p(1)$. Since $S_n$ is independent of $\nu$, increasing $\nu$ causes $\sqrt{n \nu}\, S_n$ to diverge faster under $H_1$. This suggests choosing $\nu \asymp n^{1/3} / \log n$ to maximize the power of the test.
\end{remark}

Note that Theorems \ref{coro:main_result} and \ref{corollary:consistency} hold for both fixed $\nu$ and when $\nu$ diverges to infinity.  When $\nu=1$, $\psi(1)=-\gamma$, where $\gamma=0.57721...$ is the Euler--Mascheroni constant, and $\psi_1(1)=\pi^2/6$.  As $\nu\to \infty$, 
\begin{equation*}
\psi_1(\nu) = \frac{1}{\nu} + \frac{1}{2\nu^2} + O(\nu^{-3}) \ ,
\end{equation*}
and therefore $\nu^2 \psi_1(\nu) - \nu \rightarrow 1/2$ as $\nu \rightarrow \infty$. Thus, in Theorems \ref{coro:main_result} and \ref{corollary:consistency}, $\nu^2 \psi_1(\nu) - \nu$ may be replaced  by $1/2$ for diverging $\nu$.

\subsection{Asymptotic distribution of $M_n$}\label{sect:spacings}
Since $F_0$ is assumed to be continuous, $(F_0(Z_1),\ldots,$ $F_0(Z_n))$ and $(U_{(1)},\ldots,U_{(n)})$ have the same distribution, where $0<U_{(1)}<\cdots<U_{(n)}<1$ are the order statistics from a random sample of size $n$ from a Uniform(0, 1) distribution. Therefore, $M_n$, as a function of $(F_0(Z_1), \ldots, F_0(Z_n))$, is distribution-free for any finite sample size $n$. 

It is well known that 
\begin{align*}
		&(U_{(1)}, U_{(2)} - U_{(1)},\ldots, U_{(n)} - U_{(n-1)} ) 
  \stackrel{d}{=}\bigg( \frac{E_1}{\sum^{n+1}_{h=1} E_h}, \ldots, \frac{E_n}{\sum^{n+1}_{h=1} E_h}\bigg), 
	\end{align*}
	where $E_h, h=1,\ldots,n+1$, are independent random variables, each following the standard exponential distribution with mean $1$; see, for example, Theorem 2.2 in \cite{Devroye1986}.  Therefore, 
 \begin{align*}
  & (U_{(\nu+1)} - U_{(1)}, U_{(2\nu+1)} - U_{(\nu+1)}, \ldots, U_{(n)} - U_{(n-\nu)})
    \stackrel{d}{=} \left(\frac{ \sum^{\nu+1}_{l=2}E_l}{\sum^{n+1}_{h=1}E_h}, \frac{\sum^{2\nu+1}_{l=\nu+2}E_l}{\sum^{n+1}_{h=1}E_h },\ldots, 
    \frac{\sum^{n}_{l=n-\nu+1}E_l}{\sum^{n+1}_{h=1}E_h } \right).
\end{align*}
The numerators are sums of $\nu$ independent standard exponential random variables, each of which follows a gamma distribution with shape parameter $\nu$ and scale parameter $1$.  Let $\tilde{E}_j$, $j=0,\ldots,n_{\nu}-1$ be independent and identically distributed gamma random variables with shape parameter $\nu$ and scale parameter $1$.  We show in the appendix that if $\nu=o(n)$, $M_n-\log \nu+\psi(\nu)$ is asymptotically equivalent to 
\[
\frac{1}{n}\sum_{j=0}^{n_{\nu} - 1} (\tilde{E}_j-
\nu\log(\tilde{E}_j)-\nu+\nu\psi(\nu)) \ .
\]
An interesting fact is that $\Var(\tilde{E}_1)=O(\nu)$ and $\Var(\nu \log(\tilde{E}_1))=O(\nu)$, but $\Var(\tilde{E}_1-\nu \log(\tilde{E}_1))=O(1)$.  We now state the following result.

\begin{theorem}\label{thm:main_hist_asy_dist}
If $\nu = o(n)$, then
\begin{equation*}
	\sqrt{ \frac{n\nu}{\nu^2\psi_1(\nu)-\nu}} \left( M_n -  \log \nu +  \psi(\nu) \right) \stackrel{d}{\rightarrow} N(0, 1).
\end{equation*}
\end{theorem}

Theorem~\ref{thm:main_hist_asy_dist} holds for both fixed $\nu$ and diverging $\nu$.  For fixed $\nu$, the result (in a slightly different form) was given in \cite{del1976spacings}. For diverging $\nu$, the result is given in \cite{jammalamadaka1989asymptotic}, with $\nu^2 \psi_1(\nu) - \nu$ replaced by its limit $1/2$, but that result does not hold for fixed $\nu$.

\subsection{Sufficient conditions for $R_n$ to be asymptotically negligible}\label{sect:suff_cond}
In this subsection, we provide some sufficient conditions for $R_n = O_p( \frac{\nu \log n}{n})$ in Theorem \ref{theorem:Rn_order}.  Here, $f_0$ may belong to either $H_0$ or $H_1$.  These conditions relate only to the true underlying density and not to the NPMLE.  Intuitively, it is expected that  $\log \{f_0(Z_{j\nu+l+1})(Z_{(j+1)\nu+1} - Z_{j\nu+1})\} \approx \log \{F_0(Z_{(j+1)\nu+1}) - F_0(Z_{j\nu+1})\}$,  so that $R_n$ approaches $0$ at a certain rate.  Given a density $f$, let $\tau_f$ and $\sigma_f$ denote the left endpoint and right endpoint of its support, respectively.
 
\noindent
\textbf{Conditions}:
\begin{enumerate}
    \item[(A)] There exists $x_0 > 0$ such that for all $|x| \geq x_0$, $f_0(x) \leq |x|^{-\gamma}$ for some $\gamma > 1$.
    
    \item[(B)] 
    For $a$ such that $f(a+) = \infty$ or $f(a-) = \infty$, there exists $\delta > 0$ such that for $x$ with $|x-a| < \delta$, $f(x) \leq (x-a)^{\gamma_2 - 1}$ for some $\gamma_2 \in (0, 1)$.
    
    \item[(C)] 
    \begin{enumerate}
	\item[(i)] $f_0$ is monotone (decreasing or increasing); or
	\item[(ii)] There exist $K_1 > \tau_{f_0} $ and $K_2 < \sigma_{f_0}$, such that $f_0$ is monotone on $(\tau_{f_0}, K_1]$ and $[K_2, \sigma_{f_0})$, and $\log f_0$ is Lipschitz continuous on $[K_1, K_2]$ with a Lipschitz constant $L$.
\end{enumerate}
\end{enumerate}

Condition (A) essentially requires the tail of $f_0$ to decay at a rate that is not too slow if its support is unbounded. In particular, a regularly varying random variable with any tail index $\alpha > 0$\footnote{A random variable is regularly varying with tail index $\alpha >0$ if its survival function $\mathbb{P}(X > x) = x^{-\alpha}l(x)$ for some slowly varying function $l$.} whose density is ultimately monotone\footnote{A function $f$ is said to be ultimately monotone if $f$ is monotone on $(x, \infty)$ for some $x > 0$. Thus, $f_0$ satisfying (C)(i) or (C)(ii) is ultimately monotone.} satisfies Condition (A) by the monotone density theorem (see Theorem 1.2.9 in \cite{mikosch1999regular}). When $\alpha \in (0, 1)$, such a regularly varying random variable has an infinite mean. 

Condition (B) essentially requires that $f_0$ does not grow to infinity too quickly if it is unbounded. We also allow for the possibility that the density grows to infinity at multiple points. 

\begin{theorem}\label{theorem:Rn_order}
Suppose that Conditions (A), (B), and one of Conditions (C) hold. 
Then, we have
\begin{equation*}
	R_n = O_p\left( \frac{\nu \log n}{n}\right).
\end{equation*}
\end{theorem}

 For log-concave densities, we have the following corollary of Theorem \ref{theorem:Rn_order}, as they satisfy	Conditions (A), (B) and (C) (ii).
	\begin{corollary}\label{col:lcb1b2}
 For any univariate log-concave density, $R_n = O_p\left( \frac{\nu \log n}{n}\right)$.
	\end{corollary}

\section{Average log-density ratio for shape-constrained densities}\label{sect:shape}
In this section, we study the behavior of the average log-density ratio $S_{n}$ under both the null and alternative hypotheses. Specifically, Section \ref{subsect:k_log_density_ratio} focuses on $k$-monotone densities; Section \ref{subsect:cm_log_density_ratio} on completely monotone densities; and Section \ref{subsect:lc_log_density_ratio} on log-concave densities. In each subsection, we develop the rate of convergence of $S_n$ under the null, which is generally faster than that of $M_n$. To show the divergence of the log-likelihood ratio under alternatives, one approach is to study the behavior of the NPMLE under misspecification.  However, this problem has only been studied for 1-monotone densities \citep{Patilea2001} and log-concave densities \citep{cule2010theoretical}. By using a probability inequality for the likelihood ratio from \cite{wong1995probability}, we are able to show the consistency of the test under the alternative without first investigating the limit of the NPMLE under misspecification of the hypothesis classes, under general regularity conditions.

\subsection{$k$-monotone densities}\label{subsect:k_log_density_ratio}
In this subsection, we study the behavior of the average log-density ratio for the class of $k$-monotone densities. Denote $\mathcal{F}_k$ as the class of all $k$-monotone densities with support being $[0,\infty)$ or its subsets, for any $k\in \mathbb{N}$. Let $\hat{f}_{n,k}$ be the NPMLE over the whole class $\mathcal{F}_k$.
Define
\begin{equation}\label{eq:def_Kn}
S_{n,k} := - \frac{1}{n} \sum^n_{i=1} \log \frac{\hat{f}_{n,k}(X_i)}{f_0(X_i)}.
\end{equation}
The main result in this subsection is to establish the rate of convergence of $S_{n,k}$ to $0$ under $H_0$, and the divergence of $S_{n,k}$ under $H_1$, when $f_0$ is allowed to be unbounded at $0$ and may have an unbounded support $[0,\infty)$, enabling us to obtain the asymptotic null distribution of $T_n$ and its consistency under local alternatives.

We first review existing results for bounded densities with bounded support. 
 Denote $\mathcal{F}_k([0, A])$ as the class of $k$-monotone densities with support contained in $[0, A]$, and $\mathcal{F}^B_k([0, A])$ as the subclass of densities in $\mathcal{F}_k([0,A])$ that are all bounded above by $B$.
When $f_0 \in \mathcal{F}^B_k([0,A])$, 
\cite{gao2009rate} obtained an upper bound on the bracketing entropy of the class $\mathcal{F}^B_k([0, A])$ under the Hellinger distance $h$:
\begin{equation}\label{eq:Gao_entropy}
    \log N_{[\cdot]}(\varepsilon, \mathcal{F}^B_k([0, A]), h) \leq C |\log A B |^{\frac{1}{2k}} \varepsilon^{-\frac{1}{k}},
\end{equation}
where $C$ is a constant that depends only on $k$, and $N_{[\cdot]}(\varepsilon, \mathcal{G}, \rho)$ is the bracketing number, which is the minimum number of $\varepsilon$-brackets, defined using the distance metric $\rho$, needed to cover a function class $\mathcal{G}$. Note that \cite{gao2009rate} considered $\mathcal{F}_k^B([0,A])$ instead of $\mathcal{F}_k$ because the latter is not totally bounded.
Using (\ref{eq:Gao_entropy}), for $f_0 \in \mathcal{F}^B_k([0, A])$, 
one can derive, for example, by following the argument in Corollary 7.5 of \cite{van2000empirical}, that
\begin{equation}\label{eq:K_n_order_existing_way}
    S_{n,k} = O_p\left(n^{-\frac{2k}{2k+1}}\right);
\end{equation}
see also Theorem \ref{thm:simple_k_monotone_logLRT_rate} (i).
	
	
	For the case where the density is unbounded and/or has an unbounded support, to the best of our knowledge, the corresponding literature on the rate of convergence of $S_{n,k}$ is lacking, since this class does not have a finite bracketing entropy in terms of the Hellinger distance \citep{gao2009rate}.  Novel extensions are needed to address this problem. The approach we use here relies on the following facts about the NPMLE: (i) it is bounded by $k$ times the  inverse of the minimum order statistic, and (ii) the right endpoint of its support is bounded by $k$ times the maximum order statistic; see Lemma \ref{lemma:k_monotone_bounded_in_prob} and Lemma \ref{lemma:bounds_on_kmonotone_support}, respectively.

	\begin{lemma}\label{lemma:k_monotone_bounded_in_prob}
		If $f_0 \in \mathcal{F}_k$, then $\hat{f}_{n,k}(0+) \leq k Z^{-1}_1$.	If $f_0$ is also bounded from above, then $\hat{f}_{n,k}(0+) = O_p(1)$.
	\end{lemma}
	
	\begin{remark}
		In \cite{Balabdaoui2011grenander}, the authors studied the asymptotic behavior of the Grenander estimator for a $1$-monotone density near zero. Specifically, they considered the situation when the true density is unbounded at zero and established the rate at which $\hat{f}_n(0+)$ diverges to infinity under certain regularity conditions. For example, if $f_0(x) = \gamma x^{\gamma -1}I(0 < x \leq 1)$ with $0<\gamma<1$, then Theorem 1.1 in \cite{Balabdaoui2011grenander} implies that $\hat{f}_{n,k}(0+) = O_p(n^{1/\gamma - 1})$. On the other hand, the bound $\hat{f}_{n,k}(0+) \leq k Z^{-1}_1$ implies that $\hat{f}_{n,k}(0+) = O_p(n^{1/\gamma})$. While our bound is weaker, it does not require any of the additional assumptions made in \cite{Balabdaoui2011grenander} and is sufficient for establishing that $\sqrt{n\nu} S_{n,k} = o_p(1)$; see Theorem \ref{thm:simple_k_monotone_logLRT_rate} and Corollary \ref{cor:simple_k_monotone_null}. Furthermore, to the best of our knowledge, there is no known limit theory for the NPMLE of a $k$-monotone density near zero when $k > 1$.
	\end{remark}

	\begin{lemma}\label{lemma:bounds_on_kmonotone_support}
		If $f_0 \in \mathcal{F}_k$, then   $\sigma_{\hat{f}_{n,k}} \leq k Z_n$.
		\label{lem:kZn}
	\end{lemma}
	
In general, to obtain a rate of convergence of the NPMLE, it suffices to consider a subspace where the NPMLE lies with probability approaching $1$. In view of Lemmas \ref{lemma:k_monotone_bounded_in_prob} and \ref{lemma:bounds_on_kmonotone_support}, we know that $\hat{f}_{n,k} \in \mathcal{F}^{k Z^{-1}_1}_k([0, kZ_n])$. Thus, while $f_0$ may not be bounded from above or may not have a bounded support, its NPMLE is bounded from above and has a bounded support. Then, the rate of convergence of $S_{n,k}$ to $0$ depends on how fast $Z_1$ approaches $0$ and/or how fast $Z_n$ approaches $\infty$. In the following theorem, we establish that as long as $f_0$ satisfies Conditions (A) and (B), there is only a difference in logarithmic order compared to the case where $f_0 \in \mathcal{F}_k^B([0,A])$.\footnote{For simplicity, we state the results that cover most of the common practical cases of interest. A more general statement is given in the appendix, where $f_0$ satisfying Conditions (A) and (B) is a special case.}

	\begin{theorem}\label{thm:simple_k_monotone_logLRT_rate}
		
  Suppose that $f_0 \in \mathcal{F}_k$.
  \begin{enumerate}
      \item[(i)] If $f_0 \in \mathcal{F}_k^B([0,A])$, then $S_{n,k} = O_p\left(n^{-\frac{2k}{2k+1}}\right)$, as stated in (\ref{eq:K_n_order_existing_way}).
      \item[(ii)] If $f_0$ has an unbounded support but satisfies Condition (A) and/or $f_0$ is not bounded from above but satisfies Condition (B), then
      \begin{equation*}
				S_{n,k} 
				= O_p\left(n^{-\frac{2k}{2k+1}} (\log (n) )^{\frac{1}{2k+1}} \right).
			\end{equation*} 
  \end{enumerate}
  
	\end{theorem}

\begin{remark}
    Similar to Theorem \ref{thm:simple_k_monotone_logLRT_rate}, if $f_0 \in \mathcal{F}_k$ satisfies Conditions (A) and (B), we can also obtain
    \begin{equation*}
        h(\hat{f}_{n,k}, f_0) = O_p\left(n^{-\frac{k}{2k+1}} (\log(n))^{\frac{1}{2(2k+1)}}\right).	
    \end{equation*}
    
\end{remark}

Following Theorem \ref{coro:main_result}, Theorems \ref{theorem:Rn_order} and  \ref{thm:simple_k_monotone_logLRT_rate}, we obtain the following Corollary \ref{cor:simple_k_monotone_null}, which establishes the asymptotic null distribution of the NPLRT for the class of $k$-monotone densities.
	
\begin{corollary}\label{cor:simple_k_monotone_null}
    Under Conditions in Theorem \ref{thm:simple_k_monotone_logLRT_rate} and 
    suppose that $\nu = O(n^{1/3}(\log n)^{-1})$, then (\ref{eq:asy_dist_under_H0}) holds.
\end{corollary}

Now consider a sequence of local alternatives, where the underlying density $f_n \notin \mathcal{F}_k$ satisfies  
$\varepsilon_{kn} := \inf_{f \in \mathcal{F}_k} h(f, f_n) > 0$, and assume that $\varepsilon_{kn} \to 0$ as $n \to \infty$.  
%
The following theorem shows that if $\varepsilon_{kn}$ is at least of order $n^{-1/3}$ up to a logarithmic factor and a constant, then the power of the test converges to $1$. The notation $d_n = \Omega(t_n)$ means that there exists a constant $c > 0$ such that $d_n / t_n \geq c$ for all sufficiently large $n$.

\begin{theorem}[Consistency under local alternatives]\label{coro:k_monotone_H1_local_alt_rate}
    Suppose that $\nu \asymp n^{1/3}(\log n)^{-1}$, $\log (a_n b_n) = O(\log n)$, where $\{a_n\}$ and $\{b_n\}$ satisfy (C.25) in the appendix, and $R_n = O_p(\frac{\nu \log n}{n})$. If $\varepsilon_{kn} = \Omega(n^{-1/3} \log n)$, then for any $0 \leq c < \infty$,
    \begin{equation*}
\lim_{n \rightarrow \infty} \mathbb{P}     \left(
\sqrt{ \frac{n\nu}{\nu^2\psi_1(\nu)-\nu}} \left( T_n -  \log \nu +  \psi(\nu) \right) > c\right) = 1.
\end{equation*}
\end{theorem}
\begin{remark}
The setting above is stated in general terms, without assuming a specific form for $f_n$.
 Suppose in addition that $f_n = f_0 + \varepsilon_n g \geq 0$, where $f_0 \in \mathcal{F}_k$, $\int g = 0$, and $\varepsilon_n \downarrow 0$. If $f_0$ satisfies Conditions (A) and (B), and $g$ is a step function with finitely many steps, then it can be shown that $R_n = O_p\left(\frac{\nu \log n}{n}\right)$ and $\log(a_n b_n) = O(\log n)$, following similar arguments as in the proof of Theorem \ref{theorem:Rn_order} and Lemma C.2 in the appendix.
\end{remark}

\subsection{Completely monotone densities}\label{subsect:cm_log_density_ratio}
For completely monotone densities, for which global rates of convergence have not been developed, we obtain a rough rate of convergence that is sufficient for the application of the proposed test.

 Denote $\mathcal{F}_\infty$ as the class of all completely monotone densities on $(0, \infty)$. By  Bernstein's theorem \citep{feller1971introduction},
    \begin{align*}
        \mathcal{F}_\infty &= \bigg\{f:(0,\infty) \rightarrow (0,\infty) : f(t) = \int^\infty_0 \lambda e^{-\lambda t}\, dM(\lambda),  \text{where $M$ is a distribution function} \bigg\}.
    \end{align*}
Note that a completely monotone density can be unbounded at $0$. For example, the gamma distribution with shape parameter in $(0, 1)$ is unbounded at $0$ and completely monotone. A completely monotone density can also have a heavy tail. For example, the density $f(t) = \beta (1+t)^{-(\beta+1)}$ for $\beta > 0$ has a tail that decays polynomially and is completely monotone.
    
    Similar to the class of $k$-monotone densities with unbounded support, we know from \cite{gao2009rate} that $\mathcal{F}_\infty$ is not totally bounded . Moreover, the NPMLE $\hat{f}_{n,\infty}$ for a completely monotone density always has unbounded support, as $\hat{f}_{n,\infty}(t) = \int^\infty_0 \lambda e^{-\lambda t} \, d\hat{M}_n(\lambda)$ for $t > 0$, where $\hat{M}_n$ is the NPMLE of the corresponding mixing distribution. Although completely monotone densities are often viewed as $k$-monotone with $k \in \infty$, the approach in Subsection \ref{subsect:k_log_density_ratio} is not applicable and we will derive a rate of convergence of the log-density ratio using a different method. 
    
     Here, we do not aim to derive a tight bound for the convergence rate of the log-likelihood ratio in the completely monotone case, as this is not the primary focus of the paper. Instead, we provide a loose bound in Lemma \ref{lem:new_rate_cm}, which is sufficient to show that $\sqrt{n \nu} S_{n,\infty}= o_p(1)$ for our NPLRT under the null. The slow rate is merely a consequence of our method of proof.

     The key idea is that we can still obtain an upper bound for the NPMLE. Then, we view the NPMLE as a sum of two decreasing functions with supports of the form $[0, c_{2n}]$ and $(c_{2n}, \infty)$ for some increasing sequence $c_{2n}$. Similar to the $1$-monotone case, we can obtain a finite bracketing entropy for the class of decreasing functions that are bounded above and have bounded support $[0, c_{2n}]$. The sequence $c_{2n}$ is chosen such that the tail of the NPMLE is bounded by $1/t^2$. Using the method in Lemma 7.10 of \cite{van2000empirical}, we can then obtain a finite bracketing entropy for the class of the decreasing functions with unbounded support that are bounded by $1/t^2$.
    
    

  \begin{lemma}\label{lem:new_rate_cm}
        Suppose that $f_0\in \mathcal{F}_\infty$ and satisfies Conditions (A) and (B). Then,
        \begin{equation}\label{eq:CM_LRT_rate}
            S_{n,\infty} = O_p\left( n^{-\frac{2}{3}} |\log n |^{\frac{1}{3}}\right).
        \end{equation}
    \end{lemma}

		\begin{corollary}\label{thm:simple_CM_null}
	Under the conditions in Lemma \ref{lem:new_rate_cm}.	
  Suppose that $\nu = O(n^{1/3}(\log n)^{-1})$, then (\ref{eq:asy_dist_under_H0}) holds.
		\end{corollary}

  For local alternatives in the completely monotone case, analogous results to those in the $k$-monotone case are established. The corresponding statements are provided in Subsection D.3.2 of the appendix.
		

\subsection{Log-concave densities}\label{subsect:lc_log_density_ratio}
		In this subsection, let $\mathcal{F}_{lc}$ denote the class of log-concave densities on $\mathbb{R}$, and let $\hat{f}_{n,lc}$ be the NPMLE over $\mathcal{F}_{lc}$. Define
		\begin{equation*}
			S_{n,lc} := - \frac{1}{n} \sum^n_{i=1} \log \frac{\hat{f}_{n,lc}(X_i)}{f_0(X_i)}.
		\end{equation*}
		Corollary 3.2 in \cite{doss2016global} shows that $S_{n,lc} = O_p(n^{-4/5})$. Additionally, $R_n = O_p\left(\frac{\nu \log n}{n}\right)$, as stated in Corollary \ref{col:lcb1b2}.  Therefore, following Theorem \ref{coro:main_result}, we can establish the asymptotic null distribution for the class of log-concave densities without imposing additional regularity conditions; see Corollary \ref{cor:null_dist_log_concave}.
		\begin{corollary}\label{cor:null_dist_log_concave}
			Suppose that $f_0$ is a log-concave density on $\mathbb{R}$ and $\nu = O(n^{1/3}(\log n)^{-1})$, then (\ref{eq:asy_dist_under_H0}) holds.
				

		\end{corollary}
		

The following corollary considers a sequence of local alternatives, similar to those in Corollary \ref{coro:k_monotone_H1_local_alt_rate}, and establishes the rate at which they can approach the null while still ensuring that the power converges to $1$.

\begin{theorem}[Consistency under local alternatives]\label{thm:lc_H1_local_alt_rate}
            Let $f_{n} \notin \mathcal{F}_{lc}$ be a sequence of local alternatives with common support satisfying $\varepsilon_n := \inf_{f \in \mathcal{F}_{lc}} h(f, f_n) > 0$,   $\sup_n \int |x|^4 f_n(x) dx < \infty$, and $\sup_n f_n(x) < \infty$. 
    Suppose that $\nu \asymp n^{1/3}(\log n)^{-1}$ and $R_n = O_p(\frac{\nu \log n}{n})$. If $\varepsilon_{n} = \Omega(n^{-1/3} \log n)$, then for any $0 \leq c < \infty$,
    \begin{equation*}
\lim_{n \rightarrow \infty} \mathbb{P}     \left(
\sqrt{ \frac{n\nu}{\nu^2\psi_1(\nu)-\nu}} \left( T_n -  \log \nu +  \psi(\nu) \right) > c\right) = 1.
\end{equation*}
\end{theorem}

\section{Bootstrap calibration of the test}\label{sect:bootstrap}
While the asymptotic distribution under the null hypothesis is distribution-free, and thus can be used to determine the critical value of the test, the smaller-order terms, particularly $\sqrt{n\nu} S_n$,  may converge to $0$ slowly. This can make the asymptotic approximation inaccurate in finite samples. One possible way to improve the approximation is to analyze the leading term in $S_n$ and debias accordingly. We first explore this idea in the context of 1-monotonicity to illustrate its limitations. The result is presented in Section F.1 of the appendix, and the proof relies on results from \cite{Kulikov2005} and \cite{Kulikov2008}.


  A more practical and generally applicable way to improve the accuracy of the test is to use a bootstrap procedure to determine appropriate critical values. 
Let $\{\tilde{f}_n\}$ be a sequence of estimated densities, which may be $k$-monotone, completely monotone, or log-concave, where $\tilde{f}_n$ depends on $X_1,\ldots,X_n$, and let $\tilde{F}_{n}$ be the corresponding distribution function. Let $X_{n1}^*,\ldots,X_{nn}^*$ be independent random  variables simulated from the density $\tilde{f}_{n}$, and denote $Z_{n1}^* < \ldots < Z_{nn}^*$ as their order statistics. Let $\hat{f}^*_n$ be the NPMLE based on $X_{n1}^*,\ldots,X_{nn}^*$ under the relevant hypothesis class (either $k$-monotonicity, complete monotonicity, or log-concavity) and let $f_{n,\nu}^{H,*}$ be the corresponding histogram-type estimator.
		Define
		\begin{equation*}
	T_n^* := -\frac{1}{n}\log \prod_{i=1}^n  \frac{ \hat{f}_n^{*}(X_{ni}^{*})}{ \hat{f}_{n,\nu}^{H, *}(X_{ni}^{*})}.
		\end{equation*}
		Let $\mathcal{S}$ denote the sample space, and let $\mathbb{P}^*$ represent the conditional probability given 
		the entire sequence $(X_1,X_2,\ldots)$.
		In this section, we also emphasize the probability measure in the $o_p$ notation by using $o_{\mathbb{P}^*_\omega}$ for $\omega \in \mathcal{S}$ when applicable. Let
  \begin{align*}
S_n^* &:= - \frac{1}{n}\sum^n_{i=1} \log \frac{\hat{f}_n^*(X_{ni}^*)}{\tilde{f}_n(X_{ni}^*)},\\
M_n^* &:= -  \frac{\nu}{n} \sum_{j=0}^{n_{\nu} - 1} 
\log \frac{\tilde{F}_n(Z_{n,(j+1)\nu+1}^*) - \tilde{F}_n(Z_{n,j\nu+1}^*)}{\nu/(n-1)},\\
R_n^*&:=	- \frac{1}{n} \sum_{j=0}^{n_{\nu} - 1} 
\sum^\nu_{l=1} \bigg[ \log \tilde{f}_n(Z_{n,j\nu+l+1}^*) + \log \frac{Z_{n,(j+1)\nu+1}^* - Z_{n,j\nu+1}^*}{\tilde{F}_n(Z_{n,(j+1)\nu+1}^*) - \tilde{F}_n(Z_{n,j\nu+1}^*)} \bigg] -\frac{1}{n} \log \frac{\tilde{f}_n(Z_{n1}^*)}{f^{H,*}_{n,\nu}(Z_{n1}^*)}.
\end{align*}
		
		
		
		\begin{theorem}\label{thm:bootstrap_general}
			Suppose that conditional on $X_1,X_2,\ldots$, every subsequence $\{n_k\}$ of $\{n\}$ has a further subsequence $\{n_{k_l}\}$ along which for almost all $\omega \in \mathcal{S}$,
			\begin{enumerate}
				\item [(i)]
				$\sqrt{n_{k_l} \nu} S^*_{n_{k_l}} = o_{\mathbb{P}^*_\omega}(1)$; and
				\item [(ii)] $\sqrt{n_{k_l}\nu} R^*_{n_{k_l}} =  o_{\mathbb{P}^*_\omega}(1)$.
			\end{enumerate} 
   Then,	\begin{align}\label{eq:bootstrap_consistencty_general}
	&			\sup_{x \in \mathbb{R}} \bigg| \mathbb{P}^*\left( \sqrt{ \frac{n\nu}{\nu^2\psi_1(\nu)-\nu}} \left( T_n^* -  \log \nu +  \psi(\nu) \right)  \leq x\right)  - \mathbb{P}\left( Z \leq x \right) \bigg| \stackrel{\mathbb{P}_{f_0}}{\rightarrow} 0,
			\end{align}
			where $Z\sim N\left(0, 1 \right)$.
		\end{theorem}
		
		In the appendix, we provide details showing that the conditions in Theorem \ref{thm:bootstrap_general} are satisfied for $k$-monotone, completely monotone, and log-concave densities. Specifically, for testing $k$-monotonicity and complete monotonicity, these conditions hold if the true underlying density satisfies Conditions (A) and (B); for testing log-concavity, it is sufficient for the true density to have a finite first absolute moment.  As a result of Theorem \ref{thm:bootstrap_general}, we have the following bootstrap procedure for determining the critical value of the test.

		\begin{bootstrap}\label{bootstrap_procedure}
			~
			\begin{enumerate}
				\item Simulate $X_{n1}^{(b)},\ldots,X_{nn}^{(b)}$ as independent random  variables from the NPMLE $\hat{f}_n$ of $X_1,\ldots,X_n$.
    
				\item Compute the NPMLE $\hat{f}^{(b)}_n$ and the histogram-type estimator $\hat{f}_{n,\nu}^{H,(b)}$ from $X_{n1}^{(b)},\ldots,X_{nn}^{(b)}$.
    
				\item Set
				\begin{equation*}
					T_n^{(b)} := -\frac{1}{n}\log \prod_{i=1}^n  \frac{ \hat{f}_n^{(b)}(X_{ni}^{(b)})}{ \hat{f}_{n,\nu}^{H, (b)}(X_{ni}^{(b)})}.
				\end{equation*}

				\item 
      For a test with significance level $\alpha \in (0, 1)$, 
                repeat Steps 1 to 3 $B$ times to obtain $T^{(1)}_n,\ldots,T^{(B)}_n$, and use their $1-\alpha$ quantile as the critical value for the test.
			\end{enumerate}
		\end{bootstrap}
		
		
		\begin{remark}
			Recall that an important fact used in establishing the asymptotic null distribution is that $(F_0(Z_1),\ldots,F_0(Z_n)) \stackrel{d}{=} (U_{(1)},\ldots,U_{(n)})$.  If we bootstrap from the distribution function $\hat{\mathbb{F}}_n$ corresponding to the NPMLE $\hat{f}_n$, then we still have		$(\hat{\mathbb{F}}_n(Z^*_{n1}),\ldots,\hat{\mathbb{F}}_n(Z^*_{nn})) \stackrel{d}{=} (U_{(1)},\ldots,U_{(n)})$ because $\hat{\mathbb{F}}_n$ is continuous. However, if we bootstrap from the empirical distribution, this property is lost because the empirical distribution is not continuous.  Therefore, to find the null distribution of the NPLRT, we cannot bootstrap from the empirical distribution. 
		\end{remark}
		\begin{remark}
			Although it is known that for 1-monotone densities, bootstrapping from the empirical distribution or the NPMLE does not provide a consistent estimator of the distribution of $n^{\frac{1}{3}}\{\hat{f}_{n,1}(t_0) -f_0(t_0)\}$, where $t_0 \in (0, \infty)$ is an interior point; see \cite{kosorok2008bootstrapping} and \cite{sen2010inconsistency}; our result does not contradict the above literature because the nature of the statistics of interest is different. Our statistic is a global measure of the discrepancy between the NPMLE and the true density, whereas the statistic considered in the above two papers concerns a local property of the NPMLE. 
   
   On a more technical level, bootstrapping from the NPMLE fails to capture the pointwise limiting distribution of the NPMLE at an interior point due to the lack of smoothness in the Grenander estimator, from which the bootstrap samples are generated, while the true underlying distribution function is assumed to have a differentiable density with a nonzero derivative at that interior point.  In contrast, we only require the true underlying distribution function to be continuous, and we do not need to assume the true underlying density has a nonzero derivative at any point when deriving the asymptotic null distribution.
		\end{remark}

\begin{remark}
    Regardless of whether $f_0$ belongs to the hypothesis class, the NPMLE of the density is always contained within the class. Therefore, bootstrap simulations from the NPMLE always approximate the null distribution.
\end{remark}

\section{Simulation studies}\label{sect:simulation}
In this section, we conduct simulation studies to evaluate the finite-sample performance of the proposed tests and compare them with several other tests. For our proposed NPLRT, we have to determine the value of $\nu$ in the histogram-type estimator. Following the standard way of choosing a tuning parameter in density estimation problems; see \cite{wasserman2006all},  we choose $\nu$ by minimizing the cross-validation error, defined as
\begin{equation*}
    CV(\nu) := \int_{\mathbb{R}} \left\{f^H_{n,\nu}(x)\right\}^2 dx - \frac{2}{n}\sum^n_{i=1} f^{H,-i}_{n-1,\nu}(X_i).
\end{equation*}
Here $f^{H,-i}_{n-1,\nu}$ denotes the histogram-type estimator without the $i$th observation. As discussed in Remark \ref{remark:maximize_power}, to maximize power, $\nu$ should be chosen as $\nu \asymp(n^{1/3} (\log n)^{-1})$. In our cross-validation procedure, we choose $\nu$ from $1$ to $\lfloor C n^{1/3}(\log n)^{-1} \rfloor$, where $C > 0$ is a constant. The upper bound is imposed to attain a balance between possible size distortion under the null and improvement of power under the alternatives.  Specifically, we use $C = 30$ for the monotone case, $C = 8$ for the log-concave, $2$-monotone, and completely monotone cases.  Further discussion and rationale on such choices is given in Section H of the appendix.

When $n_\nu = \frac{n - 1}{\nu}$ is not an integer, we define $f^H_{n,\nu}$ by dividing the data into $\lfloor n_\nu \rfloor$ groups as evenly as possible. For example, if $n = 10$ and $\nu = 2$, we have four groups of sizes $3, 2, 2, 2$, so that $f^H_{n,\nu}(x) = \frac{1}{4}(Z_4 - Z_1)^{-1}$ if $x \in [Z_1, Z_4]$, $\frac{1}{4}(Z_6 - Z_4)^{-1}$ if $x \in (Z_4, Z_6]$, $\frac{1}{4}(Z_8 - Z_6)^{-1}$ if $x \in (Z_6, Z_8]$, and $\frac{1}{4}(Z_{10} - Z_8)^{-1}$ if $x \in (Z_8, Z_{10}]$. If $n = 12$ and $\nu = 3$, we obtain three groups of sizes $4, 4, 3$.

The critical values are computed using the bootstrap procedure described in Section \ref{sect:bootstrap}, with $B = 200$ replicates. Based on the original data, we first apply the cross-validation procedure to select an appropriate value of $\nu$, searching over the range $1$ to $\lfloor Cn^{1/3}(\log n)^{-1} \rfloor$. This selected value of $\nu$ is then used consistently for computing each bootstrapped test statistic based on samples generated from the NPMLE.
To estimate the sizes and empirical powers, we conduct 1,000 independent repetitions. The significance level is set to 0.05 throughout, except in Subsection \ref{subsect:symmetric_lc}, where it is set to 0.1 to allow direct comparison with \cite{balabdaoui2018inference}.

The Grenander estimator, the MLE for the $1$-monotone case, can be computed by posing it as an isotonic regression problem. See, for instance, \cite{robertson1988order} for the pool-adjacent-violators algorithm for isotonic regression. Here, we make use of the R package \verb|fdrtool| \citep{fdrtoolR}. Using the nonparametric mixture representation, the computation of $k$-monotone and completely monotone MLEs is performed using the R package \verb|nspmix|; see \cite{wang2007fast}. For the computation of the log-concave MLE, we use the R package \verb|logcondens|; see \cite{dumbgen2011logcondens}.

\subsection{Decreasing densities}\label{subsect:decreasing}
We first consider testing whether a density is decreasing. The distributions considered are listed in Table \ref{table:monotone_result1}. Under $H_0$, we consider decreasing densities with an unbounded support, with a bounded support, and densities that are not bounded from above. Specifically, Exp$(1)$ denotes the exponential distribution with density $f(x) = e^{-x}I(x > 0)$, which has an unbounded support. Beta($a$, $b$) denotes the beta distribution with density $f(x) = \frac{\Gamma(a+b)}{\Gamma(a)\Gamma(b)}x^{a-1}(1-x)^{b-1} I(0 < x < 1)$, which has a bounded support. In particular, Beta$(1, 4)$ has a decreasing density. ``Unbounded'' refers to the decreasing density $f(x) = \frac{1}{2 x^{0.5}} I(0 < x < 1)$, which is not bounded from above. 
Under $H_1$, we consider different scenarios where a density deviates from a decreasing shape. Let $f_1(d, M)$ denote the density
\begin{equation*}
    f_1(x; d, M) \propto x^{-0.1}I(x < d) + M x I(d \leq x \leq 1)
\end{equation*}
for $d \in (0, 1)$ and $M > 0$. The density $f_1(x; d, M)$ is decreasing on $(0, d)$ and increasing on $(d, 1)$. Smaller $d$ and/or larger $M$ lead to more deviation from a decreasing density. The densities of Beta$(1.2,1.5)$ and Beta$(1.5,3)$ are both initially increasing and then decreasing. ``Mixture'' refers to the density $0.7 \frac{2e^{-2x}}{1-e^{-4}}I(0 \leq x \leq 2) + 0.3I(1 \leq x \leq 2)$, which is a mixture of truncated exponential and uniform distribution on $[1,2]$, having two distinct decreasing regions. Lastly, $f_2(c)$, for $c > 0$, denotes the density 
\begin{equation*}
    f_2(x;c) \propto x I(0 < x < c) + e^{-x} I(x \geq c),
\end{equation*}
where $f_2(x;c)$ is increasing on $(0, c)$ and decreasing on $(c, \infty)$. Larger values of $c$ result in greater deviation from a decreasing density. Figure \ref{fig:monotoneH1} shows the plots of these densities.
\begin{figure*}
    \centering
    \includegraphics[width = 12cm]{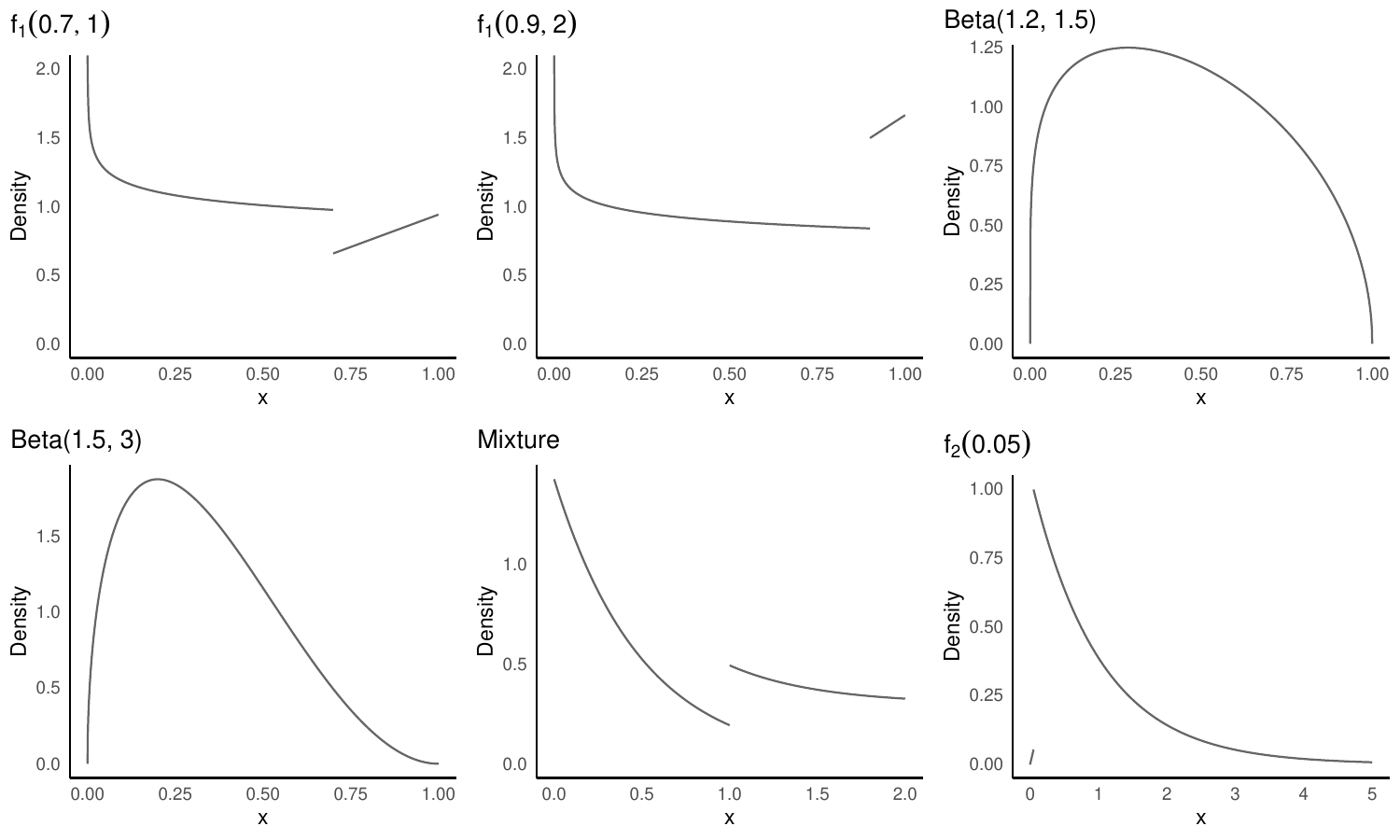}
    \caption{The density functions considered under $H_1$. The density of $f_2(0.2)$ is not shown as it is similar to that of $f_2(0.05)$.}
    \label{fig:monotoneH1}
\end{figure*}

Other natural test statistics for testing monotonicity are based on some distance measures between the empirical distribution function $\mathbb{F}_n$ and its least concave majorant $\hat{\mathbb{F}}_n$. In particular, \cite{kulikov2004testing} studied the $L_k$-distance between $\mathbb{F}_n$ and $\hat{\mathbb{F}}_n$ and derived the asymptotic distribution of the corresponding test statistic. However, their derivation of the asymptotic distribution requires additional assumptions about the underlying $f$. Specifically, they assumed $f$ is twice continuously differentiable, bounded away from $0$, bounded from above, and that its negative derivative is also bounded away from $0$. In addition, their method also requires estimating the derivative $f'$, for example, using kernel estimators. 

The simulation results reported in \cite{kulikov2004testing} show that using the critical value determined by their asymptotic distribution leads to inflated type I errors, making power comparison not meaningful. Because of these issues, we consider the $L_2$-test when the critical value is determined by a bootstrap procedure where we draw samples from the Grenander estimator. Specifically, the $L_2$ test statistics is
\begin{align*}
    L_2 &:= \int_{\mathbb{R}} (\hat{\mathbb{F}}_n(t) - \mathbb{F}_n(t))^2 \, dt.
\end{align*}
In \cite{kulikov2004testing}, two other test statistics were compared:
\begin{align*}
L_2' &:= \int_{\mathbb{R}} (\hat{\mathbb{F}}_n(t) - \mathbb{F}_n(t))^2 \, d \mathbb{F}_n(t),\\
    L_\infty  &:= \sup_{t \in \mathbb{R}} (\hat{\mathbb{F}}_n(t) - \mathbb{F}_n(t)).
\end{align*}
They showed that the uniform distribution is least favorable among all decreasing densities on $[0,1]$, ensuring that the type I error can be bounded by $\alpha$. In our comparison, the critical values of these two tests are also determined using the same bootstrap procedure as for the $L_2$-test.

\begin{table*}
\centering
\begin{tabular}{lllllll}
\hline
  \multicolumn{1}{c}{} && \multicolumn{5}{c}{$n=100$}  \\ \hline
 &Method & LRT & $L_2$ & $L_2'$ & $L_\infty$ & Bayesian\\
  \hline
$H_0$ &Exp(1) & 0.029 & 0.018 & 0.018 & 0.016  & 0.015 \\ 
 & Beta(1,4) & 0.047 & 0.013 & 0.013 & 0.021   & 0.005 \\ 
&  Unbound & 0.026 & 0.019 & 0.010 & 0.017   &  0.033 \\ \hline
$H_1$  & $f_1$(0.7,1) & 0.138 & 0.041 & 0.047 & 0.041   & 0.052 \\ 
  & $f_1$(0.9,2) & 0.407 & 0.252 & 0.233 & 0.222 & 0.135 \\ 
  &Beta(1.2,1.5) & 0.207 & 0.084 & 0.090 & 0.069  &  0.062\\ 
  &Beta(1.5,3) & 0.277 & 0.151 & 0.187 & 0.157  & 0.066\\ 
  &Mixture & 0.151 & 0.125 & 0.088 & 0.132   & 0.141\\ 
  &$f_2$(0.05) & 0.320 & 0.021 & 0.045 & 0.039   & 0.023 \\ 
  &$f_2$(0.2) & 0.592 & 0.429 & 0.505 & 0.685   & 0.108\\ 
   \hline
    \multicolumn{1}{c}{} &&  \multicolumn{5}{c}{$n=250$}  \\ \hline
 &Method & LRT & $L_2$ & $L_2'$ & $L_\infty$ & Bayesian \\ 
  \hline
$H_0$ &Exp(1) &  0.031  &0.003 & 0.010 & 0.018 & 0.016 \\ 
 & Beta(1,4) & 0.067 & 0.008 & 0.006 & 0.013   &0.044 \\ 
&  Unbound &  0.024 & 0.011 & 0.011 & 0.019   &0.041 \\ \hline
$H_1$  & $f_1$(0.7,1) & 0.235 & 0.042 & 0.041 & 0.044  & 0.120\\ 
  & $f_1$(0.9,2) &  0.729 & 0.485 & 0.478 & 0.475  & 0.424\\ 
  &Beta(1.2,1.5) &  0.334 & 0.110 & 0.147 & 0.127  & 0.125\\ 
  &Beta(1.5,3) &  0.456 & 0.370 & 0.392 & 0.353 & 0.298 \\ 
  &Mixture & 0.333 & 0.222 & 0.168 & 0.281   & 0.284 \\ 
  &$f_2$(0.05) & 0.647 & 0.056 & 0.118 & 0.152  &  0.029\\ 
  &$f_2$(0.2) & 0.916 & 0.967 & 0.966 & 0.999 & 0.703\\ 
   \hline
\end{tabular}
\caption{Test for 1-monotonicity: Comparison of the proposed LRT with $L_2$, $L_2'$, $L_\infty$ and the Bayesian test, at sample sizes $n = 100, 250$, under both $H_0$ and $H_1$. All significance levels are set to $0.05$. The table entries represent empirical rejection proportions based on $1{,}000$ independent replications.}
    \label{table:monotone_result1}
\end{table*}

\begin{table*}
\centering
\begin{tabular}{lllllll}
\hline
    \multicolumn{1}{c}{} && \multicolumn{5}{c}{$n=500$} \\ \hline
 &Method & LRT & $L_2$ & $L_2'$ & $L_\infty$ &  Bayesian\\ 
 \hline
$H_0$ &Exp(1) & 0.027 & 0.005 & 0.008 & 0.011 &  0.050 \\ 
 & Beta(1,4) & 0.068 & 0.000 & 0.015 & 0.019 &  0.026\\ 
  &Unbound & 0.035 & 0.011 & 0.007 & 0.011 & 0.010 \\  \hline
  $H_1$ & $f_1$(0.7,1) & 0.269 & 0.062 & 0.055 & 0.051   & 0.073\\ 
  &$f_1$(0.9,2) & 0.900 & 0.790 & 0.763 & 0.827 &  0.697\\ 
  &Beta(1.2,1.5) & 0.402 & 0.190 & 0.208 & 0.172  & 0.153 \\ 
  &Beta(1.5,3) & 0.692 & 0.737 & 0.745 & 0.701 &  0.487 \\ 
  &Mixture & 0.560 & 0.381 & 0.351 & 0.576 &  0.104\\ 
  &$f_2$(0.05) & 0.792 & 0.210 & 0.393 & 0.726 &  0.086\\ 
  &$f_2$(0.2) & 0.994 & 1.000 & 1.000 & 1.000 &  0.992\\ 
   \hline
     \multicolumn{1}{c}{} && \multicolumn{5}{c}{$n=1000$} \\ \hline
 &Method & LRT & $L_2$ & $L_2'$ & $L_\infty$ & Bayesian\\ 
 \hline
$H_0$ &Exp(1) &  0.020 & 0.004 & 0.009 & 0.018 & 0.024\\ 
 & Beta(1,4) &0.046 & 0.006 & 0.009 & 0.013 &0.044\\ 
  &Unbound &0.033 & 0.005 & 0.006 & 0.010 & 0.010\\  \hline
  $H_1$ & $f_1$(0.7,1) &  0.305 & 0.113 & 0.074 & 0.090  & 0.084\\ 
  &$f_1$(0.9,2) & 0.997 & 0.984 & 0.977 & 0.992 & 0.974\\ 
  &Beta(1.2,1.5) & 0.499 & 0.401 & 0.369 & 0.342  & 0.235\\ 
  &Beta(1.5,3) &  0.886 & 0.971 & 0.974 & 0.966&  0.917\\ 
  &Mixture & 0.775 & 0.680 & 0.640 & 0.894 & 0.209\\ 
  &$f_2$(0.05) &  0.879 & 0.999 & 0.878 & 0.999 & 0.196\\ 
  &$f_2$(0.2) &  1.000 & 1.000 & 1.000 & 1.000   &1.000\\ 
   \hline
\end{tabular}
\caption{Test for 1-monotonicity: Comparison of the proposed LRT with $L_2$, $L_2'$, $L_\infty$ and the Bayesian test, at sample sizes $n = 500, 1000$, under both $H_0$ and $H_1$. All significance levels are set to $0.05$. The table entries represent empirical rejection proportions based on $1{,}000$ independent replications.}
    \label{table:monotone_result2}
\end{table*}
{In addition, we compare our method with the Bayesian approach proposed in \cite{chakraborty2022rates}. We adapted the algorithm for testing monotone regression function \citep{chakraborty2021convergence}, available at \url{https://github.com/MoumitaChak/BayesNP-Shape-Inference},  to the setting for testing of monotone density. We used a non-informative Dirichlet prior and set hyperparameters following \cite{chakraborty2021convergence} and the associated software.  Specifically, for distributions with bounded support, we set $\gamma = 1/2$ and $J = \lfloor n^{1/3} \rfloor$.  For distributions with unbounded support, which require additional hyperparameters,  we set $a=r=1$ (satisfying the tail conditions of the distributions), $K = \lfloor \log n / 3 \rfloor$ and $J=\lfloor n^{1/3} (\log n)^{-1/3} \rfloor$. 
Regarding the threshold sequence, \cite{chakraborty2021convergence} considered $M_n = M_0 (\log n)^{\kappa}\, (M_0, \kappa > 0)$ and selected the optimal  of $(M_0,\kappa)$ by evaluating the test on simulated samples from a flat function, analogous to a standard uniform density in our context. While this strategy applies a fixed cutoff sequence, we found that it did not adapt well to the data distributions considered in our study.  It was overly conservative under the null, resulting in low power under the alternative.  Consequently, we implemented a distribution-specific calibration strategy. To determine $(M_0,\kappa)$, we simulated data from the specific distribution under investigation. For null distributions, we simulated directly from the distribution; for alternative distributions, we simulated from the distribution’s least concave majorant (\emph{i.e.} its projection onto the null space). {For each distribution, we generated $1000$ independent datasets with sample sizes $n=100,250,500,1000$.  We selected the combination  $(M_0,\kappa)$ that maximized the sum of empirical rejection rates across sample sizes, subject to the constraint that the empirical type I error did not exceed 
$0.05$ for all sample sizes.} The test statistics for each dataset were computed based on 
1000 posterior samples.

}

The results are presented in Tables \ref{table:monotone_result1}--\ref{table:monotone_result2}. We observe that our LRT, with bootstrap-calibrated critical values, achieves better size control than the other tests. Under the alternative hypothesis, the proposed LRT attains higher power than the  tests in most of the settings considered. 

\subsection{$2$-monotone}
In this subsection, we consider testing whether a density is $2$-monotone using the proposed NPLRT. Table \ref{table:k_monotone} lists the distributions considered under $H_0$ and $H_1$. The definitions of ``Unbounded'', $f_1$, ``Mixture'' and $f_2$ are the same as in Subsection \ref{subsect:decreasing}. In Table \ref{table:k_monotone},  we observe that the empirical rejection proportions are reasonably close to $0.05$ under $H_0$, and the power increases as the sample size grows. Compared to Tables \ref{table:monotone_result1}--\ref{table:monotone_result2}, the power under the null hypothesis of $2$-monotone is higher than that under $1$-monotone when the density is neither $1$-monotone nor $2$-monotone.

\begin{table*}
\centering
\begin{tabular}{lllllll|lllll}
  \hline
& & $n=100$ & $n=250$ & $n=500$ & $n=1000$\\ \hline
$H_0$ &Exp(1) & 0.040 & 0.029 & 0.035 & 0.029 \\ 
  &Beta(1, 2)& 0.051 & 0.043 & 0.052 & 0.058 \\ 
  & Beta(1, 3) & 0.038 & 0.035 & 0.034 & 0.043 \\ 
  \hline
$H_1$ &  $f_1$(0.7, 1) & 0.667 & 0.958 & 1.000 & 1.000 \\ 
  & $f_1$(0.9, 2) &  0.975 & 1.000 & 1.000 & 1.000 \\
  & Beta(1.2, 1.5) & 0.575 & 0.895 & 0.999 & 1.000 \\ 
  & Beta(1.2, 3) &0.368 & 0.582 & 0.882 & 0.991 \\  
  &Mixture & 0.637 & 0.937 & 0.999 & 1.000 \\ 
  & $f_2(0.05)$ & 0.304 & 0.606 & 0.867 & 0.992 \\   
    & $f_2(0.2)$ & 0.709 & 0.975 & 1.000 & 1.000 \\  
   \hline
\end{tabular}
\caption{Test for $2$-monotonicity at sample sizes $n=100, 250, 500$, and $1000$ under $H_0$ and $H_1$. All significance levels are set to $0.05$. The table entries represent empirical rejection proportions based on $1{,}000$ independent replications.}
    \label{table:k_monotone}
\end{table*}

\subsection{Completely monotone densities}
Next, we consider testing whether a density is completely monotone using the proposed NPLRT. Table  \ref{table:completely_monotone} lists the distributions considered under $H_0$ and $H_1$. Consider the mixture density $\sum^K_{k=1} p_k \lambda_k e^{-\lambda_k x}$, for $x > 0$. ``MixExp1'' corresponds to the case where $K=2$, $p = (0.3, 0.7), \lambda = (1, 5)$, while ``MixExp2'' corresponds to the case when $K=3$, $p = (0.3, 0.3, 0.4), \lambda = (1, 5, 10)$. The results are also presented in Table \ref{table:completely_monotone}, where we observe that the empirical rejection proportions are close to $0.05$ under $H_0$, and the power increases as the sample size grows. 

\begin{table*}
\centering
\begin{tabular}{lllllll|lllll}
\hline
& & $n=100$ & $n=250$ & $n=500$ & $n=1000$\\ \hline
$H_0$  & Exp & 0.050 & 0.043 & 0.059 & 0.058 \\ 
  &MixExp1 & 0.051 & 0.047 & 0.063 & 0.057 \\ 
  &MixExp2 & 0.059 & 0.060 & 0.056 & 0.051 \\    \hline
  $H_1$& Gamma(2, 1) & 0.892 & 0.999 & 1.000 & 1.000 \\ 
  & HalfNormal(1) & 0.439 & 0.755 & 0.973 & 1.000 \\ 
  & Halft(5) & 0.123 & 0.151 & 0.283 & 0.502 \\  
  & Halft(20) & 0.308 & 0.562 & 0.853 & 0.986 \\ 
   \hline
\end{tabular}
\caption{Test for complete monotonicity at sample sizes $n=100, 250, 500$, and $1000$ under $H_0$ and $H_1$. All significance levels are set to $0.05$. The table entries represent empirical rejection proportions based on $1{,}000$ independent replications.}
    \label{table:completely_monotone}
\end{table*}

\subsection{Log-concave densities}
Next, we consider testing for log-concave densities. The distributions considered are listed in Table \ref{table:logconcave_result}. Figure \ref{fig:log_density_plots} displays the log-density plots of the alternatives considered for the log-concave case. The density of Beta$(\alpha,\beta)$ is $\frac{x^{\alpha-1}(1-x)^{\beta-1}}{B(\alpha,\beta)}I(x \in (0, 1))$, where $B(\cdot,\cdot)$ is the beta function. LogNormal$(\mu,\sigma)$ refers to the distribution of a random variable whose logarithm is normally distributed with mean $\mu$ and standard deviation $\sigma$. Let $\phi(\cdot)$ denote the density of the standard normal distribution. MixNormal($\mu$) denotes the mixture of two normal distributions with density $\frac{1}{2}\phi(x) + \frac{1}{2}\phi(x-\mu)$, which is log-concave only when $\mu \leq 2$.  

We compare our test with the trace test proposed in \cite{chen2013smoothed} and the split LRT proposed in \cite{dunn2021universal}. 
 The specification of the split LRT requires an additional density estimator and we follow \cite{dunn2021universal} to use a kernel density estimator (KDE). In particular, we follow \cite{dunn2021universal} to use the \verb|kde1d| function from the \verb|kde1d| package in R, which allows users to restrict the support of the KDE. 
  We consider three choices for specifying the support in \verb|kde1d|: (i) the true support; (ii) without any restriction; and (iii) $[\min_i X_i, \max_i X_i]$, an estimated support. The resulting tests are labelled Split1, Split2, Split3, respectively, in Tables \ref{table:logconcave_result}. Note that for the mixtures of normal distributions, the support is the entire real line, so Split1 and Split2 are identical in that case.

The results are presented in Tables \ref{table:logconcave_result}. We found that the trace test can have inflated type I errors ranging from $0.2$ to $0.3$ for exponential and Laplace distributions, whereas the nominal level should be close to $0.05$. This may explain its high powers under the alternatives. On the other hand, while the split LRT controls the type I error level in finite samples via sample splitting, it tends to be overly conservative. Additionally, the performance of the split LRT with KDE relies on a good specification of the support. When the support is not provided in \verb|kde1d|, the powers drop significantly, possibly because of the well-known boundary issue of the KDE. Our proposed test is more powerful than the split LRT in most of the settings considered and does not suffer from boundary problems as in KDE, as it uses a histogram-type estimator for the alternatives. 

\begin{table*}
\centering
\resizebox{\textwidth}{!}{%
\begin{tabular}{lllllll|lllll}
\hline
  \multicolumn{1}{c}{} && \multicolumn{4}{c}{$n=100$} & \multicolumn{4}{c}{$n=250$} \\ \hline
 &Method & LRT & Trace & Split1 & Split2 & Split3 &LRT & Trace & Split1 & Split2 & Split3 \\ 
 \hline
$H_0$ &Exp(1) & 0.060 & 0.223 & 0.000 & 0.000 & 0.000 & 0.036 & 0.155 & 0.000 & 0.000 & 0.000 \\ 
 & Laplace(1) & 0.048 & 0.361 & 0.000 & 0.000 & 0.000 & 0.038 & 0.313 & 0.000 & 0.000 & 0.000 \\ 
  &Normal(0, 1) & 0.019 & 0.015 & 0.000 & 0.000 & 0.000 & 0.021 & 0.011 & 0.000 & 0.000 & 0.000 \\ 
  &MixNormal(2) & 0.026 & 0.010 & 0.000 & 0.000 & 0.000 & 0.023 & 0.013 & 0.000 & 0.000 & 0.000 \\  \hline
 $H_1$ &Beta(2, 0.5) & 0.857 & 0.870 & 0.678 & 0.000 & 0.514 & 0.999 & 0.992 & 0.999 & 0.000 & 0.996 \\ 
  &Beta(0.5, 10) & 0.954 & 0.973 & 0.784 & 0.000 & 0.867 & 1.000 & 0.999 & 0.999 & 0.000 & 0.999 \\ 
  &Beta(0.6, 10) & 0.698 & 0.791 & 0.277 & 0.000 & 0.365 & 0.963 & 0.971 & 0.862 & 0.000 & 0.938 \\ 
  &Beta(0.7, 10) & 0.380 & 0.551 & 0.036 & 0.000 & 0.043 & 0.611 & 0.712 & 0.279 & 0.000 & 0.401 \\ 
  &LogNormal(0, 1) & 0.321 & 0.865 & 0.001 & 0.000 & 0.000 & 0.550 & 0.984 & 0.090 & 0.000 & 0.025 \\ 
  &LogNormal(0, 1.1) & 0.490 & 0.951 & 0.029 & 0.002 & 0.010 & 0.760 & 0.997 & 0.403 & 0.002 & 0.223 \\ 
  &LogNormal(0, 1.2) & 0.671 & 0.981 & 0.125 & 0.002 & 0.072 & 0.941 & 1.000 & 0.754 & 0.007 & 0.607 \\ 
  &MixNormal(3.5) & 0.202 & 0.406 & 0.000 & 0.000 & 0.000 & 0.347 & 0.864 & 0.001 & 0.001 & 0.000 \\ 
  &MixNormal(4) & 0.449 & 0.723 & 0.000 & 0.000 & 0.000 & 0.812 & 0.995 & 0.085 & 0.085 & 0.000 \\ 
  &MixNormal(4.5) & 0.788 & 0.905 & 0.002 & 0.002 & 0.000 & 0.984 & 1.000 & 0.688 & 0.688 & 0.055 \\ 
   \hline 
   \multicolumn{1}{c}{} && \multicolumn{4}{c}{$n=500$} & \multicolumn{4}{c}{$n=1000$} \\ \hline
 &Method & LRT & Trace & Split1 & Split2 & Split3 &LRT & Trace & Split1 & Split2 & Split3 \\ 
   \hline
$H_0$ &Exp(1) & 0.057 & 0.129 & 0.000 & 0.000 & 0.000 & 0.052 & 0.115 & 0.000 & 0.000 & 0.000 \\ 
  &Laplace(1) & 0.037 & 0.264 & 0.000 & 0.000 & 0.000 & 0.038 & 0.241 & 0.000 & 0.000 & 0.000 \\ 
 &Normal(0, 1) & 0.030 & 0.011 & 0.000 & 0.000 & 0.000 & 0.025 & 0.005 & 0.000 & 0.000 & 0.000 \\ 
 & MixNormal(2) & 0.023 & 0.015 & 0.000 & 0.000 & 0.000 & 0.028 & 0.010 & 0.000 & 0.000 & 0.000 \\  \hline
$H_1$ & Beta(2, 0.5) & 1.000 & 1.000 & 1.000 & 0.000 & 1.000 & 1.000 & 1.000 & 1.000 & 0.000 & 1.000 \\ 
 & Beta(0.5, 10) & 1.000 & 1.000 & 1.000 & 0.000 & 1.000 & 1.000 & 1.000 & 1.000 & 0.000 & 1.000 \\ 
& Beta(0.6, 10) & 0.998 & 0.998 & 0.996 & 0.000 & 1.000 & 1.000 & 1.000 & 1.000 & 0.000 & 1.000 \\ 
  &Beta(0.7, 10) & 0.865 & 0.915 & 0.744 & 0.000 & 0.877 & 0.983 & 0.995 & 0.992 & 0.000 & 0.999 \\ 
 & LogNormal(0, 1) & 0.786 & 0.999 & 0.524 & 0.010 & 0.270 & 0.964 & 1.000 & 0.968 & 0.014 & 0.845 \\ 
 & LogNormal(0, 1.1) & 0.961 & 1.000 & 0.899 & 0.011 & 0.766 & 0.999 & 1.000 & 0.999 & 0.047 & 0.997 \\ 
 & LogNormal(0, 1.2) & 0.996 & 1.000 & 0.997 & 0.019 & 0.986 & 1.000 & 1.000 & 1.000 & 0.054 & 1.000 \\ 
 & MixNormal(3.5) & 0.621 & 0.997 & 0.059 & 0.059 & 0.000 & 0.893 & 1.000 & 0.713 & 0.713 & 0.100 \\ 
&  MixNormal(4) & 0.980 & 1.000 & 0.860 & 0.860 & 0.141 & 1.000 & 1.000 & 1.000 & 1.000 & 0.824 \\ 
 & MixNormal(4.5) & 1.000 & 1.000 & 1.000 & 1.000 & 0.772 & 1.000 & 1.000 & 1.000 & 1.000 & 0.847 \\ 
   \hline
\end{tabular}
}
\caption{Comparison of the proposed NPLRT testing for log-concavity with the trace test and split LRT at sample sizes $n=100, 250, 500, 1000$ under both $H_0$ and $H_1$. All significance levels are set to $0.05$. The table entries represent empirical rejection proportions based on $1{,}000$ independent replications.}
    \label{table:logconcave_result}
\end{table*}

\begin{figure*}
    \centering
    \includegraphics[width = 14cm]{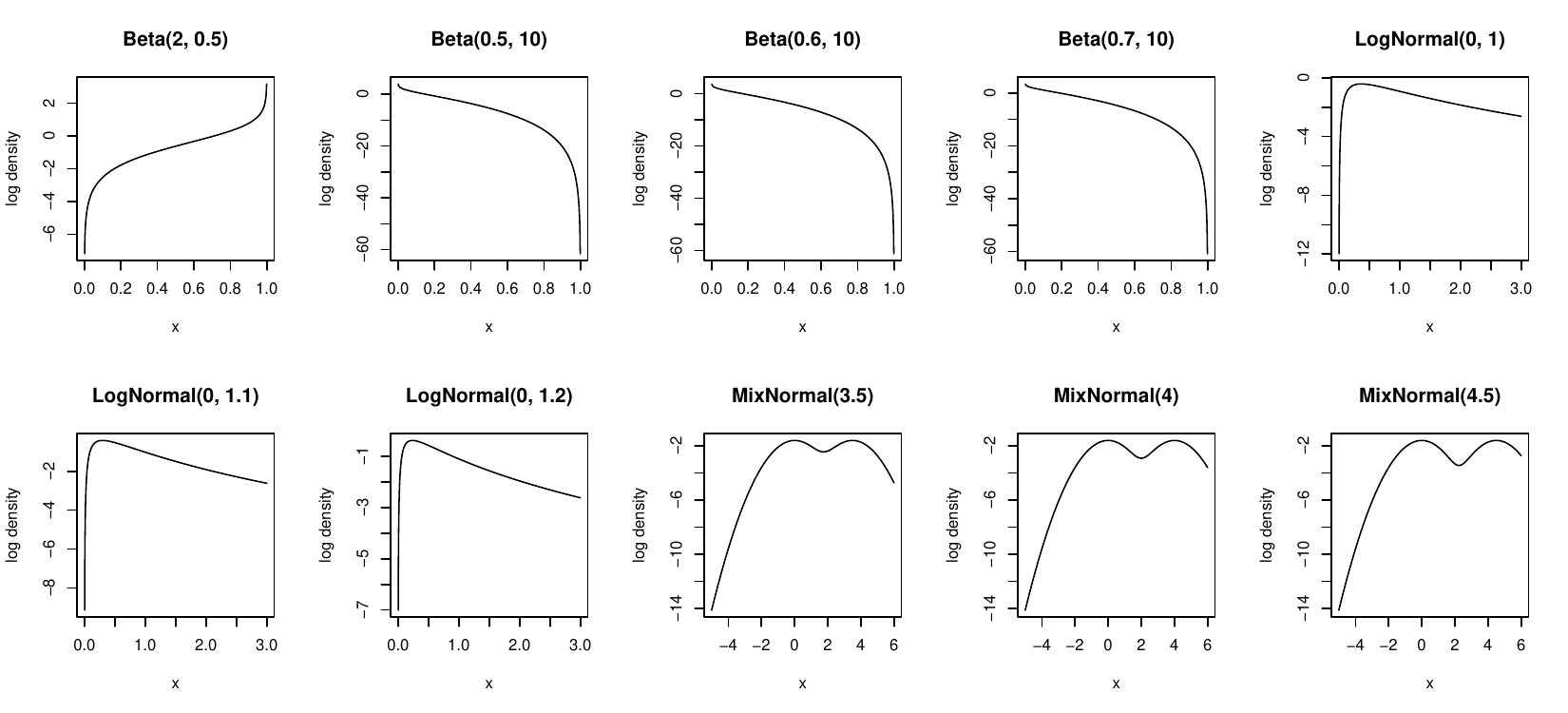}
    \caption{Log density of the distributions considered in the alternative for log-concave case.}
    \label{fig:log_density_plots}
\end{figure*}

\subsection{Symmetric log-concave densities}\label{subsect:symmetric_lc}
In \cite{balabdaoui2018inference}, the authors considered a mixture model of the form
\begin{equation*}
    g^0 = \pi^0 f^0(\cdot - u_1^0) + (1- \pi^0) f^0(\cdot - u_2^0),
\end{equation*}
where $f^0$ is a zero-symmetric log-concave density on $\mathbb{R}$, $\pi^0 \notin \{0, 1/2, 1\}$, and $u_1^0 < u_2^0$. Denote this class of densities by $\mathcal{F}_{lc,mix}$. Their goal was to test for the absence of mixing, i.e., to test the null hypothesis $H_0: u_1^0 = u_2^0$ against the alternative $H_1: u_1^0 \neq u_2^0$ and $\pi^0 \neq 1/2$, under the assumption that $f_0$ is a zero-symmetric log-concave density. They proposed a likelihood ratio statistic based on the symmetric log-concave MLE (centered at the sample median) and the MLE under the mixture model.

In Table~\ref{table:logconcave_sym}, we compare this LRT with our proposed NPLRT under two different null hypotheses: (1) $f_0 \in \mathcal{F}_{lc}$ and (2) $f_0 \in \mathcal{F}_{lc,sym}$, where $\mathcal{F}_{lc}$ denotes the class of log-concave densities, and $\mathcal{F}_{lc,sym}$ denotes the class of symmetric log-concave densities.
 Here, $\mathcal{L}(\mu)$ denotes the Laplace distribution with location parameter $\mu$, and $\pi \mathcal{L}(\mu_1) + (1 - \pi) \mathcal{L}(\mu_2)$ denotes a two-component mixture of Laplace distributions with mixing weights $(\pi, 1 - \pi)$. We adopt the same simulation setting as in \cite{balabdaoui2018inference}, with a sample size of $n = 250$ and a nominal level of $0.1$. Estimation is performed using the R packages \texttt{mixtools} \citep{tatiana2009mixtools} and \texttt{logcondens.mode} \citep{lcmode}. The parameter $\nu$ in our NPLRT is selected via cross-validation, as described at the beginning of this section, with candidate values ranging from $1$ to $\lfloor 8 n^{1/3}(\log n)^{-1} \rfloor$.

We observe that the LRT of \cite{balabdaoui2018inference} shows higher power in detecting deviations from $H_0$ for all four Laplace mixtures considered, for which these alternatives lie within their $H_1$. Note also that they are not log-concave. When we apply our NPLRT under $H_0: f_0 \in \mathcal{F}_{lc,sym}$, the power increases relative to the NPLRT under $H_0: f_0 \in \mathcal{F}_{lc}$, reflecting the effect of testing against a more restrictive null class.
In contrast, the log-normal distribution is neither symmetric, nor log-concave, nor a mixture of symmetric log-concave densities. The test of \cite{balabdaoui2018inference} rejects $H_0$ only about half the time in this case, whereas our NPLRT achieves substantially higher power, particularly when the null hypothesis is restricted to $\mathcal{F}_{lc,sym}$. Our result demonstrates that the relative performance of tests depends on the goal of the research.

\begin{table}
\centering
\resizebox{\textwidth}{!}{%
\begin{tabular}{llrrr}
  \hline
  \multicolumn{1}{c}{Distribution} & Class & \multicolumn{3}{c}{Method} \\ \hline
 & 
  & \cite{balabdaoui2018inference} 
  & Proposed LRT 
  & Proposed LRT \\ 
  & 
  & $H_0: f_0 \in \mathcal{F}_{lc, sym}$ 
  & $H_0: f_0 \in \mathcal{F}_{lc}$ 
  & $H_0: f_0 \in \mathcal{F}_{lc, sym}$ \\ 
  & 
  & $H_1: f_0 \in \mathcal{F}_{lc,mix}$ 
  & $H_1: f_0 \notin \mathcal{F}_{lc}$ 
  & $H_1: f_0 \notin \mathcal{F}_{lc, sym}$ \\ \hline
 $\mathcal{L}(0)$ & $\mathcal{F}_{lc}$, $\mathcal{F}_{lc,sym}$ & 0.110 & 0.095 & 0.073 \\ 
   $0.2 \mathcal{L}(0) + 0.8 \mathcal{L}(1)$ &  $\mathcal{F}_{lc,mix}$, $\mathcal{F}_{lc}^c$, $\mathcal{F}_{lc,sym}^c$ &0.220 & 0.067 & 0.071 \\ 
$0.2 \mathcal{L}(0) + 0.8 \mathcal{L}(3)$&  $\mathcal{F}_{lc,mix}$, $\mathcal{F}_{lc}^c$, $\mathcal{F}_{lc,sym}^c$ &0.990 & 0.238 & 0.736 \\ 
$0.4 \mathcal{L}(0) + 0.6 \mathcal{L}(1)$ &  $\mathcal{F}_{lc,mix}$, $\mathcal{F}_{lc}^c$, $\mathcal{F}_{lc,sym}^c$& 0.140 & 0.069 & 0.078 \\ 
 $0.4 \mathcal{L}(0) + 0.6 \mathcal{L}(3)$ &  $\mathcal{F}_{lc,mix}$, $\mathcal{F}_{lc}^c$, $\mathcal{F}_{lc,sym}^c$ & 0.890 & 0.435 & 0.787 \\  \hline
   LogNormal(0, 1) & $\mathcal{F}_{lc}^c$, $\mathcal{F}_{lc,sym}^c$ &0.610 & 0.633 & 1.000 \\ 
   LogNormal(0, 1.1) & $\mathcal{F}_{lc}^c$, $\mathcal{F}_{lc,sym}^c$& 0.508 & 0.832 & 1.000 \\ 
   LogNormal(0, 1.2) & $\mathcal{F}_{lc}^c$, $\mathcal{F}_{lc,sym}^c$ &0.456 & 0.959 & 1.000 \\ 
   LogNormal(0, 1.3) & $\mathcal{F}_{lc}^c$, $\mathcal{F}_{lc,sym}^c$ & 0.434 & 0.995 & 1.000 \\ 
   LogNormal(0, 1.4) & $\mathcal{F}_{lc}^c$, $\mathcal{F}_{lc,sym}^c$ & 0.456 & 0.999 & 1.000 \\ 
   \hline
\end{tabular}
}
\caption{Comparison of our proposed likelihood ratio tests under the null hypotheses $H_0: f_0 \in \mathcal{F}_{lc}$ and $H_0: f_0 \in \mathcal{F}_{lc,sym}$ with the test proposed by \cite{balabdaoui2018inference} for detecting the presence of mixing. Note that the null and alternative hypotheses differ across the three tests. The first column specifies the underlying distribution, and the second column indicates the corresponding constrained class. The significance level is set to $0.1$, and the sample size is $n = 250$, following \cite{balabdaoui2018inference}.}
    \label{table:logconcave_sym}
\end{table}

\section{Discussion}\label{sect:discussion}
In this paper, we studied a nonparametric likelihood ratio test for univariate shape-constrained densities, focusing on the classes of $k$-monotonicity, complete monotonicity, and log-concavity. A universal asymptotic null distribution is derived, and a bootstrap simulation procedure for determining the critical values of the null distribution is shown to be valid. Consistency of the test is also established. Our empirical findings suggest  that the proposed test is competitive when compared to  other methods and can be applied to a wide range of null and alternative hypotheses. One of the key advantages of our method is that it does not require strong assumptions about the underlying density in order to establish the asymptotic null distribution, the consistency of the test, and the validity of a bootstrap procedure. In contrast, $L_p$-distance-based tests often require additional assumptions on the underlying density in order to establish, for example, the asymptotic null distribution.    In the following, we discuss some additional remarks and related problems.

Recall that under $H_0$, the asymptotic distribution is normal, due to the slower convergence rate of the histogram-type estimator in the general model. This stands in sharp contrast to the classical likelihood ratio test in the parametric setting, where Wilks’ theorem implies a chi-squared asymptotic distribution. In a sense, it is as if the general parameter space has infinitely higher dimension than the restricted parameter space. Furthermore, the likelihood in the general model is not necessarily greater than the MLE under $H_0$, since we employ a sieve estimator under the alternative. One may, of course, redefine $T_n$ as
\begin{equation*}
    \tilde{T}_n := -\frac{1}{n} \log \frac{\prod_{i=1}^n \hat{f}_n(X_i)}{\max\left\{ \prod_{i=1}^n f^H_{n,\nu}(X_i), \prod_{i=1}^n \hat{f}_n(X_i) \right\}}.
\end{equation*}
Note that as $\tilde{T}_n = \max\{T_n, 1/n\}$,  it can be shown that $\tilde{T}_n$ and $T_n$ share the same asymptotic properties under both $H_0$ and $H_1$, provided that $\nu = o(n)$.

		
The NPLRT proposed in this article is not restricted to testing the classes of $k$-monotone, completely monotone, and log-concave densities. It is also valid for testing parametric classes of densities where we usually have $\frac{1}{n} \sum^n_{i=1} \log \frac{\hat{f}_n(X_i)}{f_0(X_i)} = O_p(n^{-1})$. It can also be applied to other nonparametric classes of densities that may not be shape-constrained, for instance, the class of $\alpha$-Holder densities:
		\begin{align*}
			\mathcal{P}_\alpha &:= \{f:[0, 1] \rightarrow \mathbb{R}, |f(x) - f(y)| \leq M |x  - y|^\alpha, x,y \in [0,1] \},
		\end{align*}
		where $M > 0$ and $\alpha > 0$. If $f_0$ is bounded away from $0$ and $\alpha > 1/2$, one can obtain the rate $h(\hat{f}_n, f_0) = O_p(n^{-\alpha/(2\alpha + 1)})$ and $\frac{1}{n} \sum^n_{i=1} \log \frac{\hat{f}_n(X_i)}{f_0(X_i)} = O_p(n^{-2\alpha/(2\alpha + 1)})$; see Example 7.4.6 in \cite{van2000empirical}. Thus, $\sqrt{n\nu} S_{n}=o_p(1)$ under the null hypothesis when $\nu$ does not grow too fast, and the asymptotic null distribution would be valid.
		On the other hand, when $\alpha < 1/2$, \cite{birge1993rates} showed that the rate of convergence of the MLE in Hellinger distance is not better than $(n(\log n))^{-\alpha/2}$. In this case, Theorem \ref{coro:main_result} does not apply because $S_{n}$ becomes the dominating term under $H_0$.

While our proposed likelihood ratio test makes use of a histogram-type estimator because the MLE for the class of all densities does not exist, such a modification is not necessary when a smaller class of alternatives is used in which the MLE exists. For instance, suppose we want to test the hypothesis that $H_0: f_0$ is $2$-monotone versus $H_1: f_0$ is $1$-monotone but not $2$-monotone. Denote $\hat{f}_{n,k}$ as the $k$-monotone MLE. Then, the log-density ratio
		\begin{equation*}
			\Lambda_{n} := - \frac{1}{n} \log \frac{ \prod^n_{i=1} \hat{f}_{n, 2}(X_i)}{\prod^n_{i=1} \hat{f}_{n, 1}(X_i) }
		\end{equation*}
		is well defined. We can expand $\Lambda_{n}$ as 
		\begin{equation*}
			\Lambda_{n} = \frac{1}{n} \sum^n_{i=1} \log \frac{\hat{f}_{n,1}(X_i)}{f_0(X_i)} - \frac{1}{n}\sum^n_{i=1} \log \frac{\hat{f}_{n,2}(X_i)}{f_0(X_i)}.
		\end{equation*}
    In Theorem \ref{thm:simple_k_monotone_logLRT_rate}, we already know that $\frac{1}{n} \sum^n_{i=1} \log \frac{\hat{f}_{n,1}(X_i)}{f_0(X_i)} = O_p(n^{-2/3})$ and $\frac{1}{n}\sum^n_{i=1} \log \frac{\hat{f}_{n,2}(X_i)}{f_0(X_i)} = O_p(n^{-4/5})$  under some mild regularity conditions. Therefore, to determine the asymptotic distribution of $\Lambda_{n}$,  it suffices to study that of $\frac{1}{n}\sum^n_{i=1} \log \frac{\hat{f}_{n,1}(X_i)}{f_0(X_i)}$. In  Theorem F.1 of the appendix, we show that
		\begin{equation*}
			\frac{1}{\sqrt{n}} \sum^n_{i=1} \log \frac{\hat{f}_{n,1}(X_i)}{f_0(X_i)} =  (\mu^2_{2, f_0} + \kappa_{f_0})n^{-1/6} + O_p(n^{-1/3 + \delta}).
		\end{equation*}
		A more detailed analysis is needed to establish its asymptotic distribution and to investigate whether a bootstrap procedure is valid.  

  Some multivariate extensions of spacings have been considered in the literature; see for example, \citep{deheuvels1983strong, zhou1993goodness, li2008multivariate}. Goodness-of-fit test for parametric multivariate distributions using multivariate spacings is an under-explored area, and extensions to multivariate shape-constraint distributions remain to be studied in the future.


\section*{Acknowledgements}
We thank the editor and the anonymous reviewers for their constructive
comments and suggestions, which have significantly improved the quality
of this manuscript.  Phillip
Yam also thanks The University of Texas at Dallas for
the kind invitation to be a Visiting Professor in Naveen Jindal School
of Management.

\section*{Funding}
Kwun Chuen Gary Chan and Chuan-Fa Tang were partially
supported by the US National Institutes of Health grant R01HL122212. Hok
Kan Ling acknowledges the financial supports from RGPIN-2021-03124. Phillip
Yam acknowledges the financial supports from the Research Grant Council of Hong Kong HKGRF-14300717 New kinds of Forward-backward Stochastic Systems with Applications, HKGRF-14300319 Shape-constrained Inference: Testing for Monotonicity, 
HKGRF-14301321 General Theory for Infinite Dimensional Stochastic Control: Mean Field and Some Classical Problems, HKGRF-14300123 Well-posedness of Some Poisson-driven Mean Field Learning Models and their Applications, and HKGRF-14300025 A Generic Theory for Stochastic Control against Fractional Brownian Motions. He was also supported by a grant from the Germany/Hong
Kong Joint Research Scheme sponsored by the Research Grants Council of
Hong Kong and the German Academic Exchange Service of Germany (Reference
No. G-CUHK411/23).

\clearpage

\appendix
\renewcommand{\thesection}{\Alph{section}}
\renewcommand{\thesubsection}{\Alph{section}.\arabic{subsection}}
\renewcommand{\theequation}{\thesection.\arabic{equation}}

\titleformat{\section}
  {\normalfont\Large\bfseries}
  {Appendix~\thesection}
  {1em}
  {}
\section{Real data examples}\label{sect:real}

In this section, we apply our NPLRT to several datasets, selecting the tuning parameter using the same approach as described in the simulation studies section.

\subsection{Coal-mining accidents}

\cite{maguire1952time} considered the time intervals between explosions in British mines that resulted in the loss of ten lives or more from 1875 to 1951. The dataset, comprising 109 observations, was also presented in the same paper. A histogram of the data suggests the density is decreasing; see Figure \ref{fig:coal_hist}. In particular, \cite{maguire1952time} fitted an exponential distribution to the data.  More recently, \cite{mahdavi2017new} proposed a method called the $\alpha$-power transformation to introduce an extra parameter to a family of parametric distributions and  applied their proposed method to this coal-mining accident data. 

\begin{figure}[H]
    \centering
    \includegraphics[width = 7.5cm]{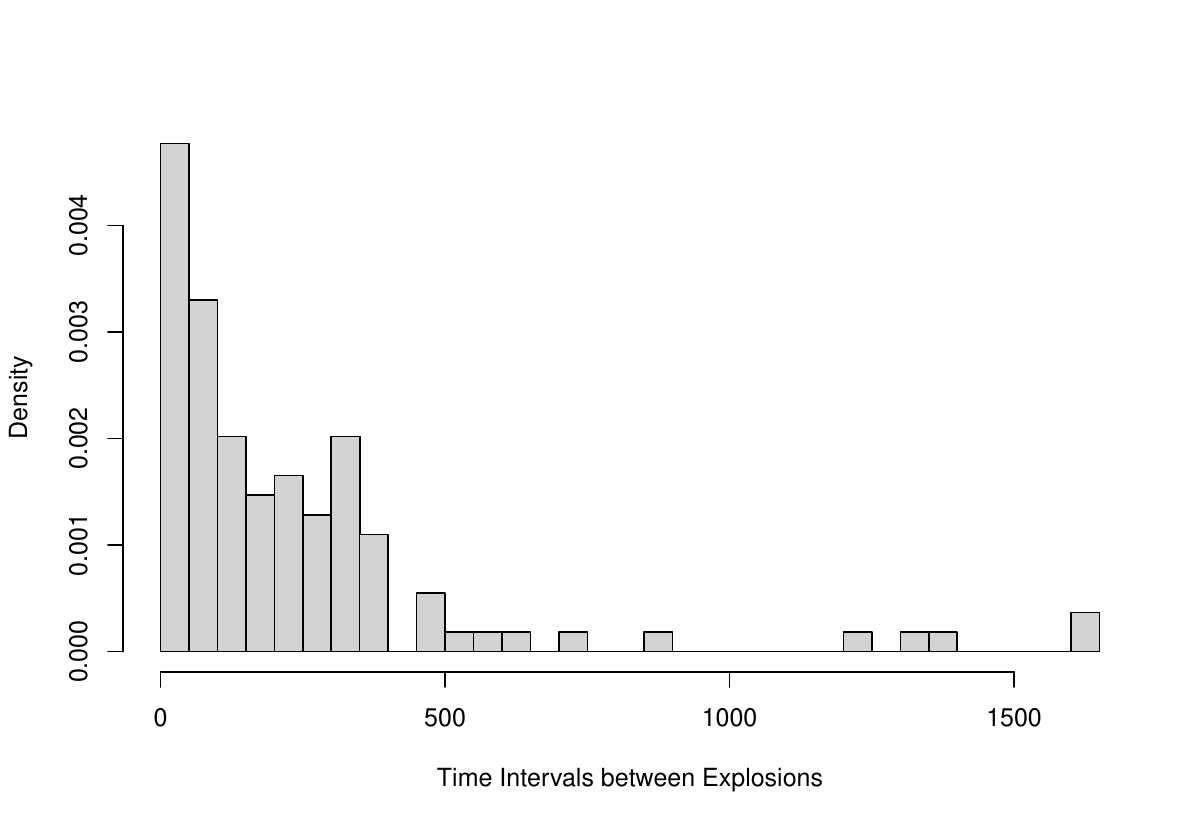}
    \caption{Histogram of the coal-mining accidents data ($n=109$)}
    \label{fig:coal_hist}
\end{figure}

We tested whether the underlying density is decreasing by applying our proposed LRT along with the $L_2, L_2', L_\infty$ tests for comparisons. The resulting $p$-values were $0.16$, $0.64$, $0.61$, and $0.70$, respectively. Since these $p$-values are not small, a decreasing density model is considered plausible. Thus, if a nonparametric method without any tuning parameter is desired, the Grenander estimator may be used. If one wishes to obtain a smoother nonparametric estimate and/or needs to address the inconsistency of the Grenander estimator at $0$, a smoothed Grenander estimator can be utilized; see \cite{groeneboom2014nonparametric}.

\subsection{Hospital length of stay}
Understanding and modeling hospital length of stay, which refers to the number of days a patient stays in the hospital, is important because it serves as a useful indicator of resource utilization and cost-efficiency. \cite{harrison1991balancing} empirically found that the length of stay of patients in departments of geriatric medicine fits extremely well with a mixture of two exponential distributions and proposed a compartmental model to explain this observation.  This implies that the density is completely monotone. More generally, the length of stay depends on patient diagnoses and characteristics \citep{dehouche2023hospital, stone2022systematic}. 

Here, we consider the hospital length of stay data from the Microsoft Machine Learning Server\footnote{https://github.com/Microsoft/ML-Server}. We consider two subgroups of patients that were flagged for renal disease and substance dependence during their encounters. Since there are a large number of tied values, we break the ties by assuming the length of stay is uniformly distributed within each day, adding a standard uniform random variable to each observation. Histograms for these two subgroups are shown in Figures \ref{fig:renal_hist} and \ref{fig:substance_hist}, respectively.

From Figure \ref{fig:renal_hist}, the density of the LoS for the renal group is unlikely to be decreasing. The $p$-values of the NPLRT test and the other three tests for decreasing density, as discussed in Section \ref{sect:simulation}, are all zero. Based on the figure, a log-concave density may be appropriate. The $p$-value from our NPLRT for log-concavity is $0.266$, indicating that there is no evidence against a log-concave fit.

For the substance dependence group, the $p$-value for testing whether the density is decreasing is $0.142$, and the $p$-values of the other three tests are also greater than $0.05$. Thus, we do not reject the null hypothesis that the density is decreasing. From Figure \ref{fig:substance_hist}, the density does not appear to be convex, and our NPLRT confirms this ($p$-value is zero).



\begin{figure}[htbp]
    \centering
    \begin{subfigure}[b]{0.45\textwidth}
        \includegraphics[width=\textwidth]{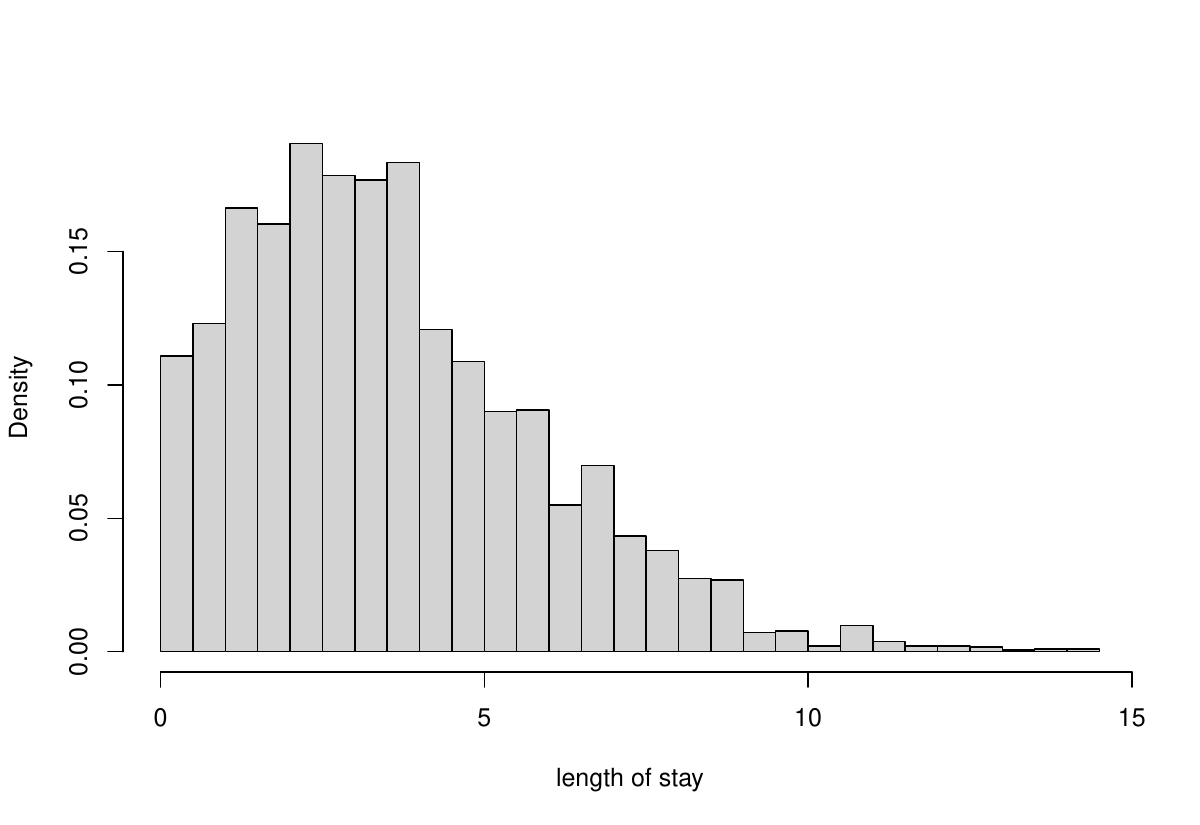}
        \caption{Renal group ($n=3642$)}
        \label{fig:renal_hist}
    \end{subfigure}
    \hfill
    \begin{subfigure}[b]{0.45\textwidth}
        \includegraphics[width=\textwidth]{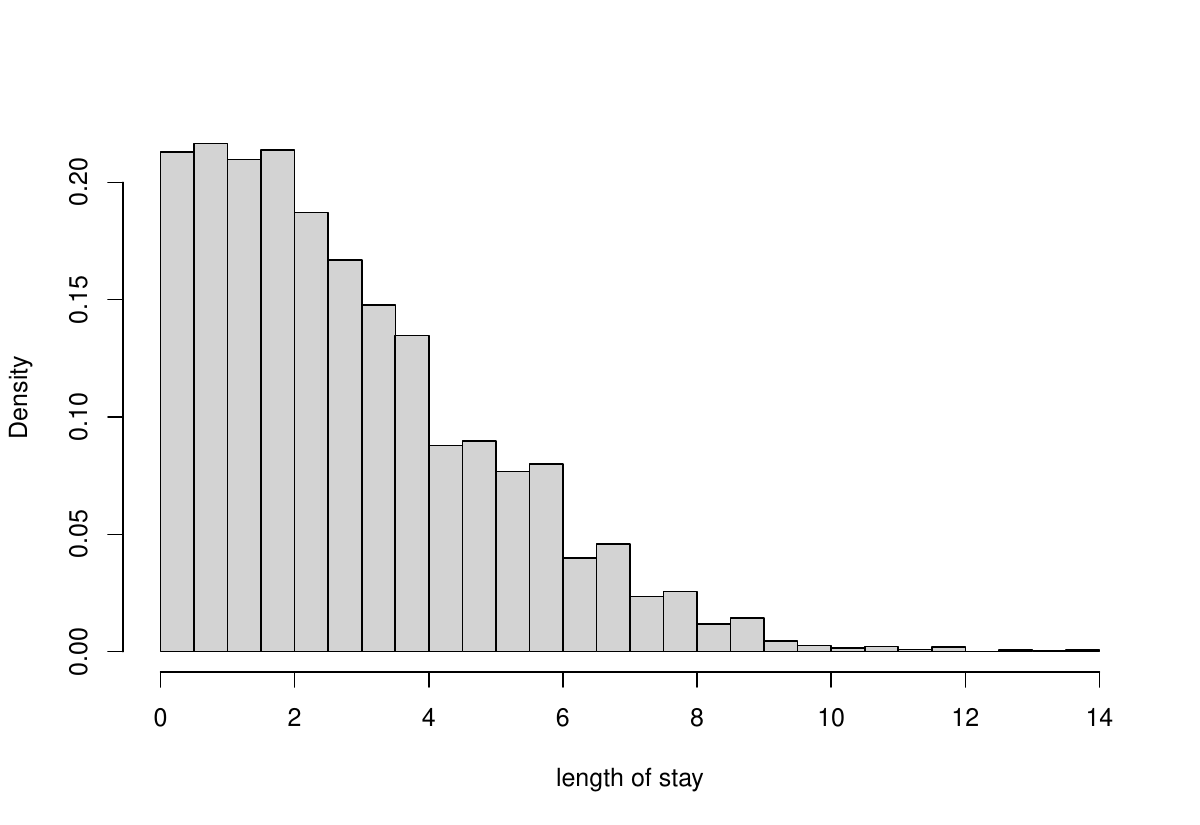}
        \caption{Substance dependence group ($n=6306$)}
        \label{fig:substance_hist}
    \end{subfigure}
    \caption{Histograms of Length of Stay (LoS) for Renal and Substance Dependence Groups}
    \label{fig:los}
\end{figure}

\subsection{Reliability}
\cite{dumbgen2009maximum}
considered modeling the reliability of a certain device. The data, available in the R package \verb|logcondens|, exhibit a skewed and non-Gaussian distribution. A kernel density estimator with a small bandwidth revealed a multimodal distribution, whereas a large bandwidth tended to overestimate the variance and place excessive emphasis on the tails.  \cite{dumbgen2009maximum} employed a slightly smoothed version of the log-concave MLE, which is also log-concave. We applied our LRT to this dataset to test for log-concavity. The resulting $p$-value is $0.995$, indicating that a log-concave density provides a good fit to the data, consistent with the findings in \cite{dumbgen2009maximum}.

\section{Proofs for Section \ref{sect:main_result}}
\subsection{Proofs for Sections \ref{sect:prelim} and \ref{sect:spacings}}
\begin{proof}[Proof of (\ref{eq:decomposition_hist})]
	Using the definition of $f^H_{n,\nu}$,
	\begin{align*}
		&	-\frac{1}{n}\sum^n_{i=1} \log \frac{f_0(X_i)}{f^H_{n,\nu}(X_i)} \\
		&= - \frac{1}{n} \sum_{j=0,\ldots,\frac{n-1}{\nu}-1} 
		\left( \sum^\nu_{l=1} \log \frac{f_0(Z_{j\nu+l+1})}{f^H_{n,\nu}(Z_{j\nu+l+1})}\right) -\frac{1}{n} \log \frac{f_0(Z_1)}{f^H_{n,\nu}(Z_1)}\\
		&= - \frac{1}{n} \sum_{j=0,\ldots, \frac{n-1}{\nu}-1}
		\left( \sum^\nu_{l=1} \log \frac{f_0(Z_{j\nu+l+1})(Z_{(j+1)\nu+1} - Z_{j\nu+1}) }{\nu/(n-1)}\right) -\frac{1}{n} \log \frac{f_0(Z_1)}{f^H_{n,\nu}(Z_1)}\\
		&= - \frac{1}{n} \sum_{j=0,\ldots,\frac{n-1}{\nu} - 1} 
		\left( \sum^\nu_{l=1} \log \frac{f_0(Z_{j\nu+l+1})(Z_{(j+1)\nu+1} - Z_{j\nu+1})}{F_0(Z_{(j+1)\nu+1}) - F_0(Z_{j\nu+1})} + 
		\nu \log \frac{F_0(Z_{(j+1)\nu+1}) - F_0(Z_{j\nu+1})}{\nu/(n-1)} \right) \\
		& \quad -\frac{1}{n} \log \frac{f_0(Z_1)}{f^H_{n,\nu}(Z_1)}\\
		&= M_n + R_n,
	\end{align*}	
where the last equality  follows from the definitions of $M_n$ and $R_n$.
\end{proof}

Before we prove Theorems \ref{coro:main_result}, \ref{corollary:consistency}, and \ref{thm:main_hist_asy_dist}, we first define some additional notations and establish several lemmas. 

Let $E_1,\ldots, E_{n+1}$ be i.i.d. standard exponential random variables. 
Denote $\tilde{E}_j := \sum^{(j+1)\nu+1}_{l=j\nu+2} E_l$ for $j=0,\ldots,\frac{n-1}{\nu}-1$. Then, $\tilde{E}_j$'s are i.i.d. random variables following the Gamma distribution with shape parameter $\nu$ and scale parameter $1$.
Define
\begin{equation*}
	Y_n := \frac{1}{n} \sum_{j=0,\ldots,\frac{n-1}{\nu}-1} (\tilde{E}_j - \nu) + \frac{1}{n-1} (E_1 + E_{n+1}).
\end{equation*}
The following Lemma \ref{lemma:dist_same} decomposes $M_n$ into a deterministic part, a term involving independent and identically distributed random variables ($M_{1n}$ below), and a remainder term, which is shown to be asymptotically negligible in Lemma \ref{lemma:M2n_order}. With the help of Lemma \ref{lemma:log_gamma_moments}, the asymptotic distribution of $M_{1n}$ is established in Lemma \ref{lemma:main_term_distribution_free}.
\begin{lemma}\label{lemma:dist_same}
	We have
	\begin{equation*}
		M_n - \frac{n-1}{n}\log \nu \stackrel{d}{=} M_{1n} + M_{2n},
	\end{equation*}
	where
	\begin{align*}
		M_{1n} &:= \frac{1}{n} \sum_{j=0,\ldots, \frac{n-1}{\nu}-1} \left( \tilde{E}_j - \nu - \nu \log \tilde{E}_j \right),\\
		M_{2n} &:= \frac{1}{n-1} (E_1+E_{n+1})- \frac{Y_n}{n} - (1-n^{-1})\frac{Y_n Y^*_n}{1+Y^*_n},
	\end{align*}
	and $Y^*_n$ lies between $0$ and $Y_n$.
\end{lemma}
\begin{proof}[Proof of Lemma \ref{lemma:dist_same}]
	Note that $(F_0(Z_1),\ldots,F_0(Z_n))$ and $(U_{(1)},\ldots,U_{(n)})$ have the same distribution, where $U_{(i)}$'s are the order statistics from a random sample from Uniform$(0,1)$ with  sample size $n$. Also,
	\begin{equation*}
		(U_{(1)}, U_{(2)} - U_{(1)},\ldots, U_{(n)} - U_{(n-1)} ) \stackrel{d}{=} \bigg( \frac{E_1}{\sum^{n+1}_{h=1} E_h}, \ldots, \frac{E_n}{\sum^{n+1}_{h=1} E_h}\bigg), 
	\end{equation*}
	where $E_h$'s are independent random variables each following the standard exponential distribution with mean $1$; see, for example, Theorem 2.2 in \cite{Devroye1986}. Therefore,
\begin{align*}
    &(F_0(Z_{\nu+1}) - F_0(Z_1), F_0(Z_{2\nu+1}) - F_0(Z_{\nu+1}), \ldots, F_0(Z_n) - F_0(Z_{n-\nu}))\\ &\stackrel{d}{=} (U_{(\nu+1)} - U_{(1)}, U_{(2\nu+1)} - U_{(\nu+1)}, \ldots, U_{(n)} - U_{(n-\nu)})\\
    &\stackrel{d}{=} \left(\frac{ \sum^{\nu+1}_{l=2}E_l}{\sum^{n+1}_{h=1}E_h}, \frac{\sum^{2\nu+1}_{l=\nu+2}E_l}{\sum^{n+1}_{h=1}E_h },\ldots, 
    \frac{\sum^{n}_{l=n-\nu+1}E_l}{\sum^{n+1}_{h=1}E_h } \right).
\end{align*}
Then, by the definition of $\tilde{E}_j$,
\begin{align*}
    M_n	& \stackrel{d}{=} - \frac{\nu}{n} \sum_{j=0,\ldots, \frac{n-1}{\nu}-1}
    \log \left( \frac{ \sum^{(j+1)\nu+1}_{l=j\nu+2} E_l}{\frac{\nu}{n-1}\sum^{n+1}_{h=1} E_h } \right) \\
    &= - \frac{\nu}{n} \sum_{j=0,\ldots, \frac{n-1}{\nu}-1} \log \left( \sum^{(j+1)\nu+1}_{l=j\nu+2} E_l \right) + \frac{n-1}{n} \log \left( \frac{\nu}{n-1}\sum^{n+1}_{h=1} E_h \right)\\
& =- \frac{\nu}{n} 	\sum_{j=0,\ldots,\frac{n-1}{\nu}-1}  \log \tilde{E}_j +  \frac{n-1}{n} \log \left( \frac{\nu}{n-1} \sum_{j=0,\ldots,\frac{n-1}{\nu}-1} \tilde{E}_j + \frac{\nu}{n-1} (E_1 + E_{n+1})\right).
\end{align*}
	As $\mathbb{E}(\tilde{E}_j) = \nu$, we write
	\begin{align*}
		M_n &\stackrel{d}{=} - \frac{\nu}{n} 	\sum_{j=0,\ldots,\frac{n-1}{\nu}-1}  \log \tilde{E}_j + 
		\frac{n-1}{n}	\log \left(\nu + \frac{\nu}{n-1} \sum_{j=0,\ldots,\frac{n-1}{\nu}-1} (\tilde{E}_j - \nu) + \frac{\nu}{n-1} (E_1 + E_{n+1})\right)\\
		&= - \frac{\nu}{n} 	\sum_{j=0,\ldots, \frac{n-1}{\nu}-1}  \log \tilde{E}_j + \frac{n-1}{n} \log \nu +  \frac{n-1}{n}\log 
		\left(1 + Y_n\right),
	\end{align*}
 where the last equality follows from the definition of $Y_n$.	By Taylor's theorem,
	\begin{align*}
		\log 
		\left(1 + Y_n \right)
		= \frac{Y_n}{1+Y^*_n} 
		= Y_n - \frac{Y_n Y^*_n}{1+Y^*_n},
	\end{align*}
	where $Y_n^*$ lies between $0$ and $Y_n$. Thus,
	\begin{align*}
		M_n &\stackrel{d}{=} 
		-\frac{\nu}{n} \sum_{j=0,\ldots, \frac{n-1}{\nu}-1}  \log \tilde{E}_j + \frac{n-1}{n}\log \nu + Y_n - \frac{Y_n}{n}  -  (1-n^{-1})\frac{Y_n Y^*_n}{1+Y^*_n}.
	\end{align*}
The claim in the lemma follows from the definitions of $Y_n, M_{1n}$ and $M_{2n}$.
	
\end{proof}

\begin{lemma}\label{lemma:M2n_order}
Suppose that $\nu = o(n)$. Then, $\sqrt{n}Y_n \stackrel{d}{\rightarrow} N(0, 1)$ and $M_{2n} = O_p(n^{-1})$.
\end{lemma}

\begin{proof}[Proof of Lemma \ref{lemma:M2n_order}]
First, note that $E_1$ and $E_{n+1}$ are $O_p(1)$. Thus,
\begin{equation*}
	Y_n = \frac{1}{n} \sum_{j=0,\ldots,\frac{n-1}{\nu}-1} (\tilde{E}_j - \nu) + O_p(n^{-1}).
\end{equation*}
Note that
\begin{equation*}
	\sum_{j=0,\ldots,\frac{n-1}{\nu}-1} \Var(\tilde{E}_j) = \frac{n-1}{\nu} \cdot \nu = n-1.
\end{equation*}
By the Lindeberg's central limit theorem, if for all $\varepsilon > 0$, we have
\begin{equation}\label{eq:Lindeberg_condition1}
	\lim_{n\rightarrow \infty} \frac{1}{n-1}\sum_{j=0,\ldots, \frac{n-1}{\nu}-1} \mathbb{E}((\tilde{E}_j - \nu)^2 \mathbbm{1}(|\tilde{E}_j - \nu| > \varepsilon \sqrt{n-1})) = 0,
\end{equation}
then 
\begin{equation*}
	\frac{1}{\sqrt{n-1}} \sum_{j=0,\ldots,\frac{n-1}{\nu}-1}(\tilde{E}_j - \nu) \stackrel{d}{\rightarrow} N(0, 1)
\end{equation*}
so that $\sqrt{n}Y_n \stackrel{d}{\rightarrow} N(0, 1)$ by Slutsky's theorem. Now, we verify the Lindeberg's condition (\ref{eq:Lindeberg_condition1}). Fix $\varepsilon  >0$.  Since $\tilde{E}_j$'s have the same distribution,
\begin{align*}
	 \frac{1}{n-1}\sum_{j=0,\ldots, \frac{n-1}{\nu}-1} \mathbb{E}((\tilde{E}_j - \nu)^2 \mathbbm{1}(|\tilde{E}_j - \nu| > \varepsilon \sqrt{n-1})) 
= \frac{1}{\nu}\mathbb{E}((\tilde{E}_1 - \nu)^2 \mathbbm{1}(|\tilde{E}_1 - \nu| > \varepsilon \sqrt{n-1})).
\end{align*}
Note that
\begin{equation*}
	\mathbb{E}\left( \frac{\tilde{E}_1-\nu}{\sqrt{n-1}}\right)^2 = \frac{\nu}{n-1} \rightarrow 0
\end{equation*}
as $n \rightarrow \infty$ since $\nu = o(n)$. Thus, $|\tilde{E}_1 - \nu|/\sqrt{n-1} \stackrel{\mathbb{P}}{\rightarrow} 0$. For any $\varepsilon' > 0$,
\begin{equation*}
	\mathbb{P}((\tilde{E}_j - \nu)^2\mathbbm{1}(|\tilde{E}_j - \nu| > \varepsilon \sqrt{n-1})) > \varepsilon') \leq \mathbb{P}(|\tilde{E}_1-\nu|/\sqrt{n-1} > \varepsilon) \rightarrow 0,
\end{equation*}
as $n \rightarrow \infty$. Thus, $(\tilde{E}_j - \nu)^2\mathbbm{1}(|\tilde{E}_j - \nu| > \varepsilon \sqrt{n-1}) \stackrel{\mathbb{P}}{\rightarrow} 0$. As
\begin{equation*}
	\frac{1}{\nu} (\tilde{E}_j - \nu)^2\mathbbm{1}(|\tilde{E}_j - \nu| > \varepsilon \sqrt{n-1}) \leq \frac{(\tilde{E}_j - \nu)^2}{\nu}
\end{equation*}
and 
\begin{equation*}
	\mathbb{E}\left(\frac{(\tilde{E}_j - \nu)^2}{\nu}\right) = 1,
\end{equation*}
we have, by the dominated convergence theorem,
\begin{equation*}
	\frac{1}{\nu}\mathbb{E}((\tilde{E}_1 - \nu)^2 \mathbbm{1}(|\tilde{E}_1 - \nu| > \varepsilon \sqrt{n-1})) \rightarrow 0
\end{equation*}
as $n \rightarrow \infty$, verifying (\ref{eq:Lindeberg_condition1}).
For the claim that $M_{2n} = O_p(n^{-1})$, observe that $|Y^*_n| \leq |Y_n| = O_p(n^{-1/2})$. Therefore,
\begin{equation*}
	M_{2n} = \frac{O_p(1)}{n} + \frac{O_p(n^{-1/2})}{n} + \frac{O_p(n^{-1/2})O_p(n^{-1/2})}{1 + O_p(n^{-1/2})} = O_p(n^{-1}).
\end{equation*}
\end{proof}

Recall that $\Gamma(\cdot), \psi(\cdot)$ and $\psi_1(\cdot)$ denote the gamma, digamma, and trigamma functions, respectively.
It is well known that the log-Gamma random variable $\log \tilde{E}_j$ has an expected value of $\psi(\nu)$ and variance of $\psi_1(\nu)$. For completeness, these results are provided in Lemma \ref{lemma:log_gamma_moments} (a) and (b).
\begin{lemma}\label{lemma:log_gamma_moments}
Let $Y \sim \text{Gamma}(\nu, 1)$. We have
\begin{enumerate}[(a)]
    \item For $m \in \mathbb{N}$, $\mathbb{E}( (\log (Y))^m) = \Gamma^{(m)}(\nu)/\Gamma(\nu)$, where $\Gamma^{(m)}$ is the $m$th derivative of $\Gamma$. In particular, $\mathbb{E}(\log (Y)) = \psi(\nu)$;
    \item $\Var(\log(Y)) = \psi_1(\nu)$;
    \item $\Cov(Y, \log (Y)) = 1$.
\end{enumerate}
\end{lemma}

\begin{proof}[Proof of Lemma \ref{lemma:log_gamma_moments}]
\begin{enumerate}[(a)]
\item First, the density of $\log Y$ is given by
\begin{equation*}
    f_{\log Y}(y) = \frac{1}{\Gamma(\nu)}e^{y \nu} e^{-e^y},\quad y \in \mathbb{R}.
\end{equation*}
Thus, for any $y \in \mathbb{R}$ and $m \in \mathbb{N}$,
\begin{equation*}
    \frac{d^m}{d \nu^m} e^{y \nu} e^{-e^y} = y^m e^{y\nu} e^{-e^y} = \Gamma(\nu) y^m f_{\log Y}(y).
\end{equation*}
Integrating both sides of the above equation with respect to $y$,
\begin{equation*}
\Gamma(\nu)		\mathbb{E}((\log Y)^m) 
=  \int^\infty_{-\infty} \frac{d^m}{d \nu^m} e^{y \nu} e^{-e^y} dy 
= \frac{d^m}{d\nu^m}  \int^\infty_{-\infty} e^{y \nu} e^{-e^y} dy 
= \Gamma^{(m)}(\nu).
\end{equation*}
In particular, $\mathbb{E}(\log(Y)) = \Gamma'(\nu)/\Gamma(\nu) = \frac{d}{d\nu} \log \Gamma(\nu) = \psi(\nu)$.
\item $\Var(\log (Y)) = \frac{\Gamma''(\nu)}{\Gamma(\nu)} - \psi^2(\nu) = \frac{d}{d\nu } \psi(\nu) = \psi_1(\nu)$.
\item First,
\begin{align*}
    \mathbb{E}(Y\log(Y)) &=  \int^\infty_{-\infty} \frac{1}{\Gamma(\nu)} e^y y e^{y\nu} e^{-e^y} dy 
    = \frac{1}{\Gamma(\nu)}  \int^\infty_{-\infty} \frac{d}{d\nu} e^{y(\nu+1)} e^{-e^y} dy \\
    &= \frac{1}{\Gamma(\nu)}  \frac{d}{d\nu} \int^\infty_{-\infty} e^{y(\nu+1)} e^{-e^y} dy
    = \frac{\Gamma'(\nu+1)}{\Gamma(\nu)} = \frac{\Gamma(\nu+1)}{\Gamma(\nu)} \cdot \frac{\Gamma'(\nu+1)}{\Gamma(\nu+1)}\\
    & = \nu\psi(\nu+1).
\end{align*}
Using the recurrence relation $\psi(\nu+1) = \psi(\nu) + 1/\nu$, we obtain
\begin{equation*}
    \Cov(Y,\log(Y)) = \nu\psi(\nu+1) - \nu\psi(\nu) = \nu (\psi(\nu+1) - \psi(\nu)) = \nu \cdot\frac{1}{\nu} = 1.
\end{equation*}
\end{enumerate}
\end{proof}

\begin{lemma}\label{lemma:main_term_distribution_free}
If $\nu = o(n)$, then
	\begin{equation}\label{eq:asy_dist_3}
		\frac{1}{\sqrt{\frac{n-1}{\nu} (\nu^2 \psi_1(\nu) - \nu)}} \sum_{j=0,\ldots, \frac{n-1}{\nu}-1}(\tilde{E}_j - \nu - \nu \log \tilde{E}_j + \nu\psi(\nu)) \stackrel{d}{\rightarrow} N(0, 1).
	\end{equation}
\end{lemma}

\begin{proof}[Proof of Lemma \ref{lemma:main_term_distribution_free}]
We shall apply the Lindeberg's central limit theorem. By Lemma \ref{lemma:log_gamma_moments} (a),
\begin{align*}
	\mathbb{E}(\tilde{E}_j - \nu - \nu \log \tilde{E}_j + \nu \psi(\nu))) = 0.
\end{align*}
In addition, by Lemma \ref{lemma:log_gamma_moments} (b) and (c), 
\begin{align*}
	\Var(\tilde{E}_j - \nu - \nu \log \tilde{E}_j) &= \Var(\tilde{E}_j) + \nu^2 \Var(\log \tilde{E}_j) -2\nu\Cov(\tilde{E}_j, \log \tilde{E}_j) \\
	&= \nu + \nu^2 \psi_1(\nu) - 2\nu = \nu^2 \psi_1(\nu) - \nu.
\end{align*}
If the Lindeberg's condition holds, we have (\ref{eq:asy_dist_3}).
Thus, it remains to verify the Lindeberg's condition  that for all $\varepsilon  >0$,
\begin{align*}
&	\lim_{n\rightarrow \infty}\frac{1}{\frac{n-1}{\nu}(\nu^2\psi_1(\nu) - \nu)} \sum_{j=0,\ldots,\frac{n-1}{\nu}-1}\mathbb{E} \bigg\{ (\tilde{E}_j - \nu - \nu \log \tilde{E}_j + \nu\psi(\nu))^2\\
&	\quad \quad \quad \quad  \cdot \mathbbm{1}\left((\tilde{E}_j - \nu - \nu \log \tilde{E}_j + \nu\psi(\nu)) > \varepsilon \sqrt{ \frac{n-1}{\nu}(\nu^2\psi_1(\nu)-\nu)}\right) \bigg\}= 0.
\end{align*}
Denote $F_\nu := \tilde{E}_1 -\nu-\nu\log \tilde{E}_1 +\nu\psi(\nu)$. The above condition is equivalent to
\begin{equation}\label{eq:Linderberg_cond2}
	\lim_{n\rightarrow \infty} \frac{1}{\nu^2\psi_1(\nu) - \nu} \mathbb{E}\left(F^2_\nu\cdot \mathbbm{1}\left(|F_\nu| > \varepsilon \sqrt{ \frac{n-1}{\nu}(\nu^2\psi_1(\nu)-\nu)}\right) \right)=0.
\end{equation}
As $\nu = o(n)$, we have
\begin{equation*}
	\mathbb{E} \left( \frac{\nu}{n-1} \cdot \frac{F^2_\nu}{\nu^2\psi_1(\nu)-\nu}\right) = \frac{\nu}{n-1} \rightarrow 0
\end{equation*}
as $n \rightarrow \infty$. This implies that $\sqrt{\frac{\nu}{n-1}} |F_\nu|/\sqrt{\nu^2\psi_1(\nu)-\nu} \stackrel{\mathbb{P}}{\rightarrow} 0 $ which in turn implies that
\begin{equation*}
	F^2_\nu\cdot \mathbbm{1}\left(|F_\nu| > \varepsilon \sqrt{ \frac{n-1}{\nu}(\nu^2\psi_1(\nu)-\nu)}\right) \stackrel{\mathbb{P}}{\rightarrow} 0.
\end{equation*}
As 
\begin{equation*}
		F^2_\nu\cdot \mathbbm{1}\left(|F_\nu| > \varepsilon \sqrt{ \frac{n-1}{\nu}(\nu^2\psi_1(\nu)-\nu)}\right) \leq F^2_\nu,
\end{equation*}
$\mathbb{E}(F^2_\nu) = \nu^2\psi_1(\nu) -\nu$ and $\lim_{n \rightarrow \infty}(\nu^2 \psi_1(\nu) - \nu) = 1/2$ if $\nu \rightarrow \infty$, (\ref{eq:Linderberg_cond2}) holds by the dominated convergence theorem.

\end{proof}

\begin{proof}[Proof of Theorem \ref{thm:main_hist_asy_dist}]
	By Lemma \ref{lemma:dist_same},
	\begin{equation}\label{eq:Mn_1}
		M_n - \frac{n-1}{n}\log \nu \stackrel{d}{=} M_{1n} + M_{2n}.
	\end{equation}
As $\nu = o(n)$, by Lemma \ref{lemma:M2n_order},
	\begin{equation}\label{eq:Mn_2}
		M_{2n} = O_p(n^{-1}).
	\end{equation}
	Note that
	\begin{align}
		&		\frac{n\sqrt{\nu}}{\sqrt{(n-1)(\nu^2 \psi_1(\nu) - \nu)}}\left( M_{1n} + \frac{n-1}{n} \psi(\nu) \right) \nonumber \\
		&=		\frac{1}{\sqrt{\frac{n-1}{\nu} (\nu^2 \psi_1(\nu) - \nu)}} \sum_{j=0,\ldots, \frac{n-1}{\nu}-1}(\tilde{E}_j - \nu - \nu \log \tilde{E}_j + \nu\psi(\nu)) \stackrel{d}{\rightarrow} N(0, 1), \label{eq:Mn_3}
	\end{align}
	where the last convergence follows from Lemma \ref{lemma:main_term_distribution_free} as $\nu = o(n)$.
	In view of (\ref{eq:Mn_1}), (\ref{eq:Mn_2}), (\ref{eq:Mn_3}) and Slutsky's theorem, the proof is completed.
\end{proof}

\begin{proof}[Proof of Theorem \ref{coro:main_result}]
By the decomposition of $T_n$, Theorem \ref{thm:main_hist_asy_dist}, the conditions on $\nu, S_n$, and $R_n$, and the Slutsky's theorem, we have
\begin{align*}
&\sqrt{ \frac{n\nu}{\nu^2\psi_1(\nu)-\nu}} \left( T_n -  \log \nu +  \psi(\nu) \right) \\
&=\sqrt{ \frac{n\nu}{\nu^2\psi_1(\nu)-\nu}} \left( M_n -  \log \nu +  \psi(\nu) \right) +  \sqrt{ \frac{n\nu}{\nu^2\psi_1(\nu)-\nu}}(S_n + R_n)\\
&= \sqrt{ \frac{n\nu}{\nu^2\psi_1(\nu)-\nu}} \left( M_n -  \log \nu +  \psi(\nu) \right) + o_p(1) + O(n^{2/3} (\log n)^{-1/2}) O_p(n^{1/3} / n)\\
& \stackrel{d}{\rightarrow} N(0, 1).
\end{align*}

\end{proof}

\begin{proof}[Proof of Theorem \ref{corollary:consistency}]
	Fix $c \geq 0, \varepsilon> 0$ and $\delta > 0$. From the condition $\nu = O(n^{1/3}/\log n)$ and $R_n = O_p\left(\frac{\nu \log n}{n}\right)$, we have as in the proof of Theorem \ref{coro:main_result} that $\sqrt{n\nu} R_n = o_p(1)$.
For all sufficiently large $n$, we have
\begin{equation}\label{eq:cons1}
    \mathbb{P}_{H_1}( |\sqrt{n\nu} R_n| < \delta) > 1 - \frac{\varepsilon}{3}.
\end{equation}
 From $\lim_{n\rightarrow \infty} \mathbb{P}(\sqrt{n \nu} S_{n} > L_n) = 1$, we have for all sufficiently large $n$ that
\begin{equation}\label{eq:cons2}
    \mathbb{P}_{H_1}( \sqrt{n \nu} S_n > L_n) > 1 - \frac{\varepsilon}{3}.
\end{equation}
 By Theorem \ref{thm:main_hist_asy_dist}, there exists $K > 0$ such that for all sufficiently large $n$,
\begin{equation}\label{eq:cons3}
    	\mathbb{P}_{H_1}\left( \left| \sqrt{ \frac{n\nu}{\nu^2\psi_1(\nu)-\nu}} \left( M_n -  \log \nu +  \psi(\nu) \right)\right| < K \right) > 1 - \frac{\varepsilon}{3}. 
\end{equation}
 	Hence, in view of (\ref{eq:cons1})-(\ref{eq:cons3}) and the fact that $L_n \uparrow \infty$, for all sufficiently large $n$, 
\begin{align*}
    	&\mathbb{P}_{H_1}\left(  \sqrt{ \frac{n\nu}{\nu^2\psi_1(\nu)-\nu}} \left( T_n -  \log \nu +  \psi(\nu) \right) > c \right)\\
     & \geq
     \mathbb{P}_{H_1}\left(  \sqrt{ \frac{n\nu}{\nu^2\psi_1(\nu)-\nu}} \left( T_n -  \log \nu +  \psi(\nu) \right)  > - K -\delta + L_n \right)
          > 1 - \varepsilon,\\
\end{align*}
and the claim follows as desired.
	
\end{proof}

\subsection{Proofs for Section \ref{sect:suff_cond}}
In this subsection, we will prove Theorem \ref{theorem:Rn_order} and Corollary 
\ref{col:lcb1b2}. We will first establish a few lemmas.

\begin{lemma}\label{lemma:Rn_terms}
Suppose one of Conditions (C) (i) or (ii) holds.
Then, we have
\begin{equation*}
    |R_n| \precsim \frac{\nu}{n}(C + |\log f_0(Z_1)| + |\log f_0(Z_n)|) +O_p\left(\frac{\log n}{n}\right),
\end{equation*}
for some constant $C$ that depends on $K_1,K_2, L$ if Condition (ii) holds and is a universal constant if Condition (i) holds.
\end{lemma}

\begin{proof}[Proof of Lemma \ref{lemma:Rn_terms}]
	Write
\begin{equation*}
	R_n = R_{1n} + R_{2n},
\end{equation*}
where
\begin{align*}
R_{1n} &:=	- \frac{1}{n} \sum_{j=0,\ldots,\frac{n-1}{\nu} - 1} 
	\sum^\nu_{l=1} \log \frac{f_0(Z_{j\nu+l+1})(Z_{(j+1)\nu+1} - Z_{j\nu+1})}{F_0(Z_{(j+1)\nu+1}) - F_0(Z_{j\nu+1})},\\
	R_{2n} &:= -\frac{1}{n}\log \frac{f_0(Z_1)}{f^H_{n,\nu}(Z_1)}.
\end{align*}

For $R_{1n}$, by the mean value theorem, there exists $Z^*_j$ lying between $Z_{j\nu+1}$ and $Z_{(j+1)\nu+1}$ such that $F_0(Z_{(j+1)\nu+1}) - F_0(Z_{j\nu+1}) = f_0(Z^*_j)(Z_{(j+1)\nu+1} - Z_{j\nu+1})$ and so
	\begin{equation*}
R_{1n}  = -\frac{1}{n}\sum_{j=0,\ldots, \frac{n-1}{\nu}-1} \sum^\nu_{l=1} \log \frac{f_0(Z_{j\nu+l+1})}{f_0(Z^*_j)}.
	\end{equation*}
 If Condition (i) holds and $f_0$ is monotone increasing, then
\begin{align}
    |n R_{1n}| & \leq \left|  \sum_{j=0,\ldots, \frac{n-1}{\nu}-1} \sum^\nu_{l=1} \log \frac{f_0(Z_{j\nu+1})}{f_0(Z_{(j+1)\nu+1})}\right| =\nu \left| \log f_0(Z_1) - \log f_0(Z_n)\right|\nonumber \\
    & \leq \nu (|\log f_0(Z_1)| + |\log f_0(Z_n)|). \label{eqnR_1n1}
\end{align}
We have the same inequality when $f_0$ is monotone decreasing.
If Condition (ii) holds, 
let 
\begin{equation*}
R^{(m)}_{1n} := \sum_{j=0,\ldots, \frac{n-1}{\nu}-1} \sum^\nu_{l=1} I^{(m)}_{n,jl} \log \frac{f_0(Z_{j\nu+l+1})}{f_0(Z^*_j)},
\end{equation*}
where 
\begin{align*}
    I^{(1)}_{n,jl} &:= \mathbbm{1}(Z_{j\nu+l+1}, Z^*_j \in [K_1,K_2]), \quad 
    I^{(2)}_{n,jl} := \mathbbm{1}( Z^*_j < K_2 < Z_{j\nu+l+1}),\\
    I^{(3)}_{n,jl} &:= \mathbbm{1}(  Z_{j\nu+l+1} < K_2 < Z^*_j), \quad 
    I^{(4)}_{n,jl}:=\mathbbm{1}( Z^*_j < K_1 < Z_{j\nu+l+1}),\\
    I^{(5)}_{n,jl} &:= \mathbbm{1}(  Z_{j\nu+l+1} < K_1 < Z^*_j), \quad 
    I^{(6)}_{n,jl} :=\mathbbm{1}(  Z_{j\nu+l+1}, Z^*_j > K_2),\\
    I^{(7)}_{n,jl} &:=\mathbbm{1}(  Z_{j\nu+l+1}, Z^*_j < K_1).
\end{align*}
Then,
\begin{equation*}
|nR_{1n}| \leq  \sum^7_{m=1}|R^{(m)}_{1n}|.
\end{equation*}
For $R^{(1)}_{1n}$, by the Lipschitz continuity of $\log f_0$ on $[K_1, K_2]$, 
\begin{align*}
    |R^{(1)}_{1n}|& \leq \sum_{j=0,\ldots, \frac{n-1}{\nu}-1} \sum^\nu_{l=1}  I^{(1)}_{n,jl} L \left| Z_{j\nu+l+1} - Z^*_j  \right|\\
    & \leq \sum_{j=0,\ldots, \frac{n-1}{\nu}-1} \sum^\nu_{l=1}  I^{(1)}_{n,jl} L(Z_{(j+1)\nu+1} - Z^*_j ) \\
    & \leq  \nu L(K_2 - K_1),
\end{align*}
where the second and third inequalities follow from the fact that $Z^*_j$ is between $Z_{j\nu+1}$ and $Z_{(j+1)\nu+1}$.

For $R^{(2)}_{1n}$, note that there will be at most one $j$, say, $\tilde{j}$, such that $\mathbbm{1}( Z^*_{\tilde{j}} < K_2 < Z_{\tilde{j}\nu+l+1})$ for some $l$. If $Z^*_{\tilde{j}} > K_1$, by telescoping and the triangle inequality,
\begin{align*}
    |R^{(2)}_{1n}|	 & \leq \sum^\nu_{l=1}\left( |\log f_0(Z_{\tilde{j} \nu+l+1}) - \log f_0(K_2)| + |\log f_0(K_2) - \log f_0(Z^*_j)|\right)\\
    &\leq \nu (|\log f_0(Z_n)| + |\log f_0(K_2)| + L(K_2 - K_1)),
\end{align*}
		where the last inequality follows from the monotonicity of $f_0$ beyond $K_2$ and the Lipschitz continuity of $\log f_0$ on $[K_1, K_2]$. 
		If $Z^*_{\tilde{j}} < K_1$, we have
		\begin{align*}
			|R^{(2)}_{1n}|	 & \leq \sum^\nu_{l=1}\bigg(  |\log f_0(Z_{\tilde{j} k+l+1}) - \log f_0(K_2)| + |\log f_0(K_2) - \log f_0(K_1)|\\
			& \quad  \quad  \quad  \quad + |\log f_0(K_1) - \log f_0(Z^*_j)|\bigg)\\
			&\leq \nu\bigg(|\log f_0(Z_n)| + |\log f_0(K_2)| + L(K_2-K_1) +  |\log f_0(K_1)| +  |\log f_0(Z_1)|\bigg).
		\end{align*}
		Thus, we always have
		\begin{align*}
			|R^{(2)}_{1n}| \leq \nu (C + |\log f_0(Z_1)|+|\log f_0(Z_n)|),
		\end{align*}
		where $C := L(K_2-K_1) + |\log f_0(K_2)| + |\log f_0(K_1)|$. It is straightforward to see the same bound holds for $|R^{(3)}_{1n}|, |R^{(4)}_{1n}|, |R^{(5)}_{1n}|$. The derivation of these bounds are similar to that for $|R^{(2)}_{1n}|$ and is therefore omitted.
		
		Now, we consider $R^{(6)}_n$. Let $j^*$ be the first $j$ such that $Z^*_j > K_2$. If $f_0$ is decreasing after $K_2$,
		\begin{align*}
			|R^{(6)}_{1n}|	&\leq \sum^\nu_{l=1} (\log f_0(K_2) - \log f_0(Z_{(j^*+1)\nu+1})) + \sum_{j=j^*+1,\ldots, \frac{n-1}{\nu}-1} \sum^\nu_{l=1}( \log f_0(Z_{j\nu+1}) - \log f_0(Z_{(j+1)\nu+1})) \\
			&\leq \nu(\log f_0(K_2) - \log f_0(Z_n) ).
		\end{align*}
		If $f_0$ is increasing after $K_2$, the bound becomes $\nu(\log f_0(Z_n) - \log f_0(K_2))$. In general, we have
		\begin{equation*}
			|R^{(6)}_{1n}| \leq \nu \left(|\log f_0(K_2)| + |\log f_0(Z_n)|\right).
		\end{equation*}
		Similarly, we also have
		\begin{equation*}
			|R^{(7)}_{1n}| \leq \nu\left(|\log f_0(K_1)| + |\log f_0(Z_1)|\right).
		\end{equation*}
		Combining the bounds for $|R^{(1)}_{1n}|,\ldots,|R^{(7)}_{1n}|$, we have
		\begin{equation}\label{eqnR_1n2}
			|nR_{1n}| \leq \nu(C_1 + 5|\log f_0(Z_1)| + 5|\log f_0(Z_n)|),
		\end{equation}
		where $C_1 := 5L(K_2-K_1) + 5|\log f_0(K_1)| + 5|\log f_0(K_2)|$. In view of (\ref{eqnR_1n1}) and (\ref{eqnR_1n2}), we have
  \begin{equation}\label{eq:R1n_order}
      |R_{1n}| \precsim \frac{v}{n}(C + |\log f_0(Z_1)| + |\log f_0(Z_n)|).
  \end{equation}

  For $R_{2n}$. Note that
\begin{align*}
	R_{2n} &= -\frac{1}{n} \log \frac{f_0(Z_1)(Z_{\nu+1} - Z_1)}{F_0(Z_{\nu+1})- F_0(Z_1)} - \frac{1}{n}\log \frac{F_0(Z_{\nu+1}) - F_0(Z_1))}{\nu/(n-1)}.
\end{align*}
By the mean value theorem, there exists $Z^*_1$ lying between $Z_{\nu+1}$ and $Z_1$ such that $F_0(Z_{\nu+1}) - F_0(Z_1) = f_0(Z^*_1)(Z_{\nu+1}-Z_1)$. Thus, using $(F_0(Z_1), F_0(Z_{\nu+1})) \stackrel{d}{=} (U_{(1)}, U_{(\nu+1)}) $ as in the proof of Lemma \ref{lemma:dist_same}, we have
\begin{equation}\label{eq:R2n1}
	R_{2n} \stackrel{d}{=}  -\frac{1}{n}\log \frac{f_0(Z_1)}{f_0(Z^*_1)} - \frac{1}{n} \log (U_{(\nu+1)} - U_{(1)}) + \frac{1}{n} \log \frac{n-1}{\nu}.
\end{equation}
Now, if Condition (i) holds, then 
\begin{equation}\label{eq:R2n2}
\left|	\log \frac{f_0(Z_1)}{f_0(Z^*_1)} \right| \leq \left|	\log \frac{f_0(Z_1)}{f_0(Z_n)} \right| \leq |\log f_0(Z_1)|+ |\log f_0(Z_n)|.
\end{equation}
If Condition (ii) holds, we have
\begin{align}
&	|\log f_0(Z^*_1)| \nonumber\\
 &= |(\log f_0(Z^*_1))\mathbbm{1}(Z^*_1 < K_1)  +
	(\log f_0(Z^*_1))\mathbbm{1}(K_1 < Z^*_1 < K_2) + (\log f_0(Z^*_1))\mathbbm{1}(Z^*_1 > K_2)| \nonumber \\
	& \leq (|\log f_0(Z_1)| + |\log f_0(K_1)|) + (L(K_2 - K_1) + |\log f_0(K_1)|) \nonumber \\
	& \quad \quad  + (|\log f_0(K_2)| + |\log f_0(Z_n)|)  \nonumber \\
	& =C_2 + |\log f_0(Z_1)| + |\log f_0(Z_n)|, \label{eq:R2n3}
\end{align}
where
\begin{equation*}
	C_2:= 2 |\log f_0(K_1)| + L(K_2-K_1) + |\log f_0(K_2)|.
\end{equation*}
Let $E_h$'s be i.i.d. standard exponential random variables the proof of Lemma \ref{lemma:dist_same}, then
\begin{align}
	|\log (U_{(\nu+1)} - U_{(1)})| & \leq |\log (U_{(2)} - U_{(1)})| \leq |\log E_2| + \left| \log \left( \sum^{n+1}_{h=1} E_h\right)\right|\nonumber\\
 &\leq |\log E_2| + \left| \log \left( \frac{1}{n+1}\sum^{n+1}_{h=1} E_h\right) \right|  + \log (n+1)= O_p(\log n).  \label{eq:R2n4}
\end{align}
In view of (\ref{eq:R2n1})-(\ref{eq:R2n4}), we have
\begin{equation}\label{eq:R2n_order}
\left|	R_{2n }\right| \leq \frac{2}{n} |\log f_0(Z_1)| + \frac{1}{n}|\log f_0(Z_n)| + O_p\left(\frac{\log n}{n}\right).
\end{equation}
The claim in the lemma follows in view of (\ref{eq:R1n_order}) and (\ref{eq:R2n_order}).
\end{proof}

The following lemma provides sufficient conditions for deriving the orders of $\log f_0(Z_1)$ and $\log f_0(Z_n)$.
\begin{lemma}\label{lemma:suff_cond_logn}
Let $X_1,X_2,\ldots$ be a sequence of random variables that are identically distributed with a common density $f$ but not necessarily independent of each other. 	
\begin{enumerate}
	\item [(i)]	Suppose that $\int_\mathbb{R} f(x)^{1- \alpha}\, dx  < \infty$ for some $\alpha > 0$. Then 
	\begin{equation}\label{eq:suff_cond_logn}
		\mathbb{P} \bigg(\limsup_{n \rightarrow \infty} \max_{i=1,\ldots,n} [ -  \{ \log f(X_i) \} (\log n)^{-1} ] \leq \frac{2}{\alpha} \bigg)=1.
	\end{equation}
	
	\item [(ii)] Suppose that $\int_\mathbb{R} f(x)^{1 + \alpha}\, dx  < \infty$ for some $\alpha > 0$. Then 
	\begin{equation}\label{eq:suff_cond_logn2}
		\mathbb{P} \bigg(\limsup_{n \rightarrow \infty} \max_{i=1,\ldots,n} [  \{\log f(X_i)\} (\log n)^{-1} ] \leq \frac{2}{\alpha} \bigg)=1.
	\end{equation}
\end{enumerate}	
\end{lemma}

\begin{proof}[Proof of Lemma \ref{lemma:suff_cond_logn}]
We only prove (\ref{eq:suff_cond_logn}) while (\ref{eq:suff_cond_logn2}) could be proven similarly and we omit it.	First, observe that
\begin{equation*}
	\mathbb{E}[ f(X_1)^{-\alpha}]  = \int_\mathbb{R} f(x)^{1- \alpha} \, dx < \infty.
\end{equation*}
By Markov's inequality, for any $m > 1$,
\begin{equation*}
	\mathbb{P}(f(X_n)^{-\alpha} > n^m) \leq \frac{ \mathbb{E}(f(X_n)^{-\alpha})}{n^m} = \frac{ \mathbb{E}(f(X_1)^{-\alpha})}{n^m}.
\end{equation*}
Hence,
\begin{equation*}
	\sum^\infty_{n=1} \mathbb{P}(f(X_n)^{-\alpha} > n^m) \leq \mathbb{E}(f(X_1)^{-\alpha}) \sum^\infty_{n=1} \frac{1}{n^m} < \infty.
\end{equation*}
By the first Borel-Cantelli lemma, with probability one, there exists an integer $N_1$ such that for all $n \geq N_1$,
\begin{equation*}
	f(X_n)^{-\alpha} \leq n^{m}.
\end{equation*}
There also exists another integer $N_2 \geq N_1$ such that for all $i=1,\ldots,N_2$,  $f(X_i)^{-\alpha} \leq N^{m}_2$. In other words, for all $n \geq N_2$ and  $i \leq n$, $f(X_i)^{-\alpha} \leq n^{m}$.	Hence, with probability one,
\begin{equation*}
	\limsup_{n \rightarrow \infty}  \max_{i=1,\ldots,n} [ - \{\log f(X_i)\} (\log n)^{-1}] \leq \frac{m}{\alpha}
\end{equation*}
and (\ref{eq:suff_cond_logn}) follows by taking $m = 2$.

\end{proof}

\begin{lemma}\label{lemma:tail_condition_imply_integral_alpha_finite}

    \begin{enumerate}[(a)]
        \item Suppose that Condition (A) holds, we have $\int_{\mathbb{R}}f_0(x)^{1- \alpha} dx < \infty$ for some $\alpha \in (0, 1)$.
       
        \item Suppose that Condition (B) holds, we have $\int_{\mathbb{R}}f_0(x)^{1+ \alpha} dx < \infty$ for some $\alpha \in (0, 1)$.
    \end{enumerate}
    
\end{lemma}

\begin{proof}[Proof of Lemma \ref{lemma:tail_condition_imply_integral_alpha_finite}]
\begin{enumerate}[(a)]
    \item Since $\gamma > 1$, there exists $\alpha \in (0, 1)$ such that $\gamma (1-\alpha) > 1$. Then,
\begin{align}\label{eq:tail_cond1}
    \int_{\mathbb{R}} f_0(x)^{1-\alpha} dx & \leq \int_{|x| \leq x_0} f_0(x)^{1-\alpha} dx + \int_{|x| > x_0} |x|^{-\gamma(1-\alpha)} dx.
\end{align}
Since $\gamma (1-\alpha) > 1$, $\int_{|x| > x_0} |x|^{-\gamma(1-\alpha)} dx < \infty$. For the first term on the RHS of (\ref{eq:tail_cond1}), we have
\begin{align*}
    \int_{|x| \leq x_0} f_0(x)^{1-\alpha} dx &= \int_{|x| \leq x_0, f_0(x) \leq 1} f_0(x)^{1-\alpha} dx +  \int_{|x| \leq x_0, f_0(x) > 1} f_0(x)^{1-\alpha}dx \\
    & \leq \int_{|x| \leq x_0} 1 dx + \int_{|x| \leq x_0} f_0(x) dx < \infty.
\end{align*}
\item Since $\gamma_2  > 0$, there exists $\alpha \in (0, 1)$ such that $(1+\alpha)(\gamma_2 - 1) > -1$.
Then
\begin{equation*}
    \int_{|x-a|< \delta} f_0(x)^{1+\alpha}dx \leq \int_{|x-a| < \delta} |x-a|^{(1+\alpha)(\gamma_2-1)} dx < \infty.
\end{equation*}
Now, $\sup_{|x-a| \geq \delta, f_0(x) \geq 1} f_0(x) < \infty$ and as $f_0$ is a density, the set such that $f_0(x) \geq 1$ must have finite Lebesgue measure. Thus,
\begin{align*}
    \int_{|x-a|\geq \delta} f_0(x)^{1+\alpha} dx &= 
    \int_{|x-a|\geq \delta, f_0(x) \geq 1} f_0(x)^{1+\alpha} dx + \int_{|x-a|\geq \delta, f_0(x) < 1} f_0(x)^{1+\alpha} dx \\
    &\leq \sup_{|x-a| \geq \delta, f_0(x) \geq 1} f_0(x)^{1+\alpha} \int_{f_0(x) \geq 1}dx  + \int_{|x-a| \geq \delta, f_0(x) < 1} f_0(x)dx  < \infty.
\end{align*}
\end{enumerate}

\end{proof}

\begin{proof}[Proof of Theorem \ref{theorem:Rn_order}]
Since one of Conditions (C) (i) or (ii) holds, by Lemma \ref{lemma:Rn_terms}, we have
\begin{equation*}
    |R_n| \precsim \frac{\nu}{n}(C + |\log f_0(Z_1)| + |\log f_0(Z_n)|) + O_p\left(\frac{\log n}{n}\right).
\end{equation*}
It remains to show that $\log f_0(Z_1) = O_p(\log n)$ and $\log f_0(Z_n) = O_p(\log n)$. Note that
\begin{equation*}
\log f_0(Z_1) \leq \max_{i=1,\ldots,n} f_0(X_i) \quad \text{and} \quad - \log f_0(Z_1) \leq \max_{i=1,\ldots,n} \{ -\log f_0(X_i) \}.
\end{equation*}
By Lemma \ref{lemma:tail_condition_imply_integral_alpha_finite}, since Conditions (A) and (B) hold, we have $\int f_0^{1\pm \alpha} < \infty$ for some $\alpha > 0$. Then, by Lemma \ref{lemma:suff_cond_logn}, we have with  probability one,
\begin{equation*}
\limsup_{n \rightarrow \infty} \left| \log f_0(Z_1) \right| (\log n)^{-1} \leq \frac{2}{\alpha}.
\end{equation*}
This implies that $\log f_0(Z_1) = O_p(\log n)$. 
The same argument also leads to $\log f_0(Z_n) = O_p(\log n)$. 
\end{proof}

\begin{proof}[Proof of Corollary \ref{col:lcb1b2}]
We shall show that the conditions in Theorem \ref{theorem:Rn_order} are satisfied so that the result follows. First, a log-concave density must be unimodal. It is either monotone over the support, or there exists $K_1 > \tau_f$ and $K_2 < \sigma_f$ such that $f_0$ is monotone on $(\tau_f, K_1]$ and $[K_2,\sigma_f)$. Furthermore, because $\log f_0$ is concave, $\log f_0$ is Lipschitz continuous on $[K_1, K_2]$. 

Note that any log-concave density has an exponential tail (see Lemma 1 in \cite{cule2010theoretical}), that is, there exist $a > 0 $ and $b \in \mathbb{R}$ such that $f_0(x) \leq e^{-a|x| + b}$ for all $x$. Thus, Conditions (A) and (B) are satisfied.
\end{proof}

\section{Additional results and proofs for Section \ref{subsect:k_log_density_ratio}: $k$-monotone densities}
\subsection{Upper bound and support of $k$-monotone MLE}
To prove Lemmas \ref{lemma:k_monotone_bounded_in_prob} and \ref{lemma:bounds_on_kmonotone_support}, we first recall that the MLE $\hat{f}_{n,k}$ is of the form: 
\begin{equation}\label{eq:k_monotone_form}
\hat{f}_{n,k}(x) = \sum^m_{j=1} \hat{w}_j \frac{k(\hat{a}_j - x)^{k-1}_+}{\hat{a}^k_j},
\end{equation}
where $m \in \mathbb{Z}^+$, and  $\{\hat{w}_1,\ldots,\hat{w}_m\}$ and $\{\hat{a}_1,\ldots,\hat{a}_m\}$ are respectively the family of weights and support points of the maximizing mixing distribution corresponding to $\hat{f}_{n,k}$. Here, $a_+$ denotes $\max(a, 0)$ for any $a \in \mathbb{R}$; see Lemma 2 in \cite{balabdaoui2010estimation}. Without loss of generality, we can assume that $\hat{a}_1 \leq \ldots \leq \hat{a}_m$. The likelihood function at $\hat{f}_{n,k}$ is:
\begin{equation*}
L_n(\hat{f}_{n,k}) = \prod^n_{i=1} \bigg\{ \sum^m_{j=1} \hat{w}_j \frac{k(\hat{a}_j - X_i)^{k-1}_+}{\hat{a}^k_j} \bigg\}.
\end{equation*} 

\begin{proof}[Proof of Lemma \ref{lemma:k_monotone_bounded_in_prob}]
We first show that $\hat{f}_{n,k}(0+) = O_p(1)$ if $f_0$ is bounded from above. Proposition 6 in \cite{gao2009rate} demonstrates that when $f_0 \in \mathcal{F}_k^B([0, A])$, we have $\hat{f}_{n,k}(0+) = O_p(1)$. A closer inspection of their proof reveals that the main step is to make use of the characterization of the NPMLE, which does not depend on the boundedness of the support of $f_0$. Therefore, the same proof remains valid in this case.

Now, we shall show that, in general, $\hat{f}_{n,k}(0+) \leq k Z^{-1}_1$. For $k=1$, $\hat{f}_{n,1}(0)Z_1 = \int^{Z_1}_0 \hat{f}_{n,k}(x)\, dx \leq \int^\infty_0 \hat{f}_n(x) \, dx = 1$. Therefore, $\hat{f}_n(0) \leq Z^{-1}_1$. For $k \geq 2$, we  use 
 the form in (\ref{eq:k_monotone_form}) and claim that $\hat{a}_1 > Z_1$. To see that, fix the values of $\hat{a}_2,\ldots,\hat{a}_{m}$ and $\hat{w}_1,\ldots,\hat{w}_m$. If $\hat{a}_1 \leq Z_1$, then $(\hat{a}_1 - X_i)^{k-1}_+ = 0$ for all $i$ while if $\hat{a}_1 > Z_1$, $(\hat{a} - X_i)_+^{k-1} > 0$ for some $i$. Hence, the maximum value of $L_n$ must be obtained when $\hat{a}_1 > Z_1$. Finally, from the proof of Proposition 6 in \cite{gao2009rate}, we have:
\begin{equation*}
	\hat{f}_{n,k}(0) \leq \frac{k \mathbb{F}_n(\hat{a}_1)}{\hat{a}_1} \leq \frac{k}{Z_1}.
\end{equation*}

\end{proof}

\begin{proof}[Proof of Lemma \ref{lemma:bounds_on_kmonotone_support}]
For $k=1$, it is well-known that $\sigma_{\hat{f}_{n,1}} = Z_n$. In the rest of this proof, we assume that $k\geq2$. In view of (\ref{eq:k_monotone_form}), $\hat{a}_m  = \sigma_{\hat{f}_{n,k}}$.
For fixed $\hat{a}_1,\ldots,\hat{a}_{m-1}$ and $\hat{w}_1,\ldots,\hat{w}_m$, we investigate how the change of $\hat{a}_m$ affects the value of the likelihood. By the definition of the MLE, we know that  $\hat{a}_m > Z_n$, since $k \geq2$ (see also Remark 2 in \cite{balabdaoui2010estimation}). Therefore, for each $i$, we shall consider the term
\begin{equation}\label{eq:k_monotone_ind}
	\hat{w}_m \frac{k (\hat{a}_m - X_i)^{k-1}_+}{\hat{a}^k_m}.
\end{equation}
Consider the function $g_i(y) = \frac{ (y- X_i)^{k-1}}{y^k}$ for $y > X_i$. Then,
\begin{equation*}
	g'_i(y) = \frac{(y- X_i)^{k-2}}{y^{k+1}}(k X_i - y) \text{ for } y > X_i,
\end{equation*} 
with which we can see that $g_i$ increases and then decreases in $y$, and attains its maximum at $k X_i$. Therefore, if $\hat{a}_m > k Z_n$, the term in (\ref{eq:k_monotone_ind}) becomes smaller for all $i=1,\ldots,n$ if $\hat{a}_m$ is set at a larger value. This shows that the optimality can be ensured only if $\hat{a}_m \leq k Z_n$.
\end{proof}

\subsection{General results on $S_{n,k}$ under $H_0$}\label{subsect:general_k_monotone_results}
In this subsection, we consider a more general version of Theorem \ref{thm:simple_k_monotone_logLRT_rate}; see Theorem \ref{thm:k_monotone_logLRT_rate}. 

Now, suppose that $\{a_n\}$ and $\{b_n\}$ are two sequences of real numbers such that $\log (a_n b_n) = o(n^{2k})$,
	\begin{equation}\label{eq:a_nb_n}
		\lim_{n \rightarrow \infty} \mathbb{P}(k Z_1^{-1} > a_n) = 0 \text{ and }		\lim_{n \rightarrow \infty} \mathbb{P}(k Z_n > b_n) = 0.
	\end{equation}
	Then, with a probability approaching $1$, we have $\hat{f}_n \in \mathcal{F}^{a_n}_k([0, b_n])$, which has a finite bracketing entropy for each $n$ (recall that \cite{gao2009rate} showed that $N_{[\cdot]}(\varepsilon, \mathcal{F}^B_k([0, A]), h) < \infty$ for each small enough $\varepsilon$ and any $A, B > 0$). The price of this relaxation is an extra factor, up to a difference in logarithm order, to be shown in Theorem \ref{thm:k_monotone_logLRT_rate}, that appears in the rate of convergence. If $f_0$ has a bounded support, we can simply take $b_n$ as $k \sigma_{f_0}$, which is finite; otherwise, given a sequence $d_n$ that converges to $0$, $b_n$ can be chosen to be $b_n = k F^{-1}_0((1-d_n)^{1/n})$. If $f_0$ is bounded from above, the condition $\lim_{n\rightarrow \infty}\mathbb{P}(k Z^{-1}_1 > a_n) = 0$ is not needed as we in that case have $\hat{f}_{n,k}(0+) = O_p(1)$; otherwise $a_n$ can be chosen to be $k/F_0^{-1}(d^{1/n}_n)$.

	\begin{theorem}\label{thm:k_monotone_logLRT_rate}
		Suppose that $f_0 \in \mathcal{F}_k$. Let $\{a_n\}$ and $\{b_n\}$ be two sequences that satisfy  (\ref{eq:a_nb_n}).
		\begin{enumerate}
			\item[(i)] If $f_0$ is unbounded at $0$ and has an unbounded support,
			\begin{equation*}
				S_{n,k} 
				= O_p\left(n^{-\frac{2k}{2k+1}} |\log (a_n b_n) |^{\frac{1}{2k+1}} \right).
			\end{equation*} 
			
			\item[(ii)] If $f_0$ is bounded from above and has an unbounded support,
			\begin{equation*}
				S_{n,k} 
				= O_p\left(n^{-\frac{2k}{2k+1}} |\log b_n |^{\frac{1}{2k+1}} \right).
			\end{equation*} 
			\item[(iii)] If $f_0$ is unbounded at $0$ and has a bounded support,
			\begin{equation*}
				S_{n,k} 
				= O_p\left(n^{-\frac{2k}{2k+1}} |\log a_n |^{\frac{1}{2k+1}} \right).
			\end{equation*} 
			
			\item[(iv)] If $f_0$ is bounded from above and has a bounded support, then $S_{n,k} = O_p\left(n^{-\frac{2k}{2k+1}}\right)$, as stated in (\ref{eq:K_n_order_existing_way}).
		\end{enumerate}
	\end{theorem}

\subsection{Proofs of Theorem \ref{thm:k_monotone_logLRT_rate},
Theorem \ref{thm:simple_k_monotone_logLRT_rate} and Corollary \ref{cor:simple_k_monotone_null}}

First, we shall introduce some notations. Given two functions $l$ and $u$, the bracket $\llbracket l, u \rrbracket$ is the set of all functions $g$ with $l \leq g \leq u$. Given a distance $\rho$, an $\varepsilon$-bracket is a bracket $\llbracket l, u \rrbracket$ for some $l$ and $u$ so that  $\rho(l, u) < \varepsilon$. For a class of functions $\mathcal{G}$, the bracketing number $N_{[\cdot]}(\varepsilon, \mathcal{G}, \rho)$ is the minimum number of $\varepsilon$-brackets needed to cover $\mathcal{G}$. The entropy with bracketing is the logarithm of the bracketing number $N_{[\cdot]}(\varepsilon, \mathcal{G}, \rho)$. The Hellinger distance between two densities $f$ and $g$ is defined as:
	\begin{align*}
		h(f,g) = \frac{1}{\sqrt{2}}\left[ \int_0^\infty \left\{ \sqrt{f(t)} - \sqrt{g(t)}\right\}^2\,dt\right]^{1/2}. 
	\end{align*} 
	Let $\tilde{J}_{[\cdot]}(\delta, \mathcal{G}, h) := \int^\delta_0 \sqrt{1 + \log N_{[\cdot]}(\varepsilon, \mathcal{G}, h)} d\varepsilon$. For more details of these definitions, see \cite{van1996weak}. 
	The notation $\lesssim$ is used to indicate the left side is bounded by the right side but up to a (universal) constant. 
	Finally, denote $\mathbb{F}_n$ to be the empirical distribution function induced by the random sample $X_1,\ldots,X_n$, and denote the empirical process by $\mathbb{G}_n(f) := \sqrt{n} \int f d (\mathbb{F}_n - F_0)$.

The method of proof of Theorem \ref{thm:k_monotone_logLRT_rate} follows the empirical process theory as described in, for example, \cite{van2000empirical}.
\begin{proof}[Proof of Theorem \ref{thm:k_monotone_logLRT_rate}]
We shall only prove (i). The other cases can be proven similarly. Fix $k \in \mathbb{N}$. For simplicity, we write $\hat{f}_n$ instead of $\hat{f}_{n,k}$. From the proof of Lemma 4.1 and Lemma 4.2 in \cite{van2000empirical}, we have
\begin{eqnarray*}
	0 & \leq  &- \frac{1}{4} S_{n, k} \leq \frac{1}{2} \int \log \frac{\hat{f}_{n} + f_0}{2f_0} d (\mathbb{F}_n - F_0) -  h^2 \left( \frac{\hat{f}_n + f_0}{2}, f_0 \right) \\
	& \leq & \frac{1}{2} \int \log \frac{\hat{f}_{n} + f_0}{2f_0} d (\mathbb{F}_n - F_0) - \frac{1}{16} h^2 \left(\hat{f}_n, f_0 \right).
\end{eqnarray*}
Let $m_{f}(t) := \log \frac{f(t) + f_0(t)}{2}$. Fix $M > 0$. Let $\delta_n := n^{-\frac{k}{2k+1}} | \log (a_n b_n)|^{\frac{1}{2(2k+1)}}$. Then, we have
\begin{eqnarray*}
	&& 	\mathbb{P} \left(  - \frac{\sqrt{n}}{2} S_{n, k} \geq \sqrt{n} 2^{2M} \delta^2_n \right) \\
	&\leq & \mathbb{P} \left( \mathbb{G}_n(m_{\hat{f}_n}) - \frac{\sqrt{n}}{8} h^2 ( \hat{f}_n, f_0 ) \geq \sqrt{n} 2^{2M} \delta^2_n, kZ^{-1}_1 \leq a_n, k Z_n \leq b_n \right)  \\
	&& \quad + \mathbb{P}(kZ^{-1}_1 > a_n) + \mathbb{P}(k Z_n > b_n),
\end{eqnarray*}
where the last two terms in the last display go to $0$ as $n$ goes to infinity by assumptions on $a_n$ and $b_n$. Denote
\begin{equation*}
    A_n	 := \mathbb{P} \left( \mathbb{G}_n(m_{\hat{f}_n}) - \frac{\sqrt{n}}{8} h^2 ( \hat{f}_n, f_0 ) \geq \sqrt{n} 2^{2M} \delta^2_n, k Z^{-1}_1 \leq a_n, k Z_n \leq b_n \right).
\end{equation*}
Using the peeling device, we obtain
\begin{eqnarray*}
	A_n	
	& \leq & \sum^{S_n}_{s=0} \mathbb{P}  \left( \sup_{f \in \mathcal{F}^{a_n}_k([0, b_n]): 2^{M+s} \delta_n \leq  h(f, f_0) < 2^{M + s + 1} \delta_n} \left( \mathbb{G}_n(m_f) - \frac{\sqrt{n}}{8} h^2(f, f_0) \right) \geq \sqrt{n} 2^{2M} \delta^2_n \right) \\
	&& \quad + \mathbb{P}  \left( \sup_{f \in \mathcal{F}^{a_n}_k([0, b_n]):   h(f, f_0) < 2^M \delta_n} \left( \mathbb{G}_n(m_f) - \frac{\sqrt{n}}{8} h^2(f, f_0) \right) \geq \sqrt{n} 2^{2M} \delta^2_n \right)\\
	& \leq & \sum^{S_n}_{s=0} \mathbb{P}  \left( \sup_{f \in \mathcal{F}^{a_n}_k([0, b_n]): 2^{M+s} \delta_n \leq  h(f, f_0) < 2^{M + s + 1} \delta_n}  \mathbb{G}_n(m_f)  \geq \sqrt{n} 2^{2M} (2^{2s} + 1) \delta^2_n \right) \\
	&& \quad + \mathbb{P}  \left( \sup_{f \in \mathcal{F}^{a_n}_k([0, b_n]):   h(f, f_0) < 2^M \delta_n}  \mathbb{G}_n(m_f) \geq \sqrt{n} 2^{2M} \delta^2_n \right)\\
	& \leq & \sum^{S_n}_{s=0} \mathbb{P}  \left( \sup_{f \in \mathcal{F}^{a_n}_k([0, b_n]): h(f, f_0) < 2^{M + s + 1} \delta_n}  \mathbb{G}_n(m_f)  \geq \sqrt{n} 2^{2M} 2^{2s}  \delta^2_n \right) \\
	&& \quad + \mathbb{P}  \left( \sup_{f \in \mathcal{F}^{a_n}_k([0, b_n]):   h(f, f_0) < 2^M \delta_n}  \mathbb{G}_n(m_f) \geq \sqrt{n} 2^{2M} \delta^2_n \right)\\
\end{eqnarray*}
where $S_n := \min\{s : 2^{M + s + 1}\delta_n \geq 1 \}$.
Define
\begin{equation*}
	\mathcal{M}_{n, \delta} := \left\{ m_f : f \in \mathcal{F}^{a_n}_k([0, b_n]), h(f, f_0) < \delta \right \}.
\end{equation*}
By Markov's inequality, 
\begin{eqnarray*}
	A_n & \leq & \sum^{S_n}_{s=0} \frac{ \mathbb{E}||\mathbb{G}_n||_{\mathcal{M}_{n, 2^{M + s+1} \delta_n}}}{ \sqrt{n} 2^{2M} 2^{2s}  \delta^2_n } + \frac{ \mathbb{E}||\mathbb{G}_n||_{\mathcal{M}_{n, 2^M \delta_n}}}{\sqrt{n} 2^{2M} \delta^2_n},
\end{eqnarray*}
where $\mathbb{E}||\mathbb{G}_n||_{\mathcal{M}_{n, \delta}} := \sup_{f \in \mathcal{M}_{n, \delta}} \mathbb{E}|\mathbb{G}_n(f)|$. By Theorem 3.4.4 in \cite{van1996weak}, 
\begin{equation*}
	\mathbb{E}||\mathbb{G}_n||_{\mathcal{M}_{n, \delta}} \lesssim \tilde{J}_{[\cdot]}(\delta, \mathcal{F}^{a_n}_k([0, b_n]), h) \left( 1+ \frac{\tilde{J}_{[\cdot]}(\delta, \mathcal{F}^{a_n}_k([0, b_n]), h)}{\delta^2 \sqrt{n}} \right).
\end{equation*}
From (\ref{eq:Gao_entropy}), we have
\begin{equation*}
	\log N_{[\cdot]}(\delta, \mathcal{F}^{a_n}_k([0, b_n]), h) \leq C_k | \log (a_n b_n)|^{\frac{1}{2k}} \delta^{-\frac{1}{k}}.
\end{equation*}
A direct calculation gives
\begin{align}
	\tilde{J}_{[\cdot]}\left(\delta, \mathcal{F}_k^{a_n}([0, b_n]), h\right)
	\lesssim C_k^{1/2}|\log (a_n b_n)|^{\frac{1}{4k}}\frac{\delta^{1-\frac{1}{2k}}}{1-\frac{1}{2k}}
	\lesssim 
	|\log (a_n b_n)|^{\frac{1}{4k}}\delta^{1-\frac{1}{2k}}
	\label{eq:JFH}
\end{align}
for all large enough $n$ as $C_k$ only depends on $k$. Define 
\begin{align*}
	\xi_{n}(\delta) := 
	|\log (a_n b_n)|^{\frac{1}{4k}} \delta^{1-\frac{1}{2k}} 
	\left(
	1+\frac{|\log (a_n b_n)|^{\frac{1}{4k}}\delta^{1-\frac{1}{2k}}
	}{\delta^2\sqrt{n}}
	\right). 
\end{align*}
Note that $\xi_{n}(\delta)/\delta$ is decreasing in $\delta$ for any fixed $n$. Thus, for any $j>0$, 
\begin{align*}
	\frac{\xi_{n}(2^j \delta_n)}{\delta^{2}_n} \leq \frac{2^j \xi_{n}(\delta_n)}{\delta^{2}_n} = 2^{j+1} \sqrt{n}.
\end{align*}
Hence,
\begin{eqnarray*}
	A_n & \lesssim & \sum^{S_n}_{s=0} \frac{ \xi_n(2^{M+s+1} \delta_n)}{ \sqrt{n}2^{2M}2^{2s} \delta^2_n} + \frac{\xi_n(2^M \delta_n)}{\sqrt{n} 2^{2M} \delta^2_n} 
	\leq  \sum^\infty_{s=0} \frac{2^{M+s+2} \sqrt{n}}{\sqrt{n}2^{2M}2^{2s}} + \frac{2^{M+1} \sqrt{n}}{\sqrt{n} 2^{2M}} = \frac{10}{2^M},
\end{eqnarray*}
which goes to $0$ as $M$ goes to $\infty$. This implies that
\begin{equation*}
    	S_{n,k} 
				= O_p\left(n^{-\frac{2k}{2k+1}} |\log (a_n b_n) |^{\frac{1}{2k+1}} \right).
\end{equation*}
%
\end{proof}

To prove Theorem \ref{thm:simple_k_monotone_logLRT_rate} and Corollary \ref{cor:simple_k_monotone_null}, we first establish the following lemma.

\begin{lemma}\label{lemma:anbn_logn_hall}
    Under Conditions (A) and (B), there exist $\{a_n\}$ and $\{b_n\}$ satisfying \eqref{eq:a_nb_n} and $\log (a_n b_n) = O(\log n)$.
\end{lemma}

\begin{proof}[Proof of Lemma \ref{lemma:anbn_logn_hall}]
Under Condition (A), for some $\gamma > 1, x_0 > 0$, we have for all $x \geq x_0$,
\begin{equation*}
    \mathbb{P}(X >x ) \leq \frac{x^{-\gamma+1}}{\gamma - 1}.
\end{equation*}
Let $b_n = n^{2/(\gamma-1)}$. Then for any $\varepsilon  > 0$,
\begin{equation*}
    n \mathbb{P}(X > \varepsilon b_n ) \leq \frac{n \varepsilon^{-\gamma+1}}{n^2(\gamma-1)}\rightarrow 0
\end{equation*}
as $n \rightarrow \infty$. By Theorem 3 in \cite{hall1979note}, we have $Z_n/b_n \stackrel{p}{\rightarrow} 0$. This implies that
\begin{equation*}
    \lim_{n \rightarrow \infty} \mathbb{P}(kZ_n > b_n) = 0.
\end{equation*}
Under Condition (B), for some $\delta > 0$, for all $x < \delta$, $F_0(x) \leq x^{\gamma_2}$ for some $\gamma_2 > 0$.  Let $a_n = n^{2/\gamma_2}$. 
\begin{equation*}
    n \mathbb{P}(X^{-1} > \varepsilon a_n) = n \mathbb{P}(X < (\varepsilon a_n)^{-1}) \leq \frac{n}{\varepsilon^{\gamma_2} n^2} \rightarrow 0.
\end{equation*}
Note that $Z_1$ is the maximum of $X^{-1}_1,\ldots,X^{-1}_n$. By Theorem 3 in \cite{hall1979note}, we have $Z_1/a_n \stackrel{p}{\rightarrow} 0$. This implies that
\begin{equation*}
    \lim_{n \rightarrow \infty} \mathbb{P}(kZ_1^{-1} > a_n) = 0.
\end{equation*}
Finally, we clearly have $\log (a_n b_n) = O(\log n)$.     
\end{proof}

\begin{proof}[Proof of Theorem \ref{thm:simple_k_monotone_logLRT_rate}]
This is a consequence of Lemma \ref{lemma:anbn_logn_hall} and Theorem \ref{thm:k_monotone_logLRT_rate}.
\end{proof}

\begin{proof}[Proof of Corollary \ref{cor:simple_k_monotone_null}]
We shall apply Theorem \ref{coro:main_result}  by verifying $\sqrt{n\nu} S_{n,k} = o_p(1)$ and $R_{n} = O_p(\frac{\nu \log n}{n})$. First, as $\nu = O(n^{1/3}/\log n)$, by Theorem \ref{thm:simple_k_monotone_logLRT_rate}, for any $k \in \mathbb{N}$,
\begin{align*}
    \sqrt{n\nu} S_{n,k} = O\left(n^{\frac{2}{3}} (\log n)^{-\frac{1}{2}}\right) O_p\left(n^{- \frac{2k}{2k+1}} (\log n)^{\frac{1}{2k+1}}\right) = o_p(1).
\end{align*}
As $f_0$ is decreasing, by Theorem \ref{theorem:Rn_order}, $R_{n} = O_p( \frac{\nu \log n}{n})$.
\end{proof}

\subsection{General results on $S_{n,k}$ under $H_1$}

\subsubsection{Fixed alternative}
We now turn to the study of the NPLRT under $H_1$. Theorem \ref{thm:k_monotone_H1} below provides mild sufficient conditions for the existence of a sequence $L_n \to \infty$ such that $\lim_{n \to \infty} \mathbb{P}(\sqrt{n\nu} S_{n,k} > L_n) = 1$ in the $k$-monotone density case. The proof of Theorem \ref{thm:k_monotone_H1} is almost identical to that of Theorem \ref{thm:k_monotone_H1_local_alt}, which is presented in Subsection \ref{subsect:proofs_app_kmonotone}; we therefore omit its proof.
\begin{theorem}\label{thm:k_monotone_H1}
		If $\inf_{f \in \mathcal{F}_k}h(f, f_0) > 0$ and $\log(a_n b_n) = o(n^{2k})$, then there exists some constant $c> 0$ such that
		\begin{equation*}
			\lim_{n\rightarrow \infty} \mathbb{P} ( S_{n,k} > c )  = 1. 
		\end{equation*}	
	\end{theorem}

As a result of Theorem \ref{thm:k_monotone_H1}, we have the following corollary where the conditions on $f_0$ are in terms of Conditions (A) and (B).
\begin{corollary}\label{thm:simple_k_monotone_H1}
    If $\inf_{f \in \mathcal{F}_k}h(f, f_0) > 0$ and $f_0$ satisfies Conditions (A) and (B), then there exists some constant $c> 0$ such that
    \begin{equation*}
        \lim_{n\rightarrow \infty} \mathbb{P} ( S_{n,k} > c )  = 1. 
    \end{equation*}	
\end{corollary}

\begin{proof}[Proof of Corollary \ref{thm:simple_k_monotone_H1}]
This is a consequence of Lemma \ref{lemma:anbn_logn_hall} and Theorem \ref{thm:k_monotone_H1}.
\end{proof}

\begin{corollary}[Consistency of the NPLRT for $k$-monotone densities]\label{cor:simple_k_monotone_alt}
    Consider $H_0: f_0 \in \mathcal{F}_k$ and $H_1: f_0 \notin \mathcal{F}_k$. Let $\nu = O(n^{1/3}(\log n)^{-1})$. Under $H_1$, suppose that  $\inf_{f \in \mathcal{F}_k}h(f, f_0) > 0$, $f_0$ satisfies Conditions (A), (B), and either (C)(i) or (C)(ii). Then, the NPLRT is consistent.
    
\end{corollary}

\begin{proof}[Proof of Corollary \ref{cor:simple_k_monotone_alt}]
By Corollary \ref{thm:simple_k_monotone_H1}, $\lim_{n \rightarrow \infty} \mathbb{P}(\sqrt{n \nu} S_{n} > \sqrt{n \nu} c) = 1$ for some $c > 0$.  By Theorem \ref{theorem:Rn_order}, $R_n = O_p( \frac{\nu \log n}{n})$.  The result then follows from Theorem \ref{corollary:consistency}.
\end{proof}

\subsubsection{Local alternatives}
Now consider a sequence of local alternatives where the underlying density $f_n \notin \mathcal{F}_k$ satisfies  
$\varepsilon_{kn} := \inf_{f \in \mathcal{F}_k} h(f, f_n) > 0$, and assume that $\varepsilon_{kn} \to 0$ as $n \to \infty$.  
Let $Z_{n1}$ and $Z_{nn}$ denote the minimum and maximum order statistics of a random sample $X_{n1}, \ldots, X_{nn}$ drawn from $f_n$.  
Let $\{a_n\}$ and $\{b_n\}$ be two sequences of real numbers such that
    \begin{equation}\label{eq:a_nb_n_local}
    \lim_{n \to \infty} \mathbb{P}(k Z_{n1}^{-1} > a_{n}) = 0 \quad \text{and} \quad \lim_{n \to \infty} \mathbb{P}(k Z_{nn} > b_{n}) = 0.
\end{equation}
With some abuse of notation, we continue to use $S_{n,k}$ and $R_n$ when $f_0$ is replaced by $f_n$ and when $\hat{f}_{n,k}$ is replaced by the NPMLE based on $X_{n1}, \ldots, X_{nn}$. 
The following theorem generalizes Corollary \ref{thm:simple_k_monotone_H1}, which corresponds to the special case where $f_n = f_0 \notin \mathcal{F}_k$ and $\inf_{f \in \mathcal{F}_k} h(f, f_0) > 0$.

\begin{theorem}\label{thm:k_monotone_H1_local_alt}
Suppose that $\varepsilon_{kn}^{-1-\frac{1}{2k}} = o(n^{1/2}(\log n)^{-1})$ and $\log(a_n b_n) = O(\log n)$, where $\{a_n\}$ and $\{b_n\}$ satisfy \eqref{eq:a_nb_n_local}, then there exists some constant $c> 0$ such that
\begin{equation*}
    \lim_{n\rightarrow \infty} \mathbb{P} ( S_{n,k} > c \varepsilon_{kn}^2 )  = 1. 
\end{equation*}	
\end{theorem}



\subsection{Proofs of Theorems \ref{thm:k_monotone_H1_local_alt}  and \ref{coro:k_monotone_H1_local_alt_rate}}\label{subsect:proofs_app_kmonotone}

\begin{proof}[Proof of Theorem \ref{thm:k_monotone_H1_local_alt}]
Let $c_j$, $j=1,\ldots,4$, be the positive constants in Theorem 1 in \cite{wong1995probability}. First, note that
\begin{eqnarray*}
	&&	\mathbb{P}\bigg( \prod^n_{i=1} \frac{\hat{f}_{n,k}(X_i)}{f_0(X_i)} \geq e^{-c_1 n\varepsilon_{kn}^2} \bigg)\\
	& \leq &
	\mathbb{P}\bigg( \prod^n_{i=1} \frac{\hat{f}_{n,k}(X_i)}{f_0(X_i)} \geq e^{-c_1 n\varepsilon_{kn}^2}, \hat{f}_{n,k} \in \mathcal{F}^{a_n}_k[0,b_n] \bigg) + \mathbb{P}( \hat{f}_{n,k} \notin \mathcal{F}^{a_n}_k[0,b_n] ) \\
	&\leq & \mathbb{P} \bigg( \sup_{f \in \mathcal{F}^{a_n}_k[0, b_n]} \prod^n_{i=1} \frac{f(X_i)}{f_0(X_i)} \geq e^{-c_1n\varepsilon_{kn}^2} \bigg) + \mathbb{P}( \hat{f}_{n,k} \notin \mathcal{F}^{a_n}_k[0,b_n] ). 
\end{eqnarray*}
To verify (3.1) in \cite{wong1995probability}, we have, using Theorem 3 in \cite{gao2009rate},
\begin{eqnarray*}
	\int^{\sqrt{2} \varepsilon_{kn}}_{\frac{\varepsilon_{kn}^2}{2^8} } \sqrt{  \log N_{[\cdot]} \bigg( \frac{u}{c_3}, \mathcal{F}^{a_n}_k([0,b_n]), h \bigg) }du 
	& \leq & \int^{\sqrt{2}\varepsilon_{kn}}_0 C^{\frac{1}{2}}_k | \log (a_n b_n)|^{\frac{1}{4k}} \bigg( \frac{u}{c_3} \bigg)^{-\frac{1}{2k}} du \\
	&\leq & C^{\frac{1}{2}}_k  c_3^{\frac{1}{2k}} 2^{\frac{1}{2} - \frac{1}{4k}} \varepsilon_{kn}^{1-\frac{1}{2k}} | \log (a_n b_n)|^{\frac{1}{4k}}.
\end{eqnarray*}
As $\log(a_n b_n) = O(\log n)$, the right hand side of the last inequality is smaller than $c_4 n^{1/2} \varepsilon_{kn}^2$ for all large enough $n$ as $\varepsilon_{kn}^{-1 - \frac{1}{2k}} = o(n^{1/2}(\log n)^{-1})$. Therefore, by Theorem 1 in \cite{wong1995probability}, for all large enough $n$,  
\begin{equation*}
	\mathbb{P} \bigg( \sup_{f \in \mathcal{F}^{a_n}_k[0, b_n]} \prod^n_{i=1} \frac{f(X_i)}{f_0(X_i)} \geq e^{-c_1n\varepsilon_{kn}^2} \bigg) \leq 4 \exp(-c_2 n \varepsilon_{kn}^2).
\end{equation*}
Hence,
\begin{eqnarray*}
	\limsup_{n \rightarrow \infty}\mathbb{P}\bigg( \prod^n_{i=1} \frac{\hat{f}_{n,k}(X_i)}{f_0(X_i)} \geq  e^{-c_1 n\varepsilon_{kn}^2} \bigg) &\leq& \limsup_{n \rightarrow \infty} 4 \exp(-c_2 n \varepsilon_{kn}^2) + \limsup_{n\rightarrow \infty} \mathbb{P}( \hat{f}_{n,k} \notin \mathcal{F}^{a_n}_k([0,b_n]) ) \\
	&=& \limsup_{n\rightarrow \infty} \mathbb{P}( \hat{f}_{n,k} \notin \mathcal{F}^{a_n}_k([0,b_n]) ) = 0,
\end{eqnarray*}
where the first equality follows as $\varepsilon_{kn}^{-1 - \frac{1}{2k}} = o(n^{1/2}(\log n)^{-1})$ implies that $n \varepsilon_{kn}^2 \rightarrow \infty$, and the second equality follows as $\hat{f}_{n,k} \in \mathcal{F}^{a_n}_k([0, b_n])$ with a probability going to $1$.
Finally, this implies that 
\begin{equation*}
	\lim_{n\rightarrow \infty} \mathbb{P} \bigg( \frac{1}{n} \sum^n_{i=1} \log \frac{f_0(X_i)}{\hat{f}_{n,k}(X_i)} > c_1 \varepsilon_{kn}^2 \bigg) = 
	\lim_{n\rightarrow \infty} \mathbb{P}\bigg( \prod^n_{i=1} \frac{\hat{f}_{n,k}(X_i)}{f_0(X_i)} <  e^{-c_1 n\varepsilon_{kn}^2} \bigg)  = 1. 
\end{equation*}

\end{proof}

\begin{proof}[Proof of Theorem \ref{coro:k_monotone_H1_local_alt_rate}]
    Recall that in Theorem \ref{corollary:consistency}, for the consistency of the test, we require the existence of a sequence $L_n \to \infty$ such that
\begin{equation*}
    \lim_{n \to \infty} \mathbb{P}( \sqrt{n \nu} \, S_{n,k} > L_n ) = 1.
\end{equation*}
By Theorem \ref{thm:k_monotone_H1_local_alt}, we obtain
\begin{equation*}
    \lim_{n \to \infty} \mathbb{P}( \sqrt{n \nu} \, S_{n,k} > c \sqrt{n \nu} \, \varepsilon_{kn}^2 ) = 1.
\end{equation*}
Setting $L_n = \sqrt{n \nu} \, \varepsilon_{kn}^2$, we then require $\varepsilon_{kn}^{-1} = o((n \nu)^{1/4})$ for the power to go to $1$. To determine the fastest rate at which $\varepsilon_{kn}$ can converge to zero, we must choose the largest possible order of $\nu$. Thus, if we take $\nu \asymp n^{1/3} (\log n)^{-1}$, we require $\varepsilon_{kn}^{-1} = o(n^{1/3}(\log n)^{-1/4})$. For example, $\varepsilon_{kn} = n^{-1/3} \log n$ satisfies this condition as well as the condition to apply Theorem \ref{thm:k_monotone_H1_local_alt} that
\begin{equation*}
    \varepsilon_{kn}^{-1-\frac{1}{2k}} = o(n^{1/2} (\log n)^{-1}), \quad \text{for each } k \in \mathbb{N}.
\end{equation*}
In other words, we allow $\varepsilon_{kn}$ to be of order $n^{-1/3}$ up to a logarithmic factor. 
\end{proof}

\section{Additional results and proofs for Section \ref{subsect:cm_log_density_ratio}: completely monotone densities}

\subsection{General results on $S_{n,\infty}$ under $H_0$}\label{subsect:cm_general}
In this subsection, we consider a more general version of Lemma \ref{lem:new_rate_cm}; see Lemma \ref{lem:rate_cm}.
    Let $\{c_{1n}\}$ and $\{c_{2n}\}$ be sequences such that $\log (c_{1n} c_{2n}) = o(n^2)$,
    \begin{equation}\label{eq:CM_cn1cn2}
        \lim_{n \rightarrow \infty} \mathbb{P}(Z^{-1}_1 > c_{1n}) =0 \text{ and } 	\lim_{n \rightarrow \infty} \mathbb{P}(Z^{1+ \beta}_n > c_{2n}) = 0,
    \end{equation}
    for some $\beta > 0$. 
    Denote 
    \begin{equation}\label{eq:def_Kn}
        S_{n, \infty} := - \frac{1}{n} \sum^n_{i=1} \log \frac{\hat{f}_{n,\infty}(X_i)}{f_0(X_i)}.
    \end{equation}
    \begin{lemma}\label{lem:rate_cm}
        Suppose that $f_0\in \mathcal{F}_\infty$, $\{c_{1n}\}$ and $\{c_{2n}\}$ satisfy (\ref{eq:CM_cn1cn2}). Then,
        \begin{equation}\label{eq:CM_LRT_rate}
            S_{n,\infty} = O_p\left( n^{-\frac{2}{3}} |\log (c_{1n} c_{2n}) |^{\frac{1}{3}}\right).
        \end{equation}
    \end{lemma}
    
    The slow rate in Lemma \ref{lem:rate_cm} is a consequence of our method of proof but is sufficient for the application of the NPLRT under mild conditions on $c_{n1}$ and $c_{n2}$.


\begin{proof}[Proof of Lemma \ref{lem:rate_cm}]
For simplicity, we write $\hat{f}_n$ instead of $\hat{f}_{n,\infty}$. 	First, from \cite{Jewell1982mixtures}, we know that the MLE is of the form
\begin{equation*}
	\hat{f}_n(t) = \sum^{r_n}_{i=1} \hat{\lambda}_i \hat{p}_i e^{-\hat{\lambda}_i t},
\end{equation*}
for some $r_n \leq n$, with $\sum^{r_n}_{i=1}\hat{p}_i = 1$ and $\hat{\lambda}_i \in [Z_n^{-1}, Z_1^{-1}]$ for each $i=1,\ldots,r_n$.  Without loss of generality, we  assume that $\hat{\lambda}_1 \leq \ldots \leq \hat{\lambda}_{r_n}$. We shall first obtain an upper bound for $\hat{f}_n$ for a large $t$. Note that
\begin{equation*}
	\hat{f}_n(t) \leq \sum^{r_n}_{i=1} \hat{p}_i \max_{j=1,\ldots,r_n} \hat{\lambda}_j e^{-\hat{\lambda}_j t} = \max_{j=1,\ldots,r_n} \hat{\lambda}_j e^{-\hat{\lambda}_j t},
\end{equation*}
since $\sum^{r_n}_{i=1} \hat{p}_i = 1$. For any fixed $t$, the function $\lambda \mapsto \lambda e^{-\lambda t}$ is unimodal and achieves its maximum when $\lambda = 1/t$. Thus, if $t > Z_n$ so that $1/t < 1/Z_n \leq \hat{\lambda}_1$, we have
\begin{equation*}
	\hat{f}_n(t) \leq \hat{\lambda}_1 e^{- \hat{\lambda}_1 t} \leq \frac{1}{Z_n} e^{- \frac{t}{Z_n}}.
\end{equation*}
Since $f_0$ has an unbounded support, $Z_n \stackrel{a.s.}{\rightarrow} \infty$. Thus, with a  probability approaching $1$,  we have $c_{2n} \geq Z^{1+\beta}_n > Z_n$ and for all $t > c_{2n}$, 
\begin{equation}\label{eq:CM_tail_bound}
	\hat{f}_n(t) \leq \frac{1}{Z_n} e^{- \frac{t}{Z_n}} \leq \frac{1}{t^2} \quad \text{ as } n \rightarrow \infty.
\end{equation}
We shall also obtain a rough upper bound for $\hat{f}_n(0)$. Observe that
\begin{equation*}
	\hat{f}_n(0) = \sum^{r_n}_{i=1} \hat{\lambda}_i \hat{p}_i \leq \sum^{r_n}_{i=1} \hat{p}_i \max_{j=1,\ldots,r_n} \hat{\lambda}_j \leq \max_{j=1,\ldots,r_n} \hat{\lambda}_j \leq Z^{-1}_1.
\end{equation*}
To obtain a rate of convergence of $\hat{f}_n$, we shall find an upper bound for the bracketing entropy  of a class of functions where the MLE will be in with a probability approaching $1$.  To this end, write
\begin{equation*}
	\hat{f}_n(t) = \hat{f}_n(t) I(t \leq c_{2n}) + \hat{f}_n(t) I(t > c_{2n}).
\end{equation*}
Define 
\begin{equation*}
	\tilde{\mathcal{F}}_n := \tilde{\mathcal{F}}^{c_{1n}}_1([0,c_{2n}]) + \mathcal{G}_{n},
\end{equation*}
where
\begin{equation*}
	\tilde{\mathcal{F}}^{c_{1n}}_1([0, c_{2n}]) := \left\{ f :[0, c_{2n}] \rightarrow \mathbb{R}^+ : \text{$f$ is decreasing}, \int^{c_{2n}}_0 f(x) \, dx \leq 1,  f(0) \leq c_{1n} \right \}
\end{equation*}
and 
\begin{equation*}
	\mathcal{G}_n:= \left\{ f : (c_{2n}, \infty) \rightarrow [0, 1] : \text{$f$ is decreasing, } f(t) \leq 1/t^2  \right\}.
\end{equation*}
Because of (\ref{eq:CM_cn1cn2}), $\lim_{n \rightarrow \infty} \mathbb{P}(\hat{f}_n(t) I(t \leq c_{2n}) \in \tilde{\mathcal{F}}^{c_{1n}}_1([0, c_{2n}]) = 1$. In view of (\ref{eq:CM_tail_bound}), we have
$\lim_{n \rightarrow \infty}	\mathbb{P}(\hat{f}_n(t) I(t > c_{2n}) \in \mathcal{G}_n) = 1$.
Therefore, 
\begin{equation}\label{eq:CM_fhat_in_tildeFn}
	\lim_{n \rightarrow \infty}	\mathbb{P}(\hat{f}_n \in \tilde{\mathcal{F}}_n) = 1.	
\end{equation}
We shall now show that $\tilde{\mathcal{F}}_n$ has a finite bracketing entropy with respect to the Hellinger distance. First, from the proof of Theorem 3 in \cite{gao2009rate}, we know the result of that theorem is also valid for $\tilde{\mathcal{F}}^{c_{1n}}_1([0, c_{2n}])$, where for any function $f$ in this class such that $\int^{c_{2n}}_0 f(x)\, dx \leq 1$ rather than just $\int^{c_{2n}}_0 f(x) \, dx = 1$. Thus,
\begin{equation}\label{eq:CM_entropy_1}
	\log N_{[\cdot]}(\varepsilon, \tilde{\mathcal{F}}^{c_{1n}}_1([0, c_{2n}]), h) \leq C | \log (c_{1n} c_{2n})|^{1/2} \varepsilon^{-1},
\end{equation}
where $C$ is a universal constant. Then, note that if $f \in \mathcal{G}_n$, then $f^{1/2} \in \mathcal{G}^{1/2}_n$, where
\begin{equation*}
	\mathcal{G}^{1/2}_n:= \left\{ f : (c_{2n}, \infty) \rightarrow [0, 1] : \text{$f$ is decreasing, } f(t) \leq 1/t  \right\}.
\end{equation*}
As $\int^\infty_{c_{2n}} 1/x^{2(1-\gamma)} \, dx < \infty$ for some $\gamma \in (0, 1)$, we can apply the idea of Lemma 7.10 in \cite{van2000empirical} to obtain that
\begin{equation}\label{eq:CM_entropy_2}
	\log N_{[\cdot]}(\varepsilon, 	\mathcal{G}^{1/2}_n, ||\cdot||_2) \leq A \varepsilon^{-1},
\end{equation}
where $A$ is a universal constant and $||\cdot||_2$ is the $L^2$-norm with respect to the Lebesgue measure on $[0, \infty)$. Now, simply note that
\begin{equation}\label{eq:CM_entropy_3}
	\log N_{[\cdot]}(\varepsilon, 	\mathcal{G}^{1/2}_n, ||\cdot||_2) = \log N_{[\cdot]}(\varepsilon, \mathcal{G}_n, h).
\end{equation}
Combining (\ref{eq:CM_entropy_1}) to (\ref{eq:CM_entropy_3}), we obtain that
\begin{equation}\label{eq:CM_tildeFn_entropy}
	\log N_{[\cdot]}(\varepsilon, 	\tilde{\mathcal{F}}_n, h) \leq  \{C |\log (c_{1n}c_{2n})|^{1/2} + A\} \varepsilon^{-1} \lesssim |\log (c_{1n}c_{2n})|^{1/2} \varepsilon^{-1},
\end{equation}
for all large $n$. With (\ref{eq:CM_fhat_in_tildeFn}) and (\ref{eq:CM_tildeFn_entropy}), we can obtain (\ref{eq:CM_LRT_rate}) as in the proof of Theorem \ref{thm:k_monotone_logLRT_rate} and we omit the details.	
\end{proof}

%

\subsection{Proofs for Subsection \ref{subsect:cm_log_density_ratio}}
The proof of the following lemma is essentially the same as that of Lemma \ref{lemma:anbn_logn_hall} and is therefore omitted.
\begin{lemma}\label{lemma:cn1cn2_logn_hall}
    Under Conditions (A) and (B), there exist $\{c_{n1}\}$ and $\{c_{n2}\}$ satisfying \eqref{eq:CM_cn1cn2} and $\log (c_{n1} c_{n2}) = O(\log n)$.
\end{lemma}

With Lemma \ref{lemma:cn1cn2_logn_hall}, we can establish Lemma \ref{lem:new_rate_cm} in view of Lemma \ref{lem:rate_cm}. The proof of Corollary \ref{thm:simple_CM_null} is similar to that of Corollary \ref{cor:simple_k_monotone_null} for $k$-monotone densities and is therefore omitted. 

\subsection{General results on $S_{n,\infty}$ under $H_1$}
\subsubsection{Fixed alternatives}
Under a fixed alternative, Theorem \ref{thm:C_monotone_H1} and Corollaries \ref{thm:simple_C_monotone_H1} and \ref{cor:simple_C_monotone_alt} for completely monotone densities parallel Theorem \ref{thm:k_monotone_H1} and Corollaries \ref{thm:simple_k_monotone_H1} and \ref{cor:simple_k_monotone_alt} for $k$-monotone densities, respectively. Their proofs are thus omitted.
	\begin{theorem}\label{thm:C_monotone_H1}
			If $\inf_{f \in \mathcal{F}_\infty}h(f, f_0) > 0$ and $\log (c_{1n} c_{2n}) = o(n^{2})$, then there exists some constant $c> 0$ such that
			\begin{equation*}
				\lim_{n\rightarrow \infty} \mathbb{P} ( S_{n,\infty} > c )  = 1. 
			\end{equation*}	
		\end{theorem}
    	\begin{corollary}\label{thm:simple_C_monotone_H1}
			If $\inf_{f \in \mathcal{F}_\infty}h(f, f_0) > 0$, $f_0$ satisfies Conditions (A) and (B), then there exists some constant $c> 0$ such that
			\begin{equation*}
				\lim_{n\rightarrow \infty} \mathbb{P} ( S_{n,\infty} > c )  = 1. 
			\end{equation*}	
		\end{corollary}

		\begin{corollary}[Consistency of the NPLRT for completely monotone densities]\label{cor:simple_C_monotone_alt}
			Consider $H_0: f_0 \in \mathcal{F}_\infty$ and $H_1: f_0 \notin \mathcal{F}_\infty$. Let $\nu = O(n^{1/3}(\log n)^{-1})$. Under $H_1$, suppose that  $\inf_{f \in \mathcal{F}_\infty}h(f, f_0) > 0$, $f_0$ satisfies Conditions (A), (B), and either (C)(i) or (C)(ii). Then, the NPLRT is consistent.
		\end{corollary}



\subsubsection{Local alternatives}
Consider a sequence of local alternatives where the underlying density $f_n \notin \mathcal{F}_\infty$ satisfies  
$\varepsilon_{\infty,n} := \inf_{f \in \mathcal{F}_\infty} h(f, f_n) > 0$, and assume that $\varepsilon_{n} \to 0$ as $n \to \infty$.  
Let $Z_{n1}$ and $Z_{nn}$ denote the minimum and maximum order statistics of a random sample $X_{n1}, \ldots, X_{nn}$ drawn from $f_n$.  
Let $\{c_{1n}\}$ and $\{c_{2n}\}$ be two sequences of real numbers such that
\begin{equation}\label{eq:a_nb_n_local_cm}
    \lim_{n \to \infty} \mathbb{P}( Z_{n1}^{-1} > c_{1n}) = 0 \quad \text{and} \quad \lim_{n \to \infty} \mathbb{P}(Z_{nn}^{1+\beta} > c_{2n}) = 0,
\end{equation}
for some $\beta > 0$. 

\begin{theorem}\label{coro:c_monotone_H1_local_alt_rate}
    Suppose that $\nu \asymp n^{1/3}(\log n)^{-1}$, $\log (c_{1n} c_{2n}) = O(\log n)$, where $\{c_{1n}\}$ and $\{c_{2n}\}$ satisfy \eqref{eq:a_nb_n_local_cm}, and $R_n = O_p(\frac{\nu \log n}{n})$. If $\varepsilon_{\infty, n} = \Omega(n^{-1/3} \log n)$, then for any $0 \leq c < \infty$,
    \begin{equation*}
\lim_{n \rightarrow \infty} \mathbb{P}     \left(
\sqrt{ \frac{n\nu}{\nu^2\psi_1(\nu)-\nu}} \left( T_n -  \log \nu +  \psi(\nu) \right) > c\right) = 1.
\end{equation*}
\end{theorem}
The argument used to obtain this result is the same as that for the $k$-monotone case, and we therefore omit the details.

\section{Additional results and proofs for Section \ref{subsect:lc_log_density_ratio}: log-concave densities}
Note that the derivation of Corollary \ref{cor:null_dist_log_concave} is explained in the paragraph preceding it. 

\subsection{Fixed alternatives}
	In \cite{doss2016global}, a bracketing entropy bound is derived for the following subset of $\mathcal{F}_{lc}$:
		\begin{equation}\label{eq:def_FMlc}
			\mathcal{F}^M_{lc} :=  \left \{ f \in \mathcal{F}_{lc}: \sup_{x \in \mathbb{R}} f(x) \leq M, \frac{1}{M} \leq  \inf_{x \in [-1,1]} f(x) \right\}, 
		\end{equation}
		where $0 < M < \infty$. Using this result, along with the deviation inequality in Theorem 1 in  \cite{wong1995probability}, we establish the following theorem.
		\begin{theorem}\label{thm:log_concave_H1}
			If $\inf_{f \in \mathcal{F}_{lc}}h(f, f_0) > 0$ and 
   $\int_{\mathbb{R}} |x| f_0(x)dx < \infty$, then there exists some finite constant $c > 0$ such that
			\begin{equation*}
				\lim_{n\rightarrow \infty} \mathbb{P} ( S_{n,lc} > c)  = 1. 
			\end{equation*}
		\end{theorem}

\begin{proof}[Proof of Theorem \ref{thm:log_concave_H1}]
The proof is similar to the one in Theorem \ref{thm:k_monotone_H1_local_alt}, where we apply the large deviation inequality in Theorem 1 in \cite{wong1995probability}. To this end, it suffices to verify (3.1) in \cite{wong1995probability} and our claim then follows as a result. 
Without loss of generality, assume that the interval $[-1,1]$ is strictly contained in the support of $f_0$. 
Because $\int_{\mathbb{R}} |x|f_0(x)\, dx < \infty$, by Lemma 3 in \cite{cule2010MLE_log_concave}, $\hat{f}_{n,lc} \in \mathcal{F}^M_{lc}$ for some $M > 0$ with a probability approaching $1$ as $n \rightarrow \infty$.  By Theorem 3.1 in \cite{doss2016global}, $\log N_{[\cdot]}(u/c_3, \mathcal{F}^M_{lc}, h) \precsim u^{-1/2}$. Hence,
\begin{equation*}
\int^{\sqrt{2}\varepsilon}_{\frac{\varepsilon^2}{2^8}} \sqrt{ \log N_{[\cdot]} \bigg( \frac{u}{c_3}, \mathcal{F}^M_{lc}, h \bigg)} du \precsim \int^{\sqrt{2}{\varepsilon}}_0 u^{-1/4} du \precsim \varepsilon^{3/4}.
\end{equation*}
Thus, the right hand side of the last inequality is smaller than $c_4 n^{1/2} \varepsilon^2$ for all large enough $n$, and so (3.1) in \cite{wong1995probability} is satisfied.
\end{proof}

		\begin{corollary}[Consistency of NPLRT for log-concave densities]\label{cor:log_convcave_consistency}
			Consider $H_0: f_0 \in \mathcal{F}_{lc}$ and $H_1: f_0 \notin \mathcal{F}_{lc}$. 
 Let $\nu = O(n^{1/3}(\log n)^{-1})$. Under $H_1$, suppose that  $\inf_{f \in \mathcal{F}_{lc}}h(f, f_0) > 0$, $f_0$ satisfies    $\int_{\mathbb{R}} |x| f_0(x)dx < \infty$, Conditions (A), (B), and either (C)(i) or (C)(ii). Then, the NPLRT is consistent.

		\end{corollary}

\begin{proof}[Proof of Corollary \ref{cor:log_convcave_consistency}]
By Theorem \ref{thm:log_concave_H1}, there exists $c> 0$ such that $\mathbb{P}(\sqrt{n\nu} S_{n,lc} > \sqrt{n \nu} c) = 1$. Under Conditions (A), (B), and either (C)(i) or (C)(ii), $R_n = O_p( \frac{\nu \log n}{n})$ by Theorem \ref{theorem:Rn_order}. The result then follows from Theorem \ref{corollary:consistency}.

\end{proof}

        \subsection{Local alternatives}
  The following theorem extends Theorem \ref{thm:log_concave_H1} to sequences of local alternatives.
        \begin{theorem}\label{thm:local_alt_lc_SN}
            Let $f_{n} \notin \mathcal{F}_{lc}$ be a sequence of local alternatives with common support satisfying $\varepsilon_n := \inf_{f \in \mathcal{F}_{lc}} h(f, f_n) > 0$,   $\sup_n \int |x|^4 f_n(x) dx < \infty$, and $\sup_n f_n(x) < \infty$.  If $\varepsilon_n^{-5/4} = o(n^{1/2})$, then there exists a constant $c > 0$ such that
\begin{equation*}
    \lim_{n \rightarrow \infty} \mathbb{P}(S_{n,lc} > c \varepsilon_n^2) = 1.
\end{equation*}
            
        \end{theorem}
\begin{proof}[Proof of Theorem \ref{thm:local_alt_lc_SN}]
     Under the conditions on $f_n$, we can modify the proof of Lemma 3 in \cite{cule2010theoretical} to show that the MLE $\hat{f}_{nn,lc}$ based on a random sample $X_{n1}, \ldots, X_{nn}$ from $f_n$ satisfies:
\begin{enumerate}
    \item[(a)] There exists a constant $C > 0$ such that
    \begin{equation*}
        \mathbb{P}\left( \limsup_{n \rightarrow \infty} \sup_{x \in \mathbb{R}} \hat{f}_{nn,lc}(x) \leq C \right) = 1.
    \end{equation*}

    \item[(b)] For any compact set $S$ in the support of $f_n$, there exists a constant $c = c(S) > 0$ such that
    \begin{equation*}
        \mathbb{P}\left( \liminf_{n \rightarrow \infty} \inf_{x \in S} \hat{f}_{nn,lc}(x) \geq c \right) = 1.
    \end{equation*}
\end{enumerate}
See also Lemma \ref{lemma:f_n*lc_UB_LB} for a similar result to Lemma 3 in \cite{cule2010theoretical}, applied to a sequence of densities, where we provide more details on the proof. In particular, for $g(x) = \exp(-|x| + b)$, where $b$ is a normalization constant such that $g$ is a density, we have, instead of \eqref{eq:lower_bound_fnlogg},
\begin{equation*}
    \inf_n \int f_n \log g = - \sup_n \int |x| f_n(x) \, dx + b > -\infty.
\end{equation*}
A sufficient condition for applying the strong law of large numbers for triangular arrays, as used in the proof of Lemma \ref{lemma:f_n*lc_UB_LB}, is that
\begin{equation*}
    \sup_{n,i} \mathbb{E}\left\{ \log g(X_{ni}) - \int f_n \log g \right\}^4 < \infty,
\end{equation*}
which holds in view of the assumption $\sup_n \int |x|^4 f_n(x) \, dx < \infty$. The remainder of the argument is similar to that of Lemma \ref{lemma:f_n*lc_UB_LB} and is therefore omitted.

    Without loss of generality, we may assume that the interval $[-1, 1]$ is strictly contained in the support of $f_n$. As a result of (a) and (b) above, with probability tending to $1$, $\hat{f}_{nn,lc} \in \mathcal{F}_{lc}^M$ for some $M > 0$. 
    
Similar to the proof of Theorem \ref{thm:log_concave_H1}, it remains to verify that the last term on the right-hand side of the following inequality
\begin{equation*}
    \int_{\varepsilon_n^2/2^8}^{\sqrt{2}\varepsilon_n} \sqrt{ \log N_{[\cdot]}\left( \frac{u}{c_3}, \mathcal{F}^M_{lc}, h \right)} \, du \precsim \int_0^{\sqrt{2}\varepsilon_n} u^{-1/4} \, du \precsim \varepsilon_n^{3/4}
\end{equation*}
is smaller than $c_4 n^{1/2} \varepsilon_n^2$ for all sufficiently large $n$. This holds since $\varepsilon_n^{-5/4} = o(n^{1/2})$. 
\end{proof}

\begin{proof}[Proof of Theorem \ref{thm:lc_H1_local_alt_rate}]
  With Theorem \ref{thm:local_alt_lc_SN}, the claim in the theorem follows as in the proof of Theorem \ref{coro:k_monotone_H1_local_alt_rate}.
\end{proof}

\section{Proofs for Section \ref{sect:bootstrap}}

\subsection{Leading term of $S_{n,1}$ under $H_0$}

With additional regularity conditions, we obtain the following Theorem \ref{thm:monotone_log_like_conv} for 1-monotonicity. Its proof makes use of a generalization of Theorem 1.1 in \cite{Kulikov2005} and Theorem 2.1 in \cite{Kulikov2008}. The former establishes the asymptotic normality of a  weighted $L_2$-error between $\hat{f}_n$ and $f_0$, while  the latter establishes the asymptotic normality of a weighted $L_1$-norm between $\hat{\mathbb{F}}_n$ and $\mathbb{F}_n$, where $\mathbb{F}_n$ is the empirical distribution function and $\hat{\mathbb{F}}_n$ is the distribution function corresponding to the MLE, which is the least concave majorant of $\mathbb{F}_n$ in this case.


		%
		\begin{theorem}\label{thm:monotone_log_like_conv}
			Suppose that $f_0$ is a twice continuously differentiable and strictly decreasing density function with support contained in $[0, 1]$ and that it satisfies the following conditions:
			\begin{enumerate}
				\item[(i)] $0 < f_0(1) \leq f_0(0) < \infty$;
				\item[(ii)] $0 < \inf_{t \in (0,1)} |f'_0(x)|$.
			\end{enumerate}
			Then,
			\begin{equation*}
				\sqrt{n} S_{n,1} =  -(\mu^2_{2, f_0} + \kappa_{f_0})n^{-1/6} + O_p(n^{-\frac{1}{3} + \delta}),
			\end{equation*}
			for any $\delta > 0$, where
			\begin{eqnarray*}
				\mu_{2, f_0} &:=&  \bigg[ \mathbb{E}[V(0)]^2 \int_0^1 \{2|f_0'(t)|^2f_0(t)^{-1}\}^{\frac{1}{3}}\,dt \bigg]^{1/2},\\
				\kappa_{f_0} &:=& \mathbb{E}[\zeta(0)]\int_0^1 \{2|f_0'(t)|^2 f_0(t)^{-1}\}^{\frac{1}{3}}dt,
			\end{eqnarray*}	
			with $V(c) = \arg\max_{t\in\mathbb{R}} \{\mathbb{W}(t) - (t-c)^2\}$, \\ $\zeta(c) =  [\mathrm{CM}_\mathbb{R} \mathcal{Z}](c)-\mathcal{Z}(c)$ and $\mathcal{Z}(c) = \mathbb{W}(c) - c^2$, for $c \in \mathbb{R}$. Here, $\mathbb{W}$ is a standard two-sided Brownian motion with $\mathbb{W}(0) = 0$ and $\mathrm{CM}_\mathbb{R}\mathcal{Z}$ denotes the least concave majorant of  $\mathcal{Z}$ on $\mathbb{R}$.
		\end{theorem}

		However, the result in Theorem \ref{thm:monotone_log_like_conv} cannot be readily applied to improve the accuracy of the test, as $\mu_{2, f_0}$ and $\kappa_{f_0}$ are unknown, and their estimation is not straightforward as they depend on the derivative of $f_0$.

\subsection{Proof of Theorem \ref{thm:monotone_log_like_conv}}

For simplicity, we write $\hat{f}_n$ for $\hat{f}_{n,1}$ in this section. Recall that $\hat{\mathbb{F}}_n$ denotes the least concave majorant of the empirical distribution function $\mathbb{F}_n$. To prove Theorem \ref{thm:monotone_log_like_conv}, we write
\begin{eqnarray*}
\frac{1}{\sqrt{n}} \sum^n_{i=1} \log \frac{\hat{f}_{n}(X_i)}{f_0(X_i)}
&=& \sqrt{n}\int_0^1 \log\frac{\hat{f}_n(t)}{f_0(t)} \,d\hat{\mathbb{F}}_n(t) + \sqrt{n}\int_0^1 \log \frac{\hat{f}_n(t)}{f_0(t)} \,d\{\mathbb{F}_n(t)-\hat{\mathbb{F}}_n(t)\} \\
&=:& A_n + B_n. 
\end{eqnarray*}
Here, $A_n$ can be viewed as the Kullback--Leibler divergence of $f_0$ from $\hat{f}_n$, and thus $A_n$ is nonnegative by a simple application of Jensen's inequality. The following 
Lemma \ref{lem:KL_weighted_L2error} will demonstrate that $A_n$ is asymptotically equivalent to a weighted $L_2$-error between $\hat{f}_n$ and $f_0$, which is asymptotically normally distributed by a generalization of Theorem 1.1 in \cite{Kulikov2005}; see Lemma \ref{lem:KL_weighted_L2error}.
Lemma \ref{lem:KL_minus_log_like} will show that $B_n$ is also nonnegative and asymptotically equivalent to a weighted $L_1$-norm between $\hat{\mathbb{F}}_n$ and $\mathbb{F}_n$, which is also asymptotically normally distributed by Theorem 2.1 in \cite{Kulikov2008}.


\begin{lemma}\label{lem:KL_weighted_L2error}
Under the conditions in Theorem \ref{thm:monotone_log_like_conv}, we have
\begin{equation*}\label{eq:KL_become_L2}
	\sqrt{n}\int_0^1 \log \frac{\hat{f}_n(t)}{f_0(t)} \,d\hat{\mathbb{F}}_n(t)
	= \sqrt{n}\int_0^1 \frac{|\hat{f}_n(t) - f_0(t)|^2}{2f_0(t)} \,dt + O_p(n^{-1/3 + \delta}),
\end{equation*}
for any $\delta > 0$.
\end{lemma}

\begin{lemma}\label{lem:KL_minus_log_like}
Under the conditions in Theorem \ref{thm:monotone_log_like_conv}, we have
\begin{equation*}
	\sqrt{n}\int_0^1 \log \frac{\hat{f}_n(t)}{f_0(t)} \,d\{\mathbb{F}_n(t)-\hat{\mathbb{F}}_n(t)\}
	=  \sqrt{n}\int_0^1 \{\hat{\mathbb{F}}_n(t) - \mathbb{F}_n(t)\} \frac{|f_0'(t)|}{f_0(t)}\,dt.
\end{equation*}
In addition, 
\begin{equation*}
	\sqrt{n}\int_0^1 \{\hat{\mathbb{F}}_n(t) - \mathbb{F}_n(t)\} \frac{|f_0'(t)|}{f_0(t)}\,dt = \kappa_{f_0}n^{-1/6} + O_p(n^{-1/3}).
\end{equation*}
\end{lemma}
\begin{remark}
Since $\hat{\mathbb{F}}_n$ is the least concave majorant of $\mathbb{F}_n$, $\hat{\mathbb{F}}_n \geq \mathbb{F}_n$. Thus, $\hat{\mathbb{F}}_n(t) - \mathbb{F}_n(t) = |\hat{\mathbb{F}}_n(t) - \mathbb{F}_n(t)|$ and so the term $\int^1_0 \{\hat{\mathbb{F}}_n(t) - \mathbb{F}_n(t) \} |f'_0(t)|/f_0(t) dt$ is a weighted $L_1$-norm of the difference between $\hat{\mathbb{F}}_n$ and $\mathbb{F}_n$.
\end{remark}


We first provide Lemmas \ref{lem:inverse_fhat_n_Op1} - \ref{lem:Lk_dist} that are used in the proof of Lemma \ref{lem:KL_weighted_L2error}.
\begin{lemma}\label{lem:inverse_fhat_n_Op1}
Suppose that $f_0$ is a decreasing density with support on $[0, 1]$ and $0 < f_0(1) \leq f_0(0) < \infty$. Then, $1/\hat{f}_n(\hat{Z}_n) = O_p(1)$. In other words, for any $\epsilon>0$, there exists $\delta >0$ such that for all sufficiently large $n$,
\begin{align*}
	\mathbb{P}( \hat{f}_n(Z_n)> \delta ) > 1-\epsilon. 
\end{align*}
\end{lemma}
\begin{proof}[Proof of Lemma \ref{lem:inverse_fhat_n_Op1}]
From page 326-328 of \cite{robertson1988order}, we know that
\begin{equation*}
	\hat{f}_n(Z_n) =  \min_{0 \leq i \leq n-1} \frac{n-i}{n(Z_n - Z_i)}.
\end{equation*}
Then
\begin{equation*}
	1/\hat{f}_n(Z_n) = \max_{0 \leq i \leq n-1} \frac{n(Z_n - Z_i)}{n-i} = \max_{h=1,\ldots,n} \frac{n(Z_n - Z_{n-h})}{h} = O_p(1),
\end{equation*}
by Corollary 5.2 (i) in \cite{chan2018estimation}.
\end{proof}

\begin{lemma}\label{lem:B2}
Under the conditions in Theorem \ref{thm:monotone_log_like_conv}, we have
\begin{align*}
	~\sqrt{n}\int_0^1
	|\hat{f}_n(t) - f_0(t)|^3\,dt = O_p(n^{-1/3 + \delta}) 
\end{align*}
for any $\delta > 0$.
\end{lemma}
\begin{proof}[Proof of Lemma \ref{lem:B2}]
For any $0 < \delta' < 2.5$, we can bound this third-degree moment term by
\begin{align*}
	~\int_0^1
	|\hat{f}_n(t) - f_0(t)|^3
	\,dt
	=&~ 
	\int_0^1\left|\hat{f}_n(t) - f_0(t)\right|^{2.5 -\delta'}\left|\hat{f}_n(t) - f_0(t)\right|^{0.5+\delta'}
	\,dt\\
	\leq&~ [2\max\{f_0(0), \hat{f}_n(0+)\}]^{0.5+\delta'}\int_0^1
	\left|\hat{f}_n(t) - f_0(t)\right|^{2.5-\delta'}
	\,dt, 
\end{align*}
where $[2\max\{f_0(0), \hat{f}_n(0+)\}]^{0.5+\delta'} = O_p(1)$.  Theorem 1.1 in \cite{Kulikov2005} with $k = 2.5-\delta'$ implies that
\begin{equation*}
	n^{1/6}\left[ n^{1/3}\left\{ \int_0^1
	\left|\hat{f}_n(t) - f_0(t)\right|^{2.5-\delta'}\,dt\right\}^{1/(2.5-\delta')} - \mu_{2.5 - \delta'}
	\right] = O_p(1),
\end{equation*}
for some finite constant $\mu_{2.5-\delta'}$. By straightforward algebra, we have
\begin{equation*}
	\left\{ \int_0^1
	\left|\hat{f}_n(t) - f_0(t)\right|^{2.5-\delta'}\,dt\right\}^{1/(2.5-\delta')} = O_p(n^{-1/2}) + \mu_{2.5 - \delta'}n^{-1/3} = O_p(n^{-1/3})
\end{equation*}
and so
\begin{equation*}
	\int_0^1
	\left|\hat{f}_n(t) - f_0(t)\right|^{2.5-\delta'} = O_p(n^{-5/6 + \delta}),
\end{equation*}
where $\delta = \delta'/3$ and the result in the lemma follows.

\end{proof}

It is indicated in Remark 1.1 in \cite{Kulikov2005} that one may obtain the asymptotic normality of the following weighted version of the $L_k$-error of $\hat{f}_n$:
\begin{equation*}
n^{1/6}\left\{n^{k/3} \int_0^1 \frac{|\hat{f}_n(t) - f_0(t)|^k}{2f_0(t)}\, dt - \mu_{k, f_0}^k\right\},
\end{equation*}
for some $\mu_{k, f_0}>0$. We state such claim formally in Lemma \ref{lem:Lk_dist} and omit the proof.

\begin{lemma}[Modified version of Theorem 1.1 in \cite{Kulikov2005}]
\label{lem:Lk_dist}
Under the conditions in Theorem \ref{thm:monotone_log_like_conv}, we have, for $1 \leq k < 2.5$,
\begin{equation*}
	n^{1/6}\left\{n^{k/3} \int_0^1 \frac{|\hat{f}_n(t) - f_0(t)|^k}{2f_0(t)}\, dt - \mu_{k, f_0}^k\right\}
\end{equation*}
converges in distribution to a zero-mean normal random variable with a finite variance as $n\to\infty$,
where
\begin{equation*}
	\mu_{k,f_0} := \bigg\{ \mathbb{E}|V(0)|^k \int^1_0 2^{\frac{2k}{3} -1} f_0(t)^{\frac{k}{3}- 1} |f'_0(t)|^{\frac{k}{3}} dt \bigg \}^{\frac{1}{k}},
\end{equation*}
with $V(\cdot)$ defined in Theorem \ref{thm:monotone_log_like_conv}.
\end{lemma}

\begin{proof}[Proof of Lemma \ref{lem:KL_weighted_L2error}]

Using Taylor's expansion, for some $f_n^*(t)$ lying between $f_0(t)$ and $\hat{f}_n(t)$, we have
\begin{eqnarray*}
	\sqrt{n}\int_0^1 \log \frac{\hat{f}_n(t)}{f_0(t)} \hat{f}_n(t)\,dt &=& \sqrt{n}\int_0^1 \frac{\hat{f}_n(t) - f_0(t)}{f_0(t)}\hat{f}_n(t)\,dt -\sqrt{n}\int_0^1\frac{\{\hat{f}_n(t) - f_0(t)\}^2}{2f_0^2(t)} \hat{f}_n(t)\,dt\\
	&& 	 + \sqrt{n}\int_0^1\frac{\{\hat{f}_n(t) - f_0(t)\}^3}{3\{f_n^*(t)\}^3} \hat{f}_n(t)\,dt, \\
	&=:& A_{n1} + A_{n2} + A_{n3}. 
\end{eqnarray*}
%
Interestingly, at the first glance, $A_{n1}$ and $A_{n2}$ do not share the same rate of convergence.  
However, they are equally important for the Kullback--Leibler divergence from $\hat{f}_n$ to $f_0$. 
In the following, we shall show that both $A_{n1}$ and $A_{n2}$ are related to a weighted $L_2$-error between $\hat{f}_n$ and $f_0$ asymptotically. 
For $A_{n3}$, it is relatively small and comparatively negligible.

\begin{enumerate}
	\item [(i)]
	For $A_{n1}$, because $\hat{f}_n$ and $f_0$ are both densities, we have $\int_0^1 \{\hat{f}_n(t) - f_0(t) \}\,dt = 0$. Thus,
	\begin{align*}
		A_{n1} 
		=&~ \sqrt{n}\int_0^1 \frac{\hat{f}_n(t) - f_0(t)}{f_0(t)}\hat{f}_n(t)\,dt 
		- \sqrt{n}\int_0^1 \{ \hat{f}_n(t) - f_0(t) \} \frac{f_0(t)}{f_0(t)}\,dt\\
		=&~ \sqrt{n}\int_0^1 \frac{\{\hat{f}_n(t) - f_0(t)\}^2}{f_0(t)}\,dt.
	\end{align*}

	\item [(ii)] For $A_{n2}$, we have
	\begin{align*}
		A_{n2} =&~-\sqrt{n}\int_0^1\frac{ \{\hat{f}_n(t) - f_0(t)\}^2}{2f_0(t)} \frac{\hat{f}_n(t)}{f_0(t)}\,dt\\
		=&~ -\sqrt{n}\int_0^1\frac{\{\hat{f}_n(t) - f_0(t)\}^2}{2f_0(t)}\left\{ \frac{f_0(t)}{f_0(t)} + \frac{\hat{f}_n(t) - f_0(t)}{f_0(t)}\right\} \,dt\\
		=&~ -\sqrt{n}\int_0^1\frac{\{\hat{f}_n(t) - f_0(t)\}^2}{2f_0(t)}\,dt
		- \sqrt{n}\int_0^1\frac{\{\hat{f}_n(t) - f_0(t)\}^3}{2\{f_0(t)\}^2}\,dt. 
	\end{align*}
	Under the conditions in Theorem \ref{thm:monotone_log_like_conv}, we can apply Lemma \ref{lem:B2} to obtain that the second term in the last displayed equation is of the order $O_p(n^{-1/3 + \delta} )$ because $f_0(1) > 0$.
	
	\item[(iii)]	We aim to show that $A_{n3} = O_p(n^{-1/3 + \delta})$.	Because $\hat{f}_n(t) = 0$ when $t>Z_n$ and, for all $t$, $\hat{f}_n(t) \leq \hat{f}_n(0+) = O_p(1)$, the absolute value of $A_{n3}$ is bounded above by 
	\begin{align*}
		\sqrt{n}\int_0^{Z_n}\frac{|\hat{f}_n(t) - f_0(t)|^3}{3\{f_n^*(t)\}^3} \hat{f}_n(t)\,dt 
		\leq O_p(1) \sqrt{n}\int_0^{Z_n}\frac{|\hat{f}_n(t) - f_0(t)|^3}{3\{f_n^*(t)\}^3}\,dt. 
	\end{align*}
	Although $f_n^*(t)$ is lying between and $\hat{f}_n$ and $f_0$, where for all $t, f_0(t) \geq f_0(1) > 0$, it is not guaranteed that $\hat{f}_n$ is bounded away from zero as well. By Lemma \ref{lem:inverse_fhat_n_Op1}, we show that $1/\hat{f}_n(1) = O_p(1)$. As a result, we have
	\begin{eqnarray*}
		\sqrt{n}\int_0^{Z_n}\frac{|\hat{f}_n(t) - f_0(t)|^3}{3\{f_n^*(t)\}^3}\,dt 
		&\leq &  \sqrt{n} \int^{Z_n}_0 \frac{|\hat{f}_n(t) - f_0(t)|^3}{3 \min\{\hat{f}_n(Z_n)^3, f_0(1)^3 \} } dt\\
		& \leq & O_p(1)  \sqrt{n} \int^1_0 |\hat{f}_n(t) - f_0(t)|^3 dt = O_p(n^{-1/3 +\delta}), 
	\end{eqnarray*}
	by Lemma \ref{lem:B2}.
\end{enumerate}
Combining (i)-(iii), the claim in the lemma follows.
%
\end{proof}

\begin{proof}[Proof of Lemma \ref{lem:KL_minus_log_like}]
First, from the proof of Lemma 4.2 in \cite{Jankowski2014}, we know 
\begin{equation*}
	\int_0^1 \log \hat{f}_n(t) \,d\mathbb{F}_n(t) =\int_0^1 \log \hat{f}_n(t) \, \hat{\mathbb{F}}_n(t).	
\end{equation*}
Therefore,
\begin{equation}\label{eq:log_fhat_n_dFndFhatn}
	\int_0^1 \log \frac{\hat{f}_n(t)}{f_0(t)} \,d \{\mathbb{F}_n(t) - \hat{\mathbb{F}}_n(t)\} = 
	-\int_0^1 \log f_0(t) \,d \{ \mathbb{F}_n(t) - \hat{\mathbb{F}}_n(t)\}. 
\end{equation}
Also, because $\mathbb{F}_n$ and $\hat{\mathbb{F}}_n$ are both distribution functions over $[0,1]$, we have 
\begin{equation*}
	\sqrt{n} \int_0^1 \log f_0(0) \,d\{\mathbb{F}_n(t) - \hat{\mathbb{F}}_n(t)\} = 0.	
\end{equation*}
Thus,
\begin{eqnarray}
	-\int_0^1\log f_0(t)\,d\{\mathbb{F}_n(t) - \hat{\mathbb{F}}_n(t)\} 
	&=& -\int_0^1  \{\log f_0(t) - \log f_0(0)\} \,d \{ \mathbb{F}_n(t) - \hat{\mathbb{F}}_n(t)\} \nonumber\\
	&=& \int_0^1 \int_0^t \frac{|f_0'(w)|}{f_0(w)} \,dw \,d \{\mathbb{F}_n(t) - \hat{\mathbb{F}}_n(t)\} \nonumber \\
	&=& 	\int_0^1 \left[\int_w^1  \,d\{\mathbb{F}_n(t) - \hat{\mathbb{F}}_n(t)\} \right] \frac{|f_0'(w)|}{f_0(w)} \,dw \nonumber \\
	&=& \int_0^1 \{ \hat{\mathbb{F}}_n(w) - \mathbb{F}_n(w)\} \frac{|f_0'(w)|}{f_0(w)} dw, \label{eq:Fubini}
\end{eqnarray}
where the third equality holds by Fubini's theorem. Under the conditions in Theorem \ref{thm:monotone_log_like_conv}, we can apply Theorem 2.1 in \cite{Kulikov2008} to obtain that
\begin{equation*}
	n^{1/6}\left[ n^{2/3} \int_0^1 \{\hat{\mathbb{F}}_n(w) - \mathbb{F}_n(w)\} \frac{|f_0'(w)|}{f_0(w)} dw - \kappa_{f_0} \right]
\end{equation*}
converges to a normal distribution with zero mean and a finite variance.
The above convergence implies that
\begin{equation}\label{eq:KL07_implication}
	\int_0^1 \{\hat{\mathbb{F}}_n(w) - \mathbb{F}_n(w)\} \frac{|f_0'(w)|}{f_0(w)} dw = \kappa_{f_0} n^{-1/6} + O_p(n^{-1/3}),
\end{equation}
where $\kappa_{f_0}$ is defined in Theorem \ref{thm:monotone_log_like_conv}.
The claim in the lemma follows in view of (\ref{eq:log_fhat_n_dFndFhatn}), (\ref{eq:Fubini}) and (\ref{eq:KL07_implication}).
\end{proof}

\begin{proof}[Proof of Theorem \ref{thm:monotone_log_like_conv}]
By Lemmas \ref{lem:KL_weighted_L2error} and \ref{lem:KL_minus_log_like}, 
\begin{equation*}
	\frac{1}{\sqrt{n}} \sum^n_{i=1} \log \frac{\hat{f}_{n}(X_i)}{f_0(X_i)} = - \sqrt{n}\int_0^1 \frac{|\hat{f}_n(t) - f_0(t)|^2}{2f_0(t)} \,dt  -  \kappa_{f_0}n^{-1/6} +  O_p(n^{-1/3 + \delta}).
\end{equation*}
It remains to show that 
\begin{equation*}
	\sqrt{n}\int_0^1 \frac{|\hat{f}_n(t) - f_0(t)|^2}{2f_0(t)}\,dt 
	= n^{-1/6}\mu_{2, f_0}^2 + O_p(n^{-1/3 + \delta}),
\end{equation*}
which is implied by Lemma \ref{lem:Lk_dist}.

\end{proof}

\subsection{Proof of Theorem \ref{thm:bootstrap_general}}
\begin{proof}[Proof of Theorem \ref{thm:bootstrap_general}]
Let $\mathcal{X}_n = \{X_1,\ldots,X_n \}$. As before, we suppress the dependence of $n$ in $\nu$ in the notation. To prove (\ref{eq:bootstrap_consistencty_general})  is equivalent to show that every subsequence  $\{n_k\}$ of $\{n\}$ has a further subsequence $\{n_{k_l}\}$ along which (\ref{eq:bootstrap_consistencty_general}) holds almost surely instead of in probability.
Similar to $T_n$, we can write $T^*_n$ as 
\begin{eqnarray*}
T_n^* = S_n^* + M_n^* + R_n^*.
\end{eqnarray*}
Clearly, $M_n^*$ is distribution-free  because $( \tilde{F}_n(Z_{n1}^*),\ldots, \tilde{F}_n(Z_{nl}^*)) \stackrel{d}{=} (U_{(n1)},\ldots,U_{(nn)})$, where $U_{ni}$'s are a random sample from the standard uniform distribution. Thus, for the entire sequence $\{n\}$, for almost all $\omega \in \mathcal{S}$, $	\sqrt{ \frac{n\nu}{\nu^2\psi_1(\nu)-\nu}} \left( M_n^* -  \log \nu +  \psi(\nu) \right) \stackrel{d}{\rightarrow} N(0, 1)
$ conditional on $\mathcal{X}_n(\omega)$ as in the proof of Theorem \ref{thm:main_hist_asy_dist} 

By conditions (i) and (ii), every subsequence $\{n_k\}$ of $\{n\}$ has a further subsequence $\{n_{k_l}\}$ along which $\sqrt{n_{k_l} \nu}S_{n_{k_l} } = o_{\tilde{\mathbb{P}}^*(\omega)} \left(1 \right)$ and $\sqrt{n_{k_l}\nu} R_{n_{k_l}}^* = o_{\tilde{\mathbb{P}}^*(\omega)} \left(1\right)$ for almost all $\omega \in \mathcal{S}$. Hence, along that subsequence $\{n_{k_l}\}$, for almost all $\omega \in \mathcal{S}$,  we have $\sqrt{ \frac{n_{k_l}\nu}{\nu^2\psi_1(\nu)-\nu}} \left( T_n^* -  \log \nu +  \psi(\nu) \right) \stackrel{d}{\rightarrow} N(0, 1)$ conditional on $\mathcal{X}_{n_{k_l}}(\omega)$. The result then follows as $N(0, 1)$ is a continuous distribution.
\end{proof}

\section{Bootstrap consistency for particular classes}
\subsection{Convergence under a sequence of underlying distributions}\label{sect:bootstrap_general}
To establish bootstrap consistency in Section \ref{sect:bootstrap}, we have to consider the MLE under the setting that the samples are generated from a sequence of underlying distributions defined on a probability space $(\mathcal{S}, \mathcal{G}, \mathbb{P})$. To this end, let $X_{n1},\ldots,X_{nn}$ be a random sample from $f_n$, a deterministic sequence of densities. Let $F_n$ be the distribution function of $f_n$ and $\mathbb{F}_{nn}$ the empirical distribution from $X_{ni}$'s. Denote the corresponding order statistics of $X_{ni}$'s by $Z_{n1} \leq  \ldots \leq  Z_{nn}$. The $k$-monotone MLE, completely monotone MLE, and log-concave MLE for $f_n$ from the sample $X_{ni}$'s are denoted by $\hat{f}_{nn,k}$, $\hat{f}_{nn, \infty}$ and $\hat{f}_{nn,lc}$, respectively. 
To prepare for the proofs of the bootstrap consistency in Theorems \ref{thm:bootstrap_consistency_k_monotone} - \ref{thm:bootstrap_consistency_LC}, we first provide upper bounds for $-\log f_n(Z_{n1})$ and $-\log f_n(Z_{nn})$ in the following Lemma \ref{lemma:fn_unimodal_logf_bound}, and establish the rates of convergence to $0$ of the respective $\hat{f}_{nn,k}$, $\hat{f}_{nn, \infty}$ and $\hat{f}_{nn,lc}$ in terms of the log-likelihood ratio to $f_n$ in Lemmas \ref{thm:bootstrap_Kn_rate} to \ref{lemma:f_n*lc_UB_LB} when $f_n$ is $k$-monotone, completely monotone and log-concave, respectively. 
To indicate the dependence of a sample point $\omega$ in the sample space $\mathcal{S}$ in different functions, notations such as $Z_{nn}(\omega)$ and $\hat{f}_{nn,k}(\cdot;\omega)$ will also be used.

\begin{lemma}\label{lemma:fn_unimodal_logf_bound}
    \begin{enumerate}[(a)]
        \item 
        Let $f_n$ be a sequence of decreasing densities satisfying $\int^\infty_{-\infty} f^{1\pm \alpha}_n(x)dx \leq n^{m_0}$ for some $\alpha, m_0 > 0$ for all sufficiently large $n$. Then, 
        \begin{equation*}
            |\log f_n(Z_{n1})| + |\log f_n(Z_{nn})| = O_p(\log n).
        \end{equation*}
        \item 
        Let $f_n$ be a sequence of log-concave densities satisfying $\int^\infty_{-\infty} f^{1- \alpha}_n(x)dx \leq n^{m_0}$ for some $\alpha, m_0 > 0$  and $f_n \leq C$ for all sufficiently large $n$. Then, 
        \begin{equation*}
            |\log f_n(Z_{n1})| + |\log f_n(Z_{nn})| = O_p(\log n).
        \end{equation*}
    \end{enumerate}    
\end{lemma}

\begin{proof}[Proof of Lemma \ref{lemma:fn_unimodal_logf_bound}]
    \begin{enumerate}[(a)]
        \item Let $m > m_0 + 1$. Since $f_n$ is decreasing,
        \begin{align*}
            \mathbb{P}( (f_n(Z_{nn}))^{-\alpha} > n^m ) &= 1 - \mathbb{P}( (f_n(Z_{nn}))^{-\alpha} \leq n^m )\\
            &= 1 - [\mathbb{P}( f_n(X_{n1})^{-\alpha} \leq n^m]^n \\
            &= 1 - [1 - \mathbb{P}( f_n(X_{n1})^{-\alpha} > n^m]^n\\
            &=1 - \left[ 1- \frac{\mathbb{E}(f_n(X_{n1})^{-\alpha}}{n^m} \right]^n\\
            &\leq 1 - \left( 1 - \frac{n^{m_0}}{n^m}\right)^n   \leq \frac{1}{n^2}.
        \end{align*}
        Hence, $\sum^\infty_{n=1} \mathbb{P}( (f_n(Z_{nn}))^{-\alpha} > n^m ) < \infty$. By the first Borel-Cantelli Lemma, with probability one, for all sufficiently large $n$, $f_n(Z_{nn})^{-\alpha} \leq n^m$ and thus $-\log f_n(Z_{nn}) \leq \frac{m}{\alpha} \log n$. Similar calculation gives
        \begin{equation*}
            \mathbb{P}( (f_n(Z_{nn}))^{\alpha} > n^m ) \leq \frac{1}{n^2}.
        \end{equation*}
        The same argument gives with probability one, $\log f_n(Z_{nn}) \leq \frac{m}{\alpha} \log n$ for all sufficiently large $n$. Therefore, $\log f_n(Z_{nn}) = O_p(\log n)$.
For $f_n(Z_{n1})$, we have
        \begin{align*}
            \mathbb{P}( (f_n(Z_{n1}))^{-\alpha} > n^m ) &=  \mathbb{P}( (f_n(X_{n1}))^{-\alpha} > n^m )^n
             \leq \left[ \frac{\mathbb{E}(f_n(X_{n1})^{-\alpha})}{n^m} \right]^n \\
            & \leq \left( \frac{n^{m_0}}{n^m}\right)^n \leq \frac{1}{n^2}.
        \end{align*}
        Similarly, 
        \begin{equation*}
            \mathbb{P}( (f_n(Z_{n1}))^{\alpha} > n^m ) \leq \frac{1}{n^2}.
        \end{equation*}
        Using the same argument as above, we obtain $\log f_n(Z_{n1}) = O_p(\log n)$.
\item Since $f_n \leq C$ for all sufficiently large $n$, 
        \begin{equation}\label{eq:log_concave_logmin}
            |\log f_n(Z_{n1})| + |\log f_n(Z_{nn})| \leq 2\max\{  |\log C|, - \log \min (f_n(Z_{n1}), f_n(Z_{nn}))\}.
        \end{equation}
Fix $m > m_0 + 1$. Then
	\begin{eqnarray*}
		\mathbb{P}( \min \{f_n(Z_{n1}), f_n(Z_{nn})\}^{-\alpha} > n^m) &=& 1 - \mathbb{P}(\min \{f_n(Z_{n1}), f_n(Z_{nn})\}^{-\alpha} \leq n^m) \\
		&=& 1 - \{ \mathbb{P}(f^{-\alpha}_n(X_{n1}) \leq n^m )  \}^n\\
		&=& 1 - \{1 - \mathbb{P}(f^{-\alpha}_n(X_{n1}) > n^m) )  \}^n \\
		& \leq & 1 - \left( 1 - \frac{ \mathbb{E} (f^{-\alpha}_n(X_{n1})) }{n^m} \right)^n \\
		& \leq & 1 - \left( 1 - \frac{ n^{m_0}}{n^m} \right)^n \\
		& \leq & \frac{1}{n^2},
	\end{eqnarray*}
	where the second equality follows because $f_n$, being log-concave, is unimodal, and the first inequality follows from Markov's inequality. Hence,  $\sum^\infty_{n=1} \mathbb{P}(\min \{f_n(Z_{n1}), f_n(Z_{nn})\} > n^m) < \infty$. By the first Borel-Cantelli Lemma, with  probability one, for all sufficiently large  $n$, $\min \{f_n(Z_{n1}), f_n(Z_{nn})\}^{-\alpha} \leq n^m$ and thus $-\log \min \{f_n(Z_{n1}), f_n(Z_{nn})\} \leq \frac{m}{\alpha} \log n$. Together with \ref{eq:log_concave_logmin}, the claim of the lemma follows.
    \end{enumerate}
    
\end{proof}

\begin{lemma}[for $k$-monotone densities]\label{thm:bootstrap_Kn_rate}
Let $\{M_n\}$ and $\{ \tilde{b}_n\}$ be sequences that can possibly go to $\infty$.	Suppose that $f_n \in \mathcal{F}_k^{M_n}([0, \tilde{b}_n])$ for all sufficiently large $n$. Then,
\begin{equation*}
	\frac{1}{n} \sum^n_{i=1} \log \frac{ \hat{f}_{nn,k}(X_{ni})}{f_n(X_{ni})} = O_p\left ( n^{-\frac{2k}{2k+1}}|\log  (\widetilde{M}_n \tilde{b}_n )|^{\frac{1}{2k+1}}\right),
\end{equation*}
where $\widetilde{M}_n$ satisfies $\widetilde{M}_n/M_n \rightarrow \infty$.
\end{lemma}

\begin{proof}[Proof of Lemma \ref{thm:bootstrap_Kn_rate}]
First, note that if $f_n \in \mathcal{F}_k^{M_n}([0, \tilde{b}_n])$, then $\hat{f}_{nn,k}(0+) = O_p(M_n)$. To see this, following the proof of Proposition 6 in \cite{gao2009rate}, we know that
\begin{equation*}
	\hat{f}_{nn,k}(0+) \leq k \sup_{t > 0} \frac{ \mathbb{F}_{nn}(t)}{F_n(t)} f_n(0+) = O_{p}(M_n),
\end{equation*}
since
\begin{equation*}
	\sup_{t > 0} \frac{ \mathbb{F}_{nn}(t)}{F_n(t)} = \sup_{ 0 < t \leq 1} \frac{\mathbb{G}_{nn}(t)}{t} = O_{p}(1)
\end{equation*}
by Daniels' theorem (see Theorem 2 on p. 345 in \cite{Shorack1986empirical}), where 
\begin{equation*}
	\mathbb{G}_{nn}(t) := \frac{1}{n} \sum^n_{i=1}I(F_n(X_{ni}) \leq t).
\end{equation*}
Since $f_n \in \mathcal{F}_k^{M_n}([0, \tilde{b}_n])$, it follows that $Z_{nn} \leq \tilde{b}_n$. By Lemma \ref{lem:kZn}, we also know $\sigma_{\hat{f}_{nn}} \leq k Z_{nn}$. Hence, $\sigma_{\hat{f}_{nn}} \leq k \tilde{b}_n$. Since $\widetilde{M}_n / M_n \rightarrow \infty$ by assumption, we then have
\begin{equation}\label{eq:fnn_k_Fk}
	\lim_{n \rightarrow \infty} \mathbb{P}(\hat{f}_{nn, k} \notin \mathcal{F}^{\widetilde{M}_n}_k([0, k \tilde{b}_n]) ) = \lim_{n \rightarrow \infty}\mathbb{P}(\hat{f}_{nn,k}(0+) > \widetilde{M}_n) = 0.
\end{equation}
Denote
\begin{equation*}
	K_{nn, k} := - \frac{1}{n} \sum^n_{i=1} \log \frac{\hat{f}_{nn,k}(X_{ni})}{f_n(X_{ni})}.
\end{equation*}
Let $\delta_n := n^{-\frac{k}{2k+1}}|\log( \tilde{M}_n \tilde{b}_n)|^{ \frac{1}{2(2k+1)}}$ and $m_{f,n}(t) := \log \frac{f(t) + f_n(t)}{2}$. Similar to the proof of Theorem \ref{thm:k_monotone_logLRT_rate}, we have
\begin{eqnarray*}
	&& 	\mathbb{P} \left(  - \frac{\sqrt{n}}{2} K_{nn, k} \geq \sqrt{n} 2^{2M} \delta^2_n \right) \\
	&\leq & \mathbb{P} \left( \mathbb{G}_n(m_{\hat{f}_{nn}}, n) - \frac{\sqrt{n}}{8} h^2 ( \hat{f}_{nn}, f_n ) \geq \sqrt{n} 2^{2M} \delta^2_n, \hat{f}_{nn,k} \in  \mathcal{F}^{\widetilde{M}_n}_k([0, k \tilde{b}_n]) \right)  \\
	&& \quad + \mathbb{P}(\hat{f}_{nn,k} \notin \mathcal{F}^{\widetilde{M}_n}_k([0, k \tilde{b}_n])),
\end{eqnarray*}
where the last term on the above display goes to $0$ as $n$ goes to $\infty$ by (\ref{eq:fnn_k_Fk}).
The rest of the proof is similar to that of Theorem \ref{thm:k_monotone_logLRT_rate}, and is therefore omitted.
\end{proof}


\begin{lemma}[for completely monotone densities]\label{lemma:CM_bootstrap_rate}
Let $f_n \in \mathcal{F}_\infty$ such that $\lim_{n \rightarrow \infty}\mathbb{P}(Z_{n1}^{-1} > \tilde{c}_{1n}) = 0$ and $\lim_{n \rightarrow \infty}\mathbb{P}(Z^{1+\beta}_{nn} > \tilde{c}_{2n}) = 0$ for some $\beta  >0$. Then
\begin{equation*}
	\frac{1}{n} \sum^n_{i=1} \log \frac{ \hat{f}_{nn,\infty}(X_{ni})}{f_n(X_{ni})} = O_p\left ( n^{-\frac{2}{3}}|\log  (\tilde{c}_{1n} \tilde{c}_{2n} )|^{\frac{1}{3}}\right),
\end{equation*}
\end{lemma}

\begin{proof}[Proof of Lemma \ref{lemma:CM_bootstrap_rate}]
The proof is essentially the same as that of Lemma \ref{lem:rate_cm} and is therefore omitted.
\end{proof}

\begin{lemma}[for log-concave densities]\label{lemma:f_n*lc_UB_LB}
Let $f_n \in \mathcal{F}_{lc}$ with support $[l_n, u_n]$, where $l_n$ and $u_n$ are sequences of real numbers that are decreasing and increasing, respectively. 
Suppose that there exists $M > 0$ such that $\sup_n \sup_{x \in \mathbb{R}} f_n(x) \leq M$ and for any compact set $S_1 \subset (\lim_n l_n, \lim_n u_n)$, there exists $m=m(S_1) > 0$ such that $\liminf_n \inf_{x \in S_1} f_n(x) \geq m$. Furthermore, suppose that there exists $f^* \in \mathcal{F}_{lc}$ such that $f^*(x) \leq e^{-a_0|x| + b_0}$, for some $a_0 > 0$ and $b_0 \in \mathbb{R}$, and for any $a \in (0, a_0)$,
\begin{equation}\label{eq:exp_L1_conv}
	\lim_{n \rightarrow \infty}	\int_{\mathbb{R}} e^{a|x|}|f_n(x) - f^*(x)|\, dx = 0.
\end{equation}
Then:
\begin{enumerate}
	\item [(a)] There exists a constant $C > 0$ such that
	\begin{equation*}
		\mathbb{P}\left ( \limsup_{n \rightarrow \infty} \sup_{x \in \mathbb{R}} \hat{f}_{nn,lc}(x) \leq C \right )  = 1.
	\end{equation*}  
	
	\item [(b)] For any compact set $S_2 \subset (\lim_n l_n, \lim_n u_n)$, there exists a constant $c = c(S_2) > 0$ such that 
	\begin{equation*}
		\mathbb{P}\left ( \liminf_{n \rightarrow \infty} \inf_{x \in S_2} \hat{f}_{nn,lc}(x) \geq c \right )  = 1.
	\end{equation*}  
	
	\item [(c)]
	\begin{equation*}
		\frac{1}{n} \sum^n_{i=1} \log \frac{ \hat{f}_{nn,lc}(X_{ni})}{f_n(X_{ni})} = O_p(n^{-4/5}),
	\end{equation*}
	which is the same as the rate when $f_n = f_0$ for all $n$, where $f_0$ is a fixed log-concave density.
\end{enumerate}
\end{lemma}

The main idea of the proof of Lemma \ref{lemma:f_n*lc_UB_LB} follows from the proof for Lemma 3 in \cite{cule2010theoretical}. In our case, we need to consider the lower and upper bounds of the log-concave MLE from a random sample from a sequence of log-concave densities instead from a fixed density.

\begin{proof}[Proof of Lemma \ref{lemma:f_n*lc_UB_LB}]

\begin{enumerate}
	\item [(a)] Let $g(x) = \exp(-|x|+b)$ for $x \in \mathbb{R}$, where $b$ is the normalization constant such that $g$ is a density.
	First, note that for any $a > 0$, $\lim_{x\rightarrow \infty} |x|/e^{a|x|} = 0$. Therefore, $\{x: |x| > e^{a|x|}\}$ is a bounded subset of $\mathbb{R}$ and $\sup\{|x|:|x| > e^{a|x|} \} < \infty$. Now, for some $a \in (0, a_0)$, we have
	\begin{eqnarray}
		\int_\mathbb{R} |x| f_n(x)\, dx &=& \int_{ \{x:|x| < e^{a|x|} \} } |x| f_n(x)\, dx + \int_{ \{x:|x| > e^{a|x|} \} }|x| f_n(x)\, dx \nonumber \\
		&\leq& \int_{\mathbb{R}} e^{a|x|}f_n(x)\, dx + \int_{{ \{x:|x| > e^{a|x|} \} }}|x| f_n(x)\, dx \nonumber \\
		&\leq & \int_{\mathbb{R}} e^{a|x|}f_n(x)\, dx + C_1, \label{eq:bound_xf_n}
	\end{eqnarray}
	where $C_1 := M\sup\{|x|: (\log |x|)/|x|> a\} \lambda({ \{x: (\log |x|)/|x|> a \} })$ and $\lambda$ is the Lebesgue measure. Thus,
	\begin{eqnarray*}
		\int_\mathbb{R} f_n(x) \log g(x) \, dx &=& - \int_\mathbb{R} |x| f_n(x) \, dx + b\\
		& \geq& - \int_\mathbb{R} e^{a|x|}f_n(x)\, dx - C_1 + b.
	\end{eqnarray*}
	As a consequence of (\ref{eq:exp_L1_conv}), we have
	\begin{equation}\label{eq:lower_bound_fnlogg}
		\liminf_{n \rightarrow \infty} \int_\mathbb{R} f_n \log g \geq - \int_\mathbb{R} e^{a|x|} f^*(x) \, dx - C_1 + b =: q + 1,
	\end{equation}
	say, with $f^*$ as specified in the statement of the lemma and $\int_{\mathbb{R}} e^{a|x|}f^*(x)\, dx < \infty$ because $a \in (0, a_0)$. Now, let $\widetilde{M} := e^{M_2}$ where $M_2$ is large enough such that $M_2 > q+1$ and such that $\int_D f_n \leq \frac{1}{4}$ whenever $D \subset \mathbb{R}$ such that $\lambda(D) \leq 2^4 (M_2 - q) e^{-M_2}$. This is possible because $f_n \leq M$ for all $n$. Let $f$ be any log-concave density with $\sup_{x \in \mathbb{R}} f(x) =\widetilde{M}$. We claim that, for all sufficiently large $n$, the log-concave density $g$ has larger in value of the log-likelihood. Equivalently, we shall show that
	\begin{equation}\label{eq:f_cannot_be_NPMLE}
		\mathbb{P} \bigg( \frac{1}{n} \sum^n_{i=1} \log f(X_{ni}) > \frac{1}{n} \sum^n_{i=1} \log g(X_{ni}) \text{ i.o.} \bigg) = 0,
	\end{equation}
	where i.o. stands for infinitely often. Observe that
	\begin{eqnarray}
		&& \mathbb{P} \bigg( \frac{1}{n} \sum^n_{i=1} \log f (X_{ni}) > \frac{1}{n} \sum^n_{i=1} \log g(X_{ni}) \text{ i.o.} \bigg) \nonumber \\
		&& \quad \leq \mathbb{P} \bigg( \frac{1}{n} \sum^n_{i=1} \log g(X_{ni}) < q \text{ i.o.} \bigg) +  \mathbb{P} \bigg( \frac{1}{n} \sum^n_{i=1} \log f(X_{ni}) > q \text{ i.o.} \bigg). \label{eq:phiX_ni}
	\end{eqnarray}
	Because of (\ref{eq:lower_bound_fnlogg}),
	\begin{eqnarray}\label{eq:gXniSLLN}
		\quad \quad \mathbb{P} \bigg( \frac{1}{n} \sum^n_{i=1} \log g(X_{ni}) < q \text{ i.o.} \bigg) \leq 
		\mathbb{P} \bigg( \frac{1}{n} \sum^n_{i=1} \log g(X_{ni}) -\int_\mathbb{R} f_n \log g < -1 \text{ i.o.} \bigg)
	\end{eqnarray}
	We now claim that the right-hand side of (\ref{eq:gXniSLLN}) is $0$. First, using similar bounds as in (\ref{eq:bound_xf_n}), it is easy to see that $\sup_{n,i} \mathbb{E}\{\log g(X_{ni}) - \int f_n \log g \}^4 < \infty$. For example, to show that $ \mathbb{E}[\{ \log g(X_{ni}) \}^4] < \infty$, we have for some $a \in (0, a_0)$,
	\begin{eqnarray*}
		&& \mathbb{E}[\{ \log g(X_{ni}) \}^4] \leq  4 \int_\mathbb{R} |x|^4 f_n(x) \, dx + 4|b|^4 \\
		& \leq & 4 \int_{\mathbb{R}} e^{a|x|}f_n(x)\, dx+ 4 M \sup \{|x|: |x|^4 > e^{a|x|}\} \lambda (\{x: |x|^4 > e^{a|x|}\}) + 4|b|^4.
	\end{eqnarray*}
	To see the bound of the last inequality is finite, simply note that by (\ref{eq:exp_L1_conv}), $\lim_n \int_{\mathbb{R}} e^{a|x|}f_n(x)\, dx = \int_{\mathbb{R}} e^{a|x|} f^*(x) \, dx < \infty$ for $a \in (0, a_0)$. Let $\{Y_{ni}\}_{i=1,\ldots,n,n \in \mathbb{N}}$ be a triangular array of scalar random variables $Y_{ni} := \log g(X_{ni}) -\int_\mathbb{R} f_n \log g$. Then, the row $Y_{n1},\ldots,Y_{nn}$ is a collection of independent random variables with the mean $0$ and $\sup_{i=1,\ldots,n,n} \mathbb{E}(Y_{ni}^4) < \infty$.	 Hence, by a strong law of large numbers for triangular arrays, 
	\begin{equation*}
		\frac{1}{n} \sum^n_{i=1} \log g(X_{ni}) -\int_\mathbb{R} f_n \log g = \frac{1}{n} \sum^n_{i=1} Y_{ni} \stackrel{a.s.}{\longrightarrow} 0
	\end{equation*}
	and so the right-hand side of (\ref{eq:gXniSLLN}) is $0$. The proof that shows the second term on the right-hand side of (\ref{eq:phiX_ni}) equals $0$ is similar to the corresponding proof in Lemma 3(a) in \cite{cule2010MLE_log_concave}  (p.262) with $X_i$ replaced by $X_{ni}$, $f_0$ replaced by $f_n$ and the fact that Hoeffding's inequality is valid for each $n$, and is therefore omitted.
	
	\item [(b)]
	Let $S$ be a compact subset of the interval $(\lim_n l_n, \lim_n u_n)$ and $\delta > 0$ be small enough that $S^{\delta} := \{x \in \mathbb{R}: \text{dist}(x, S) \leq \delta \} \subset (\lim_n l_n, \lim_n u_n)$.
	Let $f$ be any log-concave density on $\mathbb{R}$. We claim that if $c := 2\inf_{x \in S} f(x)$ is sufficiently small but positive, then $f$ cannot be the NPMLE for any large $n$ with  probability one, that is, (\ref{eq:f_cannot_be_NPMLE}) holds. Indeed, by Lemma \ref{lemma:f_n*lc_UB_LB} (a), we can assume that $\sup_{x \in \mathbb{R}} f(x) \leq C$.
	If $B \subset \mathbb{R}$ contains a ball of radius $\delta/2$ centered at a point in $S^{\delta/2}$, say $B_{\delta/2}(x_0)$, where $B_{\delta/2}(x_0)$ is the closed ball of radius $\delta$ centered at $x_0$, for some $x_0 \in S^{\delta/2}$, then
	\begin{eqnarray*}
		\liminf_n \int_B f_n &\geq &\liminf_n \int_{B_{\delta/2}(x_0)} f_n \geq \liminf_n  \inf_{x \in S^{\delta/2}} \int_{B_{\delta/2}(x)} f_n \\
		&\geq & \delta \liminf_n  \inf_{x \in S^\delta}f_n =: p > 0.
	\end{eqnarray*}
	Recall the density $g$ defined in (a) and the constant $q$ defined in (\ref{eq:lower_bound_fnlogg}); suppose that $c \in (0, C]$ is small enough such that
	\begin{equation}\label{eq:for_inequality_io}
		\frac{p}{2}\log c + \left(1 - \frac{p}{2}\right) \log C \leq q.
	\end{equation}
	Denote $B := \{x \in S^\delta : f(x) \leq c \}$. Then, $B$ contains a ball of radius $\delta/2$ centered at a point in $S^{\delta/2}$, so $\liminf_n \int_B f_n \geq p$. Then 
	\begin{equation*}
		\mathbb{P}\bigg( \frac{1}{n} \sum^n_{i=1} \log f (X_{ni})> q \bigg) \leq \mathbb{P} \bigg( \frac{1}{n}\sum^n_{i=1} I\{X_{ni} \in B \} \leq \frac{p}{2} \bigg) \leq e^{-np^2/2},
	\end{equation*}
	by (\ref{eq:for_inequality_io}) and Hoeffding's inequality. By the first Borel-Cantelli lemma,
	\begin{equation*}
		\mathbb{P} \bigg( \frac{1}{n} \sum^n_{i=1} \log f(X_{ni}) > q \text{ i.o.} \bigg) = 0.
	\end{equation*}
	Finally, arguing as in the proof of Lemma \ref{lemma:f_n*lc_UB_LB} (a) above (see (\ref{eq:phiX_ni}) and (\ref{eq:gXniSLLN})), we have (\ref{eq:f_cannot_be_NPMLE}).

	\item [(c)]
	Without loss of generality, we can assume that $[-1,1]$ is strictly inside $(\lim l_n, \lim u_n)$. By parts (a) and (b) in this lemma, we have, with a probability going to $1$ that $\hat{f}_{nn,lc} \in \mathcal{F}^M_{lc}$ for some finite $M > 0$, where $\mathcal{F}^M_{lc}$ is defined in  (\ref{eq:def_FMlc}). With the above result and the result from Theorem 3.1 in \cite{doss2016global} that $\log N_{[\cdot]}(\varepsilon, \mathcal{F}^M_{lc},h) \precsim \varepsilon^{-1/2}$, the proof of the claim regarding the rate of convergence to $0$ in this part is similar to that for the $k$-monotone case and is therefore omitted. 
\end{enumerate}

\end{proof}

\subsection{Bootstrap consistency}
In this subsection, we establish the bootstrap consistency of the NPLRT for $k$-monotone densities, completely monotone densities, and log-concave densities under mild regularity conditions.

\begin{theorem}[Bootstrap Consistency for $k$-monotone Densities]\label{thm:bootstrap_consistency_k_monotone}
Fix $k \in \mathbb{N}$. Suppose that $f_0$ is a density that may not be bounded from above or have a bounded support. Without loss of generality, assume that $\tau_{f_0} = 0$. Suppose that $f_0$ satisfies Conditions (A) and (B). Let $\nu = O(n^{1/3}(\log n)^{-1})$.
Regardless of whether $f_0$ is $k$-monotone or not, 
\begin{equation*}
	\sup_{x \in \mathbb{R}} \left| \mathbb{P}^* \left( \sqrt{ \frac{n\nu}{\nu^2\psi_1(\nu)-\nu}} \left( T_n^* -  \log \nu +  \psi(\nu) \right)  \leq x\right) - \mathbb{P}(Z \leq x ) \right| \stackrel{\mathbb{P}}{\rightarrow} 0,
\end{equation*}
where $Z \sim N(0, 1)$  and $\tilde{f}_n$ in the bootstrap procedure is the $k$-monotone MLE $\hat{f}_{n,k}$.
\end{theorem}

\begin{proof}[Proof of Theorem \ref{thm:bootstrap_consistency_k_monotone}]
First, from the proof of Lemma \ref{lemma:anbn_logn_hall}, we have $Z_1^{-1} = o_p(n^{m_0})$ and $Z_n = o_p(n^{m_0})$ for some $m_0 > 0$.

In view of Theorem \ref{thm:bootstrap_general}, it suffices  to verify Conditions (i) and (ii) in  Theorem \ref{thm:bootstrap_general}. Let $M_n = n^{m_0}$, $\tilde{b}_n = n^{m_0}$, and $\widetilde{M}_n = n^{2m_0}$.
Recall that by Lemma \ref{lemma:k_monotone_bounded_in_prob} and \ref{lemma:bounds_on_kmonotone_support}, we have $\hat{f}_{n,k}(0+) \leq kZ^{-1}_1$ and $\sigma_{\hat{f}_n} \leq k Z_n$. Thus,
\begin{eqnarray*}
	\mathbb{P}(\hat{f}_{n,k} \notin \mathcal{F}^{M_n}_k([0,\tilde{b}_n]))  & \leq & \mathbb{P}(\hat{f}_{n,k}(0+) > M_n) + \mathbb{P}(\hat{f}_n(\tilde{b}_n) > 0) \\
	& \leq & \mathbb{P}(kZ^{-1}_1 > M_n) + \mathbb{P}(kZ_n > \tilde{b}_n) \rightarrow 0,
\end{eqnarray*}
as $n$ goes to $\infty$ by assumptions. This implies that $I(\hat{f}_{n,k} \notin \mathcal{F}^{M_{n}}_k ([0, \tilde{b}_n])) \stackrel{\mathbb{P}}{\rightarrow } 0$. The last result together with the assumption that $Z_n = o_p(n^{m_0})$ imply that for any subsequence $\{n_k\}$ of $\{n\}$, there exists a further subsequence $\{n_{k_l}\}$ and $\mathcal{S}_0$ such that $\mathbb{P}(\mathcal{S}_0) = 1$ and for all $\omega \in \mathcal{S}_0$, $\hat{f}_{n_{k_l},k}(\cdot;\omega) \in \mathcal{F}^{M_{n_{k_l}}}_k([0,\tilde{b}_{n_{k_l}}])$ and $Z_{n_{k_l}}(\omega)/{n_{k_l}}^{m_0} \leq 1$  for all large enough $l$.  Therefore, Lemma \ref{thm:bootstrap_Kn_rate} implies that Condition (i) is satisfied as
\begin{equation*}
	\sqrt{n \nu} n^{- \frac{2k}{2k+1}} | \log (\widetilde{M}_n \tilde{b}_n)|^{\frac{1}{2k+1}} = o(1),
\end{equation*}
for all $k \in \mathbb{N}$ under $\nu = O(n^{1/3}/\log n)$.

In the following, the dependence on $\omega$ is not always written explicitly 	for the sake of notational simplicity. We now verify Condition (ii) in Theorem \ref{thm:bootstrap_general}. Write
\begin{equation*}
    R_{n_{k_l}}^* = R_{1{n_{k_l}}}^* +R_{2{n_{k_l}}}^*,
\end{equation*}
where
\begin{align*}
R_{1{n_{k_l}}}^* &:=	- \frac{1}{{n_{k_l}}} \sum_{j=0,\ldots,\frac{{n_{k_l}}-1}{\nu} - 1} 
	\sum^\nu_{l=1} \log \frac{\hat{f}_{{n_{k_l}},k}(Z_{j\nu+l+1}^*)(Z_{(j+1)\nu+1}^* - Z_{j\nu+1}^*)}{\hat{F}_{{n_{k_l}},k}(Z_{(j+1)\nu+1}^*) - \hat{F}_{{n_{k_l}},k}(Z_{j\nu+1}^*)},\\
	R_{2n}^* &:= -\frac{1}{n}\log \frac{\hat{f}_{n,k}(Z_1)}{f^{H,*}_{n,\nu}(Z_1)}.
\end{align*}
For $R_{1{n_{k_l}}}^*$, by the mean value theorem, there exists $\tilde{Z}^*_j$ lying between $Z_{j\nu+1}^*$ and $Z_{(j+1)\nu+1}^*$ such that $\hat{F}_{{n_{k_l}},k}(Z_{(j+1)\nu+1}^*) - \hat{F}_{{n_{k_l}},k}(Z_{j\nu+1}^*) = \hat{f}_{{n_{k_l}},k}(\tilde{Z}^*_j)(Z_{(j+1)\nu+1}^* - Z_{j\nu+1}^*)$ and so
	\begin{equation*}
R_{1{n_{k_l}}}^*  = -\frac{1}{{n_{k_l}}}\sum_{j=0,\ldots, \frac{{n_{k_l}}-1}{\nu}-1} \sum^\nu_{l=1} \log \frac{\hat{f}_{n_{k_l},k}(Z_{j\nu+l+1}^*)}{\hat{f}_{n_{k_l},k}(Z^*_j)}.
	\end{equation*}
 By the monotonicity of $\hat{f}_{n_{k_l},k}$,
\begin{align}
    |n_{k_l} R_{1n_{k_l}}^*| & \leq \left|  \sum_{j=0,\ldots, \frac{n_{k_l}-1}{\nu}-1} \sum^\nu_{l=1} \log \frac{\hat{f}_{n_{k_l},k}(Z_{j\nu+1}^*)}{\hat{f}_{n_{k_l},k}(Z_{(j+1)\nu+1}^*)}\right| 
    = \nu \left| \log \hat{f}_{n_{k_l},k}(Z_1^*) - \log \hat{f}_{n_{k_l},k}(Z_{n_{k_l}}^*)\right|\nonumber \\
    & \leq \nu (|\log \hat{f}_{n_{k_l},k}(Z_1^*)| + |\log \hat{f}_{n_{k_l},k}(Z_{n_{k_l}}^*)|). 
\end{align}
%
We now apply Lemma \ref{lemma:fn_unimodal_logf_bound} (a) with $f_{n_{k_l}} = \hat{f}_{n_{k_l},k}$. To verify the condition in that lemma, note that $\int^\infty_0 \hat{f}^{1\pm\alpha}_{n_{k_l},k}(x) \, dx \leq k \hat{f}_{n_{k_l}}^{1\pm\alpha}(0) Z_{n_{k_l}} \leq k M^{1\pm\alpha}_{n_{k_l}} \tilde{b}_{n_{k_l}} = k n_{k_l}^{m_0(1\pm\alpha) + m_0} \leq n_{k_l}^{m_0(1\pm\alpha) + m_0 + 1}$ for some $\alpha \in (0, 1)$ and for all large $l$. By Lemma \ref{lemma:fn_unimodal_logf_bound} (a), we have $|\log \hat{f}_{n_{k_l},k}(Z_1^*)| +  |\log \hat{f}_{n_{k_l},k}(Z_n^*)| = O_{\mathbb{P}^*_\omega}(\log n_{k_l})$.
Therefore, 
\begin{equation*}
    |\sqrt{n_{k_l} \nu} R_{1n_{k_l}}^*| \leq 
    \frac{\nu^{3/2}}{\sqrt{n_{k_l}}}O_{\mathbb{P}^*_\omega}(\log n_{k_l}) = o_{\mathbb{P}^*_\omega}(1)
\end{equation*}
when $v = O(n^{1/3}/\log n)$. Following the proof of Lemma \ref{lemma:Rn_terms}, we have
\begin{equation*}
    |n_{k_l} R^*_{2n_{k_l}}| \leq |\log \hat{f}_{n_{k_l},k}(Z^*_1)| + |\log \hat{f}_{n_{k_l},k}(Z^*_{n_{k_l}})| + |\log(\hat{F}_{n_{k_l},k}(Z_{\nu+1}^*) - 
    \hat{F}_{n_{k_l},k}(Z_{1}^*))| + \log \frac{n_{k_l}-1}{\nu},
\end{equation*}
which is $O_{\mathbb{P}^*_\omega}(\log n_{k_l})$ as in (\ref{eq:R2n4}). Thus, Condition (ii) in Theorem \ref{thm:bootstrap_general} is satisfied.

\end{proof}

\begin{theorem}[Bootstrap Consistency for Completely Monotone Densities]\label{thm:bootstrap_consistency_CM}
Suppose that $f_0$ is a density that may not be bounded from above or have a bounded support. Without loss of generality, assume that $\tau_{f_0} = 0$. Suppose that $f_0$ satisfies Conditions (A) and (B). Let $\nu = O(n^{1/3}(\log n)^{-1})$.
Regardless of whether $f_0$ is completely monotone or not, 
\begin{equation*}
	\sup_{x \in \mathbb{R}} \left| \mathbb{P}^* \left( \sqrt{ \frac{n\nu}{\nu^2\psi_1(\nu)-\nu}} \left( T_n^* -  \log \nu +  \psi(\nu) \right)  \leq x\right) - \mathbb{P}(Z \leq x ) \right| \stackrel{\mathbb{P}}{\rightarrow} 0,
\end{equation*}
where $Z \sim N(0, 1)$  and $\tilde{f}_n$ in the bootstrap procedure is the completely monotone MLE $\hat{f}_{n,\infty}$.
\end{theorem}

\begin{proof}[Proof of Theorem \ref{thm:bootstrap_consistency_CM}]
First, from the proof of Lemma \ref{lemma:anbn_logn_hall}, we have $Z_1^{-1} = o_p(n^{m_0})$ and $Z_n = o_p(n^{m_0})$ for some $m_0 > 0$.

Since $Z^{-1}_1 = o_p(n^{m_0})$ and $Z_n = o_p(n^{m_0})$, for any subsequence $\{n_k\}$ of $\{n\}$, there exists a further subsequence $\{n_{k_l}\}$ and $\mathcal{S}_0$ such that $\mathbb{P}(\mathcal{S}_0) = 1$ and for all $\omega \in \mathcal{S}_0$, $Z^{-1}_1(\omega)/n^{m_0}_{k_l} \rightarrow 0$ and $Z_{n_{k_l}}(\omega)/n^{m_0}_{k_l} \rightarrow 0$. Now, note that the distribution function of the completely monotone MLE $\hat{F}_{n,\infty}(t)$ is given by $ = 1 - \sum^{r_n}_{i=1} \hat{p}_i e^{-\hat{\lambda}_i t}$. Thus,
\begin{equation*}
	1 - e^{-\frac{t}{Z_n}} \leq 1 - \max_{i=1,\ldots,r_n} e^{-\hat{\lambda}_i t} \leq \hat{F}_n(t) \leq 1 -  \min_{i=1,\ldots,r_n} e^{-\hat{\lambda}_i t} \leq 1 - e^{-\frac{t}{Z_1}}.
\end{equation*}
Hence, we have
\begin{eqnarray*}
	\mathbb{P}^*_{\omega}( \{Z^*_{n_{k_l},1}\}^{-1} > t) &=& 1 - \mathbb{P}^*_\omega( \{Z^*_{n_{k_l}1}\}^{-1} \leq t) \\
	&=& 1 - \{\mathbb{P}^*_\omega( \{X^*_{n_{k_l},1}\}^{-1} \leq t) \}^{n_{k_l}} \\
	&=& 1 - \left \{\mathbb{P}^*_\omega \left(X^*_{n_{k_l},1}  > \frac{1}{t} \right) \right \}^{n_{k_l}} \\
	&=& 1 - \{1 - \hat{F}_{n_{k_l}}(1/t; \omega)\}^{n_{k_l}} \\
	& \leq & 1 - e^{- \frac{n_{k_l}}{Z_1(\omega) t}}
\end{eqnarray*}
and
\begin{equation*}
	\mathbb{P}^*_\omega \left(Z^*_{n_{k_l},n_{k_l}} > t\right) =  1 - \{\mathbb{P}^*_\omega(X^*_{n_{k_l},1} \leq t) \}^{n_{k_l}} \leq 1 - \left(1 - e^{-\frac{t}{Z_{n_{k_l}}(\omega)}}\right)^{n_{k_l}}.
\end{equation*}
Therefore, as $l$ goes to infinity,
\begin{equation*}
	\mathbb{P}^*_{\omega}(\{Z^*_{n_{k_l}, 1}\}^{-1} > n^{m_0 + 1}_{k_l}) \leq 1 - e^{- \frac{Z^{-1}_1(\omega)}{n^{m_0}_{k_l}} } \rightarrow 0
\end{equation*}
and
\begin{equation*}
	\mathbb{P}^*_\omega(Z^*_{n_{k_l}, n_{k_l}} > n^{m_0 + 1}_{k_l}) \leq 1 - \left(1 - \exp \left\{- \frac{n^{m_0}_{k_l}}{Z_{n_{k_l}}(\omega)} n_{k_l} \right\} \right)^{n_{k_l}}
	\rightarrow 1 - e^0 = 0.
\end{equation*}
We now apply Lemma \ref{lemma:CM_bootstrap_rate} with $f_{n_{k_l}} = \hat{f}_{n_{k_l}, \infty}(\cdot; \omega)$, $\tilde{c}_{1n_{k_l}} = n^{m_0+1}_{k_l}$ and $\tilde{c}_{2n_{k_l}} = n^{\frac{m_0+1}{1+\beta}}_{k_l}$. Therefore, Condition (i) in Theorem \ref{thm:bootstrap_general} is satisfied.

To verify Condition (ii) in Theorem \ref{thm:bootstrap_general}, similar to the $k$-monotone case, it suffices to show that 
\begin{equation}\label{eq:temp2}
	|\log \hat{f}_{n_{k_l},\infty}(Z^*_1)| + |\log \hat{f}_{n_{k_l},\infty}(Z^*_{n_{k_l},n_{k_l}})| = O_{ \mathbb{P}^*_\omega}(\log n).
\end{equation}
To apply Lemma \ref{lemma:fn_unimodal_logf_bound} (a), note that $\hat{f}_{n_{k_l},\infty}(0) \leq Z^{-1}_1$ and for all large $l$, $\hat{f}_{n_{k_l}, \infty}(x) \leq 1/x^2$ for all $x \geq Z^{1+\beta}_{n_{k_l}}(\omega)$ (see the proof of Lemma \ref{lem:rate_cm}). Hence, for all large $l$,
\begin{eqnarray*}
	\int^\infty_0 \hat{f}^{1\pm \alpha}_{n_{k_l}, \infty}(x;\omega)\, dx &=&  \int^{Z^{1+\beta}_{n_{k_l}}(\omega)}_0 \hat{f}^{1\pm\alpha}_{n_{k_l}, \infty}(x;\omega)\, dx + \int^\infty_{Z_{n_{k_l}}^{1+\beta}(\omega)} \hat{f}^{1\pm\alpha}_{n_{k_l}, \infty}(x;\omega)\, dx  \\
	&	\leq & Z^{-(1\pm\alpha)}_1(\omega) Z^{1+\beta}_{n_{k_l}}(\omega) + \int^\infty_{Z^{1+\beta}_{n_{k_l}}(\omega)} \frac{1}{x^{2(1\pm \alpha)} } \, dx \\
	& \leq &  n_{k_l}^{m_0(1\pm \alpha) + m_0(1+\beta)} + \frac{1}{Z^{(1+\beta)(1\pm 2\alpha) }_{n_{k_l}}(\omega)} \\
	& \leq & n_{k_l}^{m_0(1  \pm \alpha) + m_0(1+\beta) + 1},
\end{eqnarray*}
as $Z^{(1+\beta)(1\pm 2\alpha)}_1(\omega) > 1$ for all large $l$. Therefore, we can apply Lemma \ref{lemma:fn_unimodal_logf_bound} (a) and (\ref{eq:temp2}) is satisfied.
%
%
\end{proof}

\begin{theorem}[Bootstrap Consistency for Log-concave Densities]\label{thm:bootstrap_consistency_LC}
Regardless of whether $f_0$ is log-concave or not, suppose that $\int_{\mathbb{R}} |x| f_0(x) \, dx< \infty$ and $\nu = O(n^{1/3}(\log n)^{-1})$. We have
\begin{equation*}
	\sup_{x \in \mathbb{R}} \left| \mathbb{P}^* \left( \sqrt{ \frac{n\nu}{\nu^2\psi_1(\nu)-\nu}} \left( T_n^* -  \log \nu +  \psi(\nu) \right)  \leq x\right) - \mathbb{P}(Z \leq x ) \right| \stackrel{\mathbb{P}}{\rightarrow} 0,
\end{equation*}
where $Z \sim N(0, 1)$ and  $\tilde{f}_n$ in the bootstrap procedure is the log-concave MLE $\hat{f}_n$.
\end{theorem}

\begin{proof}[Proof of Theorem \ref{thm:bootstrap_consistency_LC}]
In view of Theorem \ref{thm:bootstrap_general}, it suffices  to verify Conditions (i) and (ii) in  Theorem \ref{thm:bootstrap_general}.
%
%
First, note that the support of $\hat{f}_{n, lc}$ is $[Z_1, Z_n]$. We now claim that $Z_n - Z_1 = o_p(n^2)$. Indeed, note that $|Z_n - Z_1| \leq |Z_n| + |Z_1| \leq 2 \max_{i=1,\ldots,n}|X_i|$ and for any $\delta > 0$,
\begin{equation*}
	\mathbb{P}\left(\max_{i=1,\ldots,n} |X_i| > \delta n^2\right) = 1 - \{ 1- \mathbb{P}(|X_1| > \delta n^2) \}^n \leq 1- \left\{ 1- \frac{\mathbb{E}(|X_1|)}{\delta n^2} \right\}^n \rightarrow 0,
\end{equation*}
as $n \rightarrow \infty$, where the inequality follows from Markov's inequality. Hence, $Z_n - Z_1 = o_p(n^2)$, which then implies that for any subsequence $\{n_k\}$ of $\{n\}$, there exists a further subsequence $\{n_{k_l}\}$ and $\mathcal{S}_0$ such that $\mathbb{P}(\mathcal{S}_0) = 1$ and for any $\omega \in \mathcal{S}_0$,  $\{Z_{n_{k_l}}(\omega)-Z_1(\omega)\}/n^{2}_{k_l} = o(n^2_{k_l})$.
Now, fix a $\omega \in \mathcal{S}_0$.
Because $\int_\mathbb{R} |x| f_0(x)\, dx < \infty$, Lemma 3  and Theorem 4 in \cite{cule2010MLE_log_concave} imply that $\hat{f}_{n,lc}(\cdot; \omega)$ satisfies all the conditions in Lemma \ref{lemma:f_n*lc_UB_LB} for all large enough $n$ (where without loss of generality, we assume that the sets with probability one from Lemma 3 and Theorem 4 in \cite{cule2010MLE_log_concave} contain $\mathcal{S}_0$). Hence, Lemma \ref{lemma:f_n*lc_UB_LB}(c) implies that Condition (i) is satisfied.

In the following, the dependence on $\omega$ is not always written explicitly for the sake of notational simplicity.  We now verify Condition (ii) in Theorem \ref{thm:bootstrap_general}. Write
\begin{equation*}
    R_{n_{k_l}}^* = R_{1{n_{k_l}}}^* +R_{2{n_{k_l}}}^*,
\end{equation*}
where
\begin{align*}
R_{1{n_{k_l}}}^* &:=	- \frac{1}{{n_{k_l}}} \sum_{j=0,\ldots,\frac{{n_{k_l}}-1}{\nu} - 1} 
	\sum^\nu_{l=1} \log \frac{\hat{f}_{{n_{k_l}},lc}(Z_{j\nu+l+1}^*)(Z_{(j+1)\nu+1}^* - Z_{j\nu+1}^*)}{\hat{F}_{{n_{k_l}},lc}(Z_{(j+1)\nu+1}^*) - \hat{F}_{{n_{k_l}},lc}(Z_{j\nu+1}^*)},\\
	R_{2n}^* &:= -\frac{1}{n}\log \frac{\hat{f}_{n,lc}(Z_1)}{f^{H,*}_{n,\nu}(Z_1)}.
\end{align*}
For $R_{1{n_{k_l}}}^*$, by the mean value theorem, there exists $\tilde{Z}^*_j$ lying between $Z_{j\nu+1}^*$ and $Z_{(j+1)\nu+1}^*$ such that $\hat{F}_{{n_{k_l}},lc}(Z_{(j+1)\nu+1}^*) - \hat{F}_{{n_{k_l}},lc}(Z_{j\nu+1}^*) = \hat{f}_{{n_{k_l}},lc}(\tilde{Z}^*_j)(Z_{(j+1)\nu+1}^* - Z_{j\nu+1}^*)$ and so
	\begin{equation*}
R_{1{n_{k_l}}}^*  = -\frac{1}{{n_{k_l}}}\sum_{j=0,\ldots, \frac{{n_{k_l}}-1}{\nu}-1} \sum^\nu_{l=1} \log \frac{\hat{f}_{n_{k_l},lc}(Z_{j\nu+l+1}^*)}{\hat{f}_{n_{k_l},lc}(Z^*_j)}.
	\end{equation*}
Since $\hat{f}_{n_{k_l},lc}$ is log-concave, it is unimodal. Let $s_{n_{k_l}}$ be its mode. Let $j^*$ be such that $s_{n_{k_l}} \in [Z^*_{j^*\nu+1}, Z^*_{(j^*+1)\nu+1}]$. We have
\begin{align*}
    |n_{k_l} R_{1n_{k_l}}^*| & \leq \left|  \sum_{j=0,\ldots, j^*-1} \sum^\nu_{l=1} \log \frac{\hat{f}_{n_{k_l},lc}(Z_{(j+1)\nu+1}^*)}{\hat{f}_{n_{k_l},lc}(Z_{j\nu+1}^*)}\right| + \left|  \sum_{j=j^*+1,\ldots, \frac{n_{k_l}-1}{\nu}} \sum^\nu_{l=1} \log \frac{\hat{f}_{n_{k_l},lc}(Z_{(j+1)\nu+1}^*)}{\hat{f}_{n_{k_l},lc}(Z_{j\nu+1}^*)}\right| \\
    &\quad + \left| \sum^\nu_{l=1} \log \frac{ \hat{f}_{n_{k_l},lc}(s_{n_{k_l}})}{\hat{f}_{n_{k_l},lc}(Z^*_{j^*\nu+1}) }\right| + 
    \left| \sum^\nu_{l=1} \log \frac{\hat{f}_{n_{k_l},lc}(Z^*_{(j^*+1)\nu+1}) }{ \hat{f}_{n_{k_l},lc}(s_{n_{k_l}})} \right|
    \\
    &= \nu | \log \hat{f}_{n_{k_l},lc}(s_{n_{k_l}}) - \log \hat{f}_{n_{k_l},lc}(Z^*_1)| + 
    \nu | \log \hat{f}_{n_{k_l},lc}(s_{n_{k_l}}) - \log \hat{f}_{n_{k_l},lc}(Z^*_n)| \\
    & \leq 2 \nu |\log f_{n_{k_l},lc}(s_{n_{k_l}})| + \nu (| \log \hat{f}_{n_{k_l},lc}(Z^*_1)| + 
   | \log \hat{f}_{n_{k_l},lc}(Z^*_{n_{k_l}})|).
\end{align*}
Note that for all large $n$, we know $\hat{f}_{n_{k_l},lc} \leq C$ for some constant $C$; see \cite{cule2010theoretical}. We now verify the condition in Lemma \ref{lemma:fn_unimodal_logf_bound} (b), 
\begin{equation*}
	\int^\infty_{-\infty} \hat{f}_{n_{k_l},lc}^{1-\alpha}(x) \, dx \leq C^{1-\alpha} (\sigma_{\hat{f}_{n_{k_l}}} - \tau_{\hat{f}_{n_{k_l}}}) = C^{1-\alpha}(Z_{n_{k_l}}(\omega) - Z_1(\omega)) = o(n^2_{k_l}).
\end{equation*}
Thus, by Lemma  \ref{lemma:fn_unimodal_logf_bound} (b), we have
\begin{equation*}
|n_{k_l} R_{1n_{k_l}}^*| = O_{\mathbb{P}^*_\omega}(\log n_{k_l}).
\end{equation*}
The proof of $|n_{k_l} R_{2n_{k_l}}^*| = O_{\mathbb{P}^*_\omega}(\log n_{k_l})$ is similar to that in the proof of Theorem \ref{thm:bootstrap_consistency_k_monotone} and is therefore omitted. Thus, Condition (ii) in Theorem \ref{thm:bootstrap_general} is satisfied.
\end{proof}

{
\section{Choice of $\nu$ based on bootstrap calibration}

We investigate the relationship between $\nu$ and the size distortion under the null, and discuss our recommendation of the upper bound in $\nu$ that is used in the cross-validation procedure.  The size distortion is essentially due to the difference between the distribution of $T_n$ and the bootstrapped statistics $T_n^*$ in small samples.  We observe that $T_n$ and $T_n^*$ are approximately Gaussian, the difference in mean between $T_n$ and $T_n^*$ are typically small compared to their respective variances, and therefore the difference in the distribution of $T_n$ and $T_n^*$ is mainly due to the variance.  Recall $T_n=S_n+M_n+R_n$, where $S_n$ depends on the null distribution class but not $\nu$, $M_n$ is a function of independent standard exponential distributions and is the dominant asymptotically normal term, and $R_n$ is $O_p(\frac{v\log(n)}{n})$.  Likewise, $T_n^*=S_n^*+M_n^*+R_n^*$, where the distribution of $M_n^*$ is the same as that of $M_n$ in finite samples.  For log-concave and $k$-monotone null densities with $k>1$, $R_n$ (and $R_n^*$) are typically non-negligible compared to $S_n$ (and $S_n^*$), and the difference between the variance of $T_n$ and $T_n^*$ increases with $\nu$, contributed mostly by the terms $R_n$ and $R_n^*$.  Therefore, the size distortion happens when $\nu$ is large, which can lead to a too conservative or too liberal type I error depending on the true distribution.  Also, the asymptotic approximation requires $\nu=O(n^{1/3}(\log n)^{-1}$).  Therefore, we propose an upper bound such that the cross validation searches through $\nu=1,\ldots, \lfloor C n^{1/3} (\log n)^{-1} \rfloor$ for some $C$.  $C$ is chosen to be $8$ for all cases except testing against the 1-monotone density which will be explained below.  This choice can limit in the size distortion for many different distributions.  The choice of $C=6$ to $C=10$ were considered and the empirical performance is very similar to $C=8$ under null and alternative distributions.  

The situation is different for monotone densities, in which $S_n$ (and $S_n^*$) converges to $0$ at a slower rate compared to other null cases and we observe a much larger magnitude of $S_n$ (and $S_n^*$) compared to $R_n$ (and $R_n^*$).  Therefore, $\Var(T_n)\approx \Var(S_n+M_n)$ and $\Var(T_n^*)\approx \Var(S_n^*+M_n^*)$ are more similar under a wider range of $\nu$ and the size distortion over $\nu$ is less pronounced.  In this case, our cross validation searches through $\nu=1,\ldots, \lfloor C n^{1/3} (\log n)^{-1} \rfloor$ for $C=30$.  This choice does not exhibit size distortion and enables higher power under a range of alternative distribution compared to a smaller $C$.


}

\bibliographystyle{apalike}
\bibliography{general_abb, shape_bayesian_abb}  
	
\end{document}